\documentclass{siamltex}
\usepackage{amsfonts}
\usepackage{amssymb}
\usepackage{amsfonts}
\usepackage{mathtools}
\usepackage{color}

\newcommand{\R}{{\mathbb R}}

\newcommand{\N}{{\mathbb N}}
\newcommand{\K}{{\mathbb K}}
\newcommand{\C}{{\mathbb C}}
\renewcommand{\H}{{\mathbb H}}
\newcommand{\G}{{\mathbb G}}

\newcommand{\A}{{\mathbb A}}

\newcommand{\bx}{\mathbf{x}}
\newcommand{\bb}{\mathbf{b}}

\newcommand{\br}{\mathbf{r}}

\newcommand{\bv}{\mathbf{v}}
\newcommand{\bw}{\mathbf{w}}
\newcommand{\be}{\mathbf{e}}

\newcommand{\bs}{\mathbf{s}}

\newcommand{\by}{\mathbf{y}}
\newcommand{\bz}{\mathbf{z}}

\newcommand{\bzero}{\mathbf{0}}
\newcommand{\kP}{k_{\rm P}}
\newcommand{\kprec}{k_{\rm P}}

\newcommand{\spn}{\textrm{span}}

\newcommand{\el}{\end{list}}
\newtheorem{algorithm}[theorem]{Algorithm}

\newcommand{\balg}{\begin{algorithm}}
\newcommand{\ealg}{\end{algorithm}}
\newcommand{\bl}{\begin{list}{\ }{
\leftmargin=.325in}}

\title{Arnoldi decomposition, GMRES, and preconditioning \\ for linear discrete ill-posed 
problems}\medskip
\author{
Silvia Gazzola\thanks{Department of Mathematical Sciences, University of Bath, Bath 
BA2 7AY, United Kingdom. E-mail: {\tt s.gazzola@bath.ac.uk}.}
\and
Silvia Noschese\thanks{Dipartimento di Matematica ``Guido Castelnuovo'', SAPIENZA 
Universit\`a di Roma, P.le A. Moro, 2, I-00185 Roma, Italy. E-mail: 
{\tt noschese@mat.uniroma1.it}. Research partially supported by a grant from SAPIENZA 
Universit\`a di Roma and by INdAM-GNCS.}
\and
Paolo Novati\thanks{Dipartimento di Matematica e Geoscienze, Universit\`a di Trieste, via 
Valerio, 12/1, I-34127 Trieste, Italy. E-mail: {\tt novati@units.it}. Research partially 
supported by INdAM-GNCS and FRA-University of Trieste.}
\and
Lothar Reichel\thanks{Department of Mathematical Sciences, Kent State University, Kent,
OH 44242, USA. E-mail: {\tt reichel@math.kent.edu}. Research partially supported by 
NSF grants DMS-1720259 and DMS-1729509.}
}
\begin{document}

\maketitle

\begin{abstract}
GMRES is one of the most popular iterative methods for the solution of large linear 
systems of equations that arise from the discretization of linear well-posed problems,
such as Dirichlet boundary value problems for elliptic partial differential equations. The method
is also applied to iteratively solve linear systems of equations that are obtained by 
discretizing linear ill-posed problems, such as many inverse problems. However, GMRES does not always perform well when 
applied to the latter kind of problems. This paper seeks to shed some light on reasons for
the poor performance of GMRES in certain situations, and discusses some remedies based on 
specific kinds of preconditioning. The standard implementation of GMRES is based on the 
Arnoldi process, which also can be used to define a solution subspace for Tikhonov or TSVD 
regularization, giving rise to the Arnoldi--Tikhonov and Arnoldi-TSVD methods, 
respectively. The performance of the GMRES, the Arnoldi--Tikhonov, and the 
Arnoldi-TSVD methods is discussed. Numerical examples illustrate properties of these
methods.
\end{abstract}


%\AMS
%65F10, 65F22
%\endAMS

\section{Introduction}\label{sec1}
This paper considers the solution of linear systems of equations
\begin{equation}\label{linsys}
A \bx = \bb,\qquad A\in{\C}^{m\times m},\qquad \bx,\bb\in{\C}^m, 
\end{equation}
with a large matrix $A$ with many ``tiny'' singular values of different orders of 
magnitude. In particular, $A$ is severely ill-conditioned and may be rank-deficient. 
Linear systems of equations (\ref{linsys}) with a matrix of this kind are commonly 
referred to as linear discrete ill-posed problems. They arise, for instance, from the 
discretization of linear ill-posed problems, such as Fredholm integral equations of the 
first kind with a smooth kernel. 

In many linear discrete ill-posed problems that arise in science and engineering, the 
right-hand side vector $\bb$ is determined through measurements and is contaminated by a 
measurement error $\be\in\C^m$. Thus,
\begin{equation}\label{btilde}
\bb=\bb_{\rm exact}+\be,
\end{equation}
where $\bb_{\rm exact}\in{\C}^m$ denotes the unknown error-free right-hand side associated 
with $\bb$. We will assume that $\bb_{\rm exact}$ is in the range of $A$, denoted by
${\mathcal R}(A)$, because this facilitates the use of the discrepancy principle to 
determine a suitable value of a regularization parameter; see below for details. The 
error-contaminated right-hand side $\bb$ is not required to be in ${\mathcal R}(A)$. 

We would like to compute the solution of minimal Euclidean norm, $\bx_{\rm exact}$, of the 
consistent linear discrete ill-posed problem 
\begin{equation}\label{nflinsys}
A\bx=\bb_{\rm exact}.
\end{equation}
Since the right-hand side $\bb_{\rm exact}$ is not known, we seek to determine an 
approximation of $\bx_{\rm exact}$ by computing an approximate solution of the available 
linear system of equations (\ref{linsys}). We note that due to the severe ill-conditioning 
of the matrix $A$ and the error $\be$ in $\bb$, the least-squares solution of minimal 
Euclidean norm of (\ref{linsys}) generally is not a useful approximation of 
$\bx_{\rm exact}$. 

A popular approach to determine a meaningful approximation of $\bx_{\rm exact}$ is to 
apply an iterative method to the solution of (\ref{linsys}) and terminate the iterations 
early enough so that the error in $\bb$ is not significantly propagated into the computed 
approximate solution. The most popular iterative methods for the solution of large linear 
discrete ill-posed problems are LSQR by Paige and Saunders 
\cite{EHN,Hankebook,Hansenbook,PS}, which is based on partial Golub--Kahan decomposition 
of $A$, and GMRES \cite{CLR1,CLR2,GNR2}, which is based on partial Arnoldi 
decomposition of $A$. Here ``GMRES'' refers to both the standard GMRES method proposed by
Saad and Schultz \cite{SS} as well as to variants that have been found to perform better 
when applied to the solution of linear discrete ill-posed problems; see, e.g., 
\cite{DR,JH,NRS} for examples. 

The LSQR method requires the evaluation of two matrix-vector products in each iteration,
one with $A$ and one with its conjugate transpose, which we denote by $A^*$. GMRES only 
demands the computation of one matrix-vector product with $A$ per iteration. This 
makes GMRES attractive to use when it is easy to evaluate matrix-vector products with $A$ 
but not with $A^*$. This is, for instance, the case when $A$ approximates a Fredholm 
integral operator of the first kind and matrix-vector products with $A$ are evaluated by a
multipole method. Then $A$ is not explicitly formed and matrix-vector products with $A^*$ 
are difficult to compute; see, e.g., \cite{GR} for a discussion on the multipole method. 
It may be difficult to evaluate matrix-vector products with $A^*$ also when solving 
nonlinear problems and $A$ represents a Jacobian matrix, whose entries are not explicitly 
computed; see \cite{CJ2} for a discussion on such a solution method. 

The fact that GMRES does not require the evaluation of matrix-vector products with $A^*$
leads to that for many linear discrete ill-posed problems (\ref{linsys}), this method
requires fewer matrix-vector product evaluations than LSQR to determine a desired 
approximate solution, see, e.g., \cite{CLRX,CLR4,CLR2} for illustrations, as well as 
\cite{CLR5} for related examples. However, there also are linear discrete ill-posed 
problems (\ref{linsys}), whose solution with LSQR requires fewer matrix-vector product 
evaluations than GMRES, or for which LSQR furnishes a more accurate approximation of 
$\bx_{\rm exact}$ than GMRES; see below for illustrations, as well as \cite{HJ}. Reasons 
for poor performance of GMRES include:
\vskip2mm
\begin{enumerate}
\item
The low-dimensional solution subspaces used by GMRES are poorly suited to represent 
$\bx_{\rm exact}$. It is often not possible to rectify this problem by carrying out many 
iterations, 
%and this way increase the dimension of the solution subspace, 
since this typically results in severe propagation of the error $\be$ in $\bb$ into the iterates 
determined by GMRES.
\item
The desired solution $\bx_{\rm exact}$ may be approximated accurately in solution 
subspaces generated by GMRES, but the method determines iterates that furnish poor 
approximations of $\bx_{\rm exact}$.
\item
The GMRES iterates suffer from contamination of propagated error due to the fact that the
initial vector in the Arnoldi decomposition used for the solution of (\ref{linsys}) is a 
normalization of the error-contaminated vector $\bb$. 
\end{enumerate}
\vskip2mm

It is the purpose of the present paper to discuss the above mentioned shortcomings of 
GMRES, illustrate situations when they occur, and provide some remedies. Section 
\ref{sec2} defines the Arnoldi process and GMRES, and shows that the solution subspaces 
used by GMRES may be inappropriate. Also LSQR is briefly discussed, and distances to relevant 
classes of matrices are introduced. In Section \ref{sec3}, we define the set of generalized 
Hermitian matrices and the set of generalized Hermitian positive semidefinite matrices. 
The distance of the matrix $A$ in (\ref{linsys}) to these sets sheds light on how quickly 
GMRES applied to the solution of the linear system of equations (\ref{linsys}) will 
converge. Section \ref{sec4} describes ``preconditioning techniques.'' The 
``preconditioners'' discussed do not necessarily reduce the condition number, and they are
not guaranteed to reduce the number of iterations. Instead, they are designed to make the 
matrix of the preconditioned linear system of equations closer the set of generalized 
Hermitian positive semidefinite matrices. This often results in that the computed solution is 
a more accurate approximation of $\bx_{\rm exact}$ than approximate solutions of the 
unpreconditioned linear system (\ref{linsys}).  In Section \ref{sec5} we consider the 
situation when GMRES applied to the solution of (\ref{linsys}) yields poor approximations 
of $\bx_{\rm exact}$, but the solution subspace generated by the Arnoldi process contains 
an accurate approximation of $\bx_{\rm exact}$. We propose to carry out sufficiently many 
steps of the Arnoldi process and compute an approximation of $\bx_{\rm exact}$ in the 
solution subspace generated by Tikhonov regularization or truncated singular value 
decomposition. Both regularization methods allow the use of a solution subspace of larger 
dimension than GMRES.
%Error propagation is addressed in Section \ref{sec6}, where variants 
%of GMRES that have been proposed to reduce the error propagation of standard GMRES are 
%reviewed. 
A few computed examples that illustrate the discussion of the previous sections 
are presented in Section \ref{sec7}, and Section \ref{sec8} contains concluding remarks.

\section{GMRES for linear discrete ill-posed problems}\label{sec2}
GMRES is a popular iterative method for the solution of large linear systems of equations
with a square nonsymmetric matrix (\ref{linsys}) that arise from the discretization of 
well-posed problems; see, e.g., Saad \cite{Saadbook}. The $k$th iterate, $\bx_k$, 
determined by GMRES, when applied to the solution of (\ref{linsys}) with initial iterate 
$\bx_0=\bzero$, satisfies
\begin{equation}\label{gmres1}
\|A\bx_k-\bb\|= \min_{\bx\in{\K}_k(A,\bb)}\|A\bx-\bb\|,\qquad \bx_k\in{\K}_k(A,\bb),
\end{equation}
where
\[
{\K}_k(A,\bb)=\mbox{span}\{\bb,A\bb,\ldots,A^{k-1}\bb\}
\]
is a Krylov subspace and $\|\cdot\|$ denotes the Euclidean vector norm. We tacitly assume 
that $k$ is sufficiently small so that $\mbox{dim}({\K}_k(A,\bb))=k$, which in turn guarantees that the iterate $\bx_k$ is uniquely defined. We will 
throughout this section assume that $1\leq k\ll m$. The standard implementation of GMRES 
\cite{Saadbook,SS} is based on the Arnoldi process, here given with the modified 
Gram--Schmidt implementation.

\vskip5pt
\begin{algorithm}\label{alg:arnoldi}
{\rm\small
{\sc The Arnoldi process}
\begin{tabbing}
mmmmmmm\=mm\=mm\=mm\=mm\=mm\=   \kill
\> 0.\> \bf Input $A\in\C^{m\times m}$, ${\bb}\in\C^m\backslash\{\bzero\}$\\
\> 1.\> $\bv_1:=\bb/\|\bb\|$;\\
\> 2.\> \bf for $j=1,2,\ldots, k$ do \\
\> 3.\> \> $\bw:=A{\bv}_j$; \\
\> 4.\> \> \bf for $i=1,2,\ldots j$ do \\
\> 5.\> \> \>  $h_{i,j}:={\bv}_i^* \bw$; $\bw:=\bw-{\bv}_ih_{i,j}$; \\
\> 6.\> \> \bf end for \\
\> 7.\> \> $h_{j+1,j}:=\|\bw\|$; $\bv_{j+1}:=\bw/h_{j+1,j}$;\\
\> 8.\> \bf end for
\end{tabbing}
}
\end{algorithm}
\vskip5pt

Algorithm \ref{alg:arnoldi} generates orthonormal vectors $\bv_1,\bv_2,\ldots,\bv_{k+1}$, 
the first $k$ of which form a basis for ${\K}_k(A,\bb)$. Define the matrices 
$V_k=[\bv_1,\bv_2,\ldots,\bv_k]$ and \linebreak[4]$V_{k+1}=[V_k,\bv_{k+1}]$. The scalars 
$h_{i,j}$ determined by the algorithm define an upper Hessenberg matrix 
$H_{k+1,k}\in{\C}^{(k+1)\times k}$, i.e., the $h_{i,j}$ are the nontrivial entries of 
$H_{k+1,k}$. Using these matrices, the recursion formulas for the Arnoldi process can be 
expressed as a \emph{partial Arnoldi decomposition},
\begin{equation}\label{arndec}
A V_k = V_{k+1} H_{k+1,k}.
\end{equation}
The above relation is applied to compute the GMRES iterate $\bx_k$ as follows: Express (\ref{gmres1}) as
\begin{equation}\label{gmres2}
\min_{\bx\in{\K}_k(A,\bb)}\|A\bx-\bb\|=
\min_{\by\in{\C}^k}\|AV_k\by-\bb\|=
\min_{\by\in{\C}^k}\|H_{k+1,k}\by-\be_1\|\bb\|\,\|,
\end{equation}
where the orthonormality of the columns of $V_k$ and the fact that $\bb=V_{k+1}\be_1\|\bb\|$ 
have been exploited. Throughout this paper $\be_j=[0,\ldots,0,1,0,\ldots,0]^*$ denotes the
$j$th axis vector. The small minimization problem on the right-hand side of (\ref{gmres2}) can 
be solved conveniently by QR factorization of $H_{k+1,k}$; see \cite{Saadbook}. Denote the
solution by $\by_k$. Then $\bx_k=V_k\by_k$ solves (\ref{gmres1}) and $\br_k=\bb-A\bx_k$ 
is the associated residual error. Since ${\K}_{k-1}(A,\bb)\subset{\K}_k(A,\bb)$, we have 
$\|\br_{k}\|\leq\|\br_{k-1}\|$; generally this inequality is strict. Note that 
$\|\br_k\|=\|\be_1\|\bb\|-H_{k+1,k}\by_k\|$, so that the norm of the residual 
vector can be monitored using projected quantities, which are inexpensive to compute. We remark that a reorthogonalization procedure can be considered with Algorithm \ref{alg:arnoldi}, by running an additional modified Gram--Schmidt step for the vector $\bw$ after step 6 has been performed: this has the effect of assuring the columns of $V_{k+1}$ a better numerical orthogonality.

Assume that a fairly accurate bound $\delta>0$ for the norm of the error $\be$ in $\bb$ is
available,
\begin{equation}\label{errbd}
\|\be\|\leq\delta,
\end{equation}
and let $\tau\geq 1$ be a user-chosen parameter that is independent of $\delta$. The 
\emph{discrepancy principle} prescribes the iterations of GMRES applied to the 
solution of (\ref{linsys}) to be terminated as soon as an iterate $\bx_k$ has been 
determined such that the associated residual error $\br_k$ satisfies
\begin{equation}\label{discr}
\|\br_k\|\leq\tau\delta.
\end{equation}
The purpose of this stopping criterion is to terminate the iterations before the iterates
$\bx_k$ are severely contaminated by propagated error that stems from the error $\be$ in 
$\bb$. Note that the residual $\br_{\rm exact}=\bb-A\bx_{\rm exact}$ satisfies the 
inequality (\ref{discr}). This follows from the consistency of (\ref{nflinsys}), and (\ref{errbd}). Also 
iterations with LSQR are commonly terminated with the discrepancy principle; see, e.g.,
\cite{CLR1,EHN,Hankebook} for discussions on the use of the discrepancy principle for 
terminating iterations with GMRES and LSQR. 

The LSQR method \cite{PS} is an implementation of the conjugate gradient 
method applied to the normal equations,
\begin{equation}\label{normeq}
A^*A\bx=A^*\bb,
\end{equation}
with a Hermitian positive semidefinite matrix. LSQR circumvents the explicit formation of
$A^*A$. When using the initial iterate $\bx_0=\bzero$, LSQR determines approximate 
solutions of (\ref{linsys}) in a sequence of nested Krylov subspaces ${\K}_k(A^*A,A^*\bb)$, 
$k=1,2,\ldots~$. The $k$th iterate, $\bx_k$, computed by LSQR satisfies
\[
\|A\bx_k-\bb\|= \min_{\bx\in{\K}_k(A^*A,A^*\bb)}\|A\bx-\bb\|,\qquad 
\bx_k\in{\K}_k(A^*A,A^*\bb);
\]
see \cite{Bj,PS} for further details on LSQR. 

When $\delta$ in (\ref{errbd}) is fairly large, only few iterations can be carried out by
GMRES or LSQR before (\ref{discr}) is satisfied. In particular, an accurate approximation
of $\bx_{\rm exact}$ can then be determined by GMRES only if $\bx_{\rm exact}$ can be 
approximated well in a low-dimensional Krylov subspace ${\K}_k(A,\bb)$.  Moreover, it has 
been observed that GMRES based on the Arnoldi process applied to $A$ with initial 
vector $\bb$ may determine iterates $\bx_k$ that are contaminated by more propagated error
than iterates generated by LSQR; see \cite{HJ}. A reason for this is that a 
normalization of the error-contaminated vector $\bb$ is the first column of the matrix 
$V_k$ in the Arnoldi decomposition (\ref{arndec}), and the error $\be$ in $\bb$ is 
propagated to all columns of $V_k$ by the Arnoldi process. A remedy for this difficulty is
to use a modification of the Arnoldi decomposition,
\begin{equation}\label{modarndec}
A\widehat{V}_k=V_{k+j}H_{k+j,k},
\end{equation}
with $j\geq 2$. The columns of $\widehat{V}_k\in\C^{m\times k}$ form an orthonormal basis 
for the Krylov subspace $\K_k(A,A^{(j-1)}\bb)$, in which we are looking for an 
approximate solution. Moreover, the columns of 
$V_{k+j}\in\C^{m\times(k+j)}$ form an orthonormal basis for $\K_{k+j}(A,\bb)$, and all 
entries of the matrix $H_{k+j,k}\in\C^{(k+j)\times k}$ below the $j$th subdiagonal 
vanish; see \cite{DR} for details. The special case when $j=2$ is discussed in \cite{NRS}. 
When $j=1$, the decomposition \eqref{modarndec} simplifies to \eqref{arndec}.

A reason why applying the decomposition (\ref{modarndec}) may be beneficial is that, in
our typical applications, the matrix $A$ is a low-pass filter. Therefore, the high-frequency error
in the vector $\widehat{V}_k\be_1=A^{(j-1)}\bb/\|A^{(j-1)}\bb\|$ is damped.
%We may replace the decompositions (\ref{arndec}) and (\ref{arndec2}) by modified decompositions
%that are analogous to (\ref{modarndec}) in the solution and preconditioning methods of 
%Sections \ref{sec2} and \ref{sec3}. In our experience, the use of modified Arnoldi
%decompositions generally is not very beneficial when the preconditioners (\ref{prec2}) 
%or (\ref{prec1}).

The following examples illustrate that GMRES may perform poorly also when there is no 
error in $\bb$, in the sense that GMRES may require many iterations to solve the system of 
equations or not be able to compute a solution at all. While the coefficient matrices of these examples 
are artificial, related ones (e.g., the test problem \texttt{heat}) can be found in Hansen's \emph{Regularization Tools} \cite{Ha}, and also arise in image restoration when the available image has been 
contaminated by motion blur; see \cite[Section 4]{DMR}.

\emph{Example 2.1.} Let $A$ in (\ref{linsys}) be the downshift matrix
\begin{equation}\label{shiftmat}
A=\left[\begin{array}{cccccc}
0 & 0 & 0 & \cdots & 0 & 0 \\
1 & 0 & 0 & \cdots  & 0 & 0 \\
0 & 1 & 0 & \cdots  & 0 & 0 \\
  & 0 & \ddots  & \vdots  & 0 & 0 \\
  &   & \ddots  &         & 0 & 0 \\
  &   &   &  0  & 1 & 0 
\end{array}\right]\in{\C}^{m\times m}
\end{equation}
and let $\bb=\be_2$. The minimal-norm solution of the linear system of equations 
(\ref{linsys}) is $\bx_{\rm exact}=\be_1$. Since 
${\K}_k(A,\bb)={\rm span}\{\be_2,\be_3,\ldots,\be_{k+1}\}$, it follows that the solution 
of (\ref{gmres1}) is $\bx_k=\bzero$ for $1\leq k<m$. These solutions are poor 
approximations of $\bx_{\rm exact}$. GMRES breaks down at step $m$ due to division by zero 
in Algorithm \ref{alg:arnoldi}. Thus, when $m$ is large GMRES produces poor approximations
of $\bx_{\rm exact}$ for many iterations before breakdown. While breakdown of GMRES can be 
handled, see \cite{RY}, the lack of convergence of the iterates towards $\bx_{\rm exact}$ 
for many steps remains. The poor performance of GMRES in this example stems 
from the facts that $A$ is a shift operator and the desired solution $\bx_{\rm exact}$ has 
few nonvanishing entries. We remark that the minimal-norm solution 
$\bx_{\rm exact}=\be_1$ of (\ref{linsys}) lives in ${\K}_1(A^*A,A^*\bb)$ and LSQR 
determines this solution in one step. 
~~~$\Box$

%Example 2.2. Let $A$ and $\bb$ be defined as in Example 2.1. 

Example 2.1 illustrates that replacing a linear discrete ill-posed problem \eqref{linsys} 
with a non-Hermitian coefficient matrix $A$ by a linear discrete ill-posed problem \eqref{normeq} having
a Hermitian positive semidefinite matrix $A^*A$ may be beneficial. To shed some light on 
the possible benefit of this kind of replacement, with the aim of developing suitable 
preconditioners different from $A^*$, we will discuss the distance of a square matrix $A$
to the set of Hermitian matrices $\H$, the set of anti-Hermitian (skew-Hermitian) matrices 
$\A$, the set of normal matrices $\N$, the set of Hermitian positive semidefinite matrices 
$\H_+$, and the set of Hermitian negative semidefinite matrices $\H_-$. We are interested 
in the distance to the set of normal matrices, because it is known that GMRES may converge 
slowly when the matrix $A$ in (\ref{linsys}) is far from $\N$. Specifically, the rate of 
convergence of GMRES may be slow when $A$ has a spectral factorization with a very 
ill-conditioned eigenvector matrix; see \cite[Theorem 3]{NRT1} and \cite{NRT2} for 
discussions. We remark that when $A$ belongs to the classes $\N$, $\H$, $\A$, or $\H_+$, 
the Arnoldi process and GMRES can be simplified; see, e.g., Eisenstat \cite{Ei}, Huckle 
\cite{Hu}, Paige and Saunders \cite{PS0}, and Saad \cite[Section 6.8]{Saadbook}. 

We measure distances between a matrix $A$ and the sets $\H$, $\A$, $\N$, and $\H_\pm$ in 
the Frobenius norm, which for a matrix $M$ is defined as 
$\|M\|_F=({\rm trace}(M^*M))^{1/2}$. The following proposition considers the matrix of 
Example 2.1.

\begin{proposition}\label{prop1}
Let the matrix $A\in\C^{m\times m}$ be defined by (\ref{shiftmat}). The relative distances
in the Frobenius norm to the sets of the Hermitian and anti-Hermitian matrices are
\begin{equation}\label{distsym}
\frac{{\rm dist}_F(A,\H)}{\|A\|_F}=\frac{1}{\sqrt{2}}
\end{equation}
and
\begin{equation}\label{distantisym}
\frac{{\rm dist}_F(A,\A)}{\|A\|_F}=\frac{1}{\sqrt{2}},
\end{equation}
respectively. Moreover,
\begin{equation}\label{distnormal}
\frac{{\rm dist}_F(A,\N)}{\|A\|_F}\leq \frac{1}{\sqrt{m}},
\end{equation}
\begin{equation}
\label{distspsdspec}
\frac{{\rm dist}_F(A,{\H_+})}{\|A\|_F}=\frac{\sqrt{3}}{2}
\end{equation}
and
\begin{equation}\label{distsnsdspec}
\frac{{\rm dist}_F(A,{\H_-})}{\|A\|_F}=\frac{\sqrt{3}}{2}.
\end{equation}
\end{proposition}

\begin{proof}
The distance (\ref{distsym}) is shown in \cite[Section 5]{NPR1}, and (\ref{distantisym})
can be shown similarly. Thus, the matrix $A$ is equidistant to the sets $\H$ and $\A$.
The upper bound (\ref{distnormal}) for the distance to the set of normal matrices is 
achieved for a circulant matrix; see \cite[Section 9]{NPR1}. The distance to $\H_+$ is
given by 
\begin{equation}\label{distspsd}
{\rm dist}_F(A,\H_+)=\left(\sum_{\lambda_i(A_{\H})<0}\lambda_i^2(A_{\H})+
\|A_{\A}\|_F^2\right)^{1/2},
\end{equation}
see Higham \cite[Theorem 2.1]{Hi}. Here $A_{\H}=(A+A^*)/2$ and $A_{\A}=(A-A^*)/2$ denote 
the Hermitian and skew-Hermitian parts of $A$, respectively, and
$\lambda_1(A_{\H}),\ldots,\lambda_m(A_{\H})$ are the eigenvalues of $A_{\H}$. We note 
that the distance in the Frobenius norm to the set $\H_+$ is the same as the distance
to the set of Hermitian positive definite matrices. The eigenvalues of $A_{\H}$ are known 
to be 
\begin{equation}\label{evalsym}
\lambda_j(A_{\H})=\cos\frac{\pi j}{m+1},\qquad j=1,2,\ldots,m;
\end{equation}
see, e.g., \cite[Section 2]{NPR2}. The expression (\ref{distspsdspec}) now follows from
$\|A\|_F^2=m-1$, $\|A_{\A}\|^2=(m-1)/2$, and the fact that the sum in (\ref{distspsd}) 
evaluates to $(m-1)/4$. Finally, (\ref{distsnsdspec}) follows from 
\[
{\rm dist}_F(A,\H_-)=\left(\sum_{\lambda_i(A_{\H})>0}\lambda_i^2(A_{\H})+
\|A_{\A}\|_F^2\right)^{1/2}
\]
and the fact that the eigenvalues (\ref{evalsym}) are allocated symmetrically with 
respect to the origin.
\end{proof}

Proposition \ref{prop1} shows the matrix (\ref{shiftmat}) to be close to a normal matrix
and Example 2.1 illustrates that closeness to normality is not sufficient for GMRES to
give an accurate approximation of the solution within a few iterations. Indeed, we can 
modify the matrix (\ref{shiftmat}) to obtain a normal matrix and, as the following
example shows, GMRES requires many iterations to solve the resulting linear system of
equations.

\emph{Example 2.2.} Let the matrix $A$ be a circulant obtained by setting the $(1,m)$-entry of 
the matrix (\ref{shiftmat}) to one, and let the right-hand side $\bb$ be the same as in
Example 2.1. Then the solution is $\bx_{\rm exact}=\be_1$. Similarly as in Example 2.1, 
GMRES yields the iterates $\bx_k=\bzero$ for $1\leq k<m$. The solution is not achieved 
until the iterate $\bx_m$ is computed. A related example is presented by Nachtigal et al.
\cite{NRT1}. We remark that the matrix $A^*A$ is the identity, so the first iterate determined 
by LSQR with initial iterate $\bx_0=\bzero$ is $\bx_{\rm exact}$. Thus, LSQR performs much 
better than GMRES also for this example.~~~$\Box$

%Example 2.4. Let $A$ and $\bb$ be the matrix and right-hand side of Example 2.2. 

The dependence of the convergence behavior of GMRES on the eigenvalues and eigenvectors of 
$A$ is complicated; see Du et al. \cite{DDTM} for a recent discussion and references. It 
is therefore not clear whether replacing the matrix $A$ in (\ref{linsys}) by a matrix that
is closer to the sets $\H$, $\H_+$, or $\H_-$ by choosing a suitable preconditioner, and 
then applying GMRES to the preconditioned linear system of equations so obtained, will 
yield an improved approximation of $\bx_{\rm exact}$. Moreover, we do not want the 
preconditioner to give severe propagation of the error $\be$ in $\bb$ into the computed 
iterates; see Hanke et al. \cite{HNP} for an insightful discussion on the construction of 
preconditioners for linear discrete ill-posed problems. Nonetheless, when $A$ is close to 
the set $\H_+$ and the eigenvalues of $A$ cluster in a small region in the complex plane,
convergence typically is fairly rapid. This suggests that we should determine a 
preconditioner such that the preconditioned matrix is close to the set $\H_+$. Since the 
convergence of GMRES applied to the solution of (\ref{linsys}) is invariant under 
multiplication of the matrix $A$ by a complex rotation $\mathrm{e}^{\mathrm{i}\varphi}$, 
where $\mathrm{i}=\sqrt{-1}$ and $-\pi<\varphi\leq\pi$, it suffices that the 
preconditioned matrix is close to a ``rotated'' Hermitian positive semidefinite matrix 
$\mathrm{e}^{\mathrm{i}\varphi}N$, where $N\in\H_+$. In the following we refer to the set
of ``rotated'' Hermitian matrices as the set of \emph{generalized Hermitian matrices},
and denote it by $\G$. It contains normal matrices $A\in\C^{m\times m}$, whose 
eigenvalues are collinear; see below. The set of ``rotated'' Hermitian positive 
semidefinite matrices is denoted by $\G_+$ and referred to as the set of \emph{generalized 
Hermitian positive semidefinite matrices}.

\section{Generalized Hermitian and Hermitian positive semidefinite matrices}\label{sec3}
In the following, we show some properties of generalized Hermitian and generalized 
Hermitian positive semidefinite matrices.

\begin{proposition}\label{S1}
The matrix $A\in\C^{m\times m}$ is generalized Hermitian if and only if there exist 
$\varphi\in (-\pi,\pi]$ and $\alpha\in \C$ such that
\begin{equation}\label{Bherm}
A=\mathrm{e}^{\mathrm{i}\varphi}B+\alpha\,I,
\end{equation}
where $B\in\C^{m\times m}$ is an Hermitian matrix and $I\in\C^{m\times m}$ denotes the 
identity.
\end{proposition}

\begin{proof}
Let $A\in\C^{m\times m}$ be a generalized Hermitian matrix. Then there is a unitary matrix 
$U\in\C^{m\times m}$ such that $A=U\Lambda U^*$, where 
$\Lambda=\mathrm{diag}[\lambda_1,\ldots,\lambda_m]\in\C^{m\times m}$, and $A$ has collinear 
eigenvalues, i.e., there exist $\varphi\in (-\pi,\pi]$ and $\alpha \in \C$ such that 
\[
\Lambda=\mathrm{e}^{\mathrm{i}\varphi}D+\alpha I,
\]
where $D =\mathrm{diag}[d_1,\ldots,d_m]\in\R^{m\times m}$, so that 
$\lambda_i=\mathrm{e}^{\mathrm{i}\varphi}d_i+\alpha$ for $1\leq i\leq m$.
Thus, the matrix
\[
B:=\mathrm{e}^{-\mathrm{i}\varphi}(A-\alpha I)=\mathrm{e}^{-\mathrm{i}\varphi}(U\Lambda U^*-\alpha I)
=U(\mathrm{e}^{-\mathrm{i}\varphi}(\Lambda-\alpha I))U^*=UDU^*
\]
is Hermitian.

Conversely, if $B=\mathrm{e}^{-\mathrm{i}\varphi}(A-\alpha I)$ is Hermitian, then 
$B=UDU^*$, where $U\in\C^{m\times m}$ is unitary and $D\in\R^{m\times m}$ is diagonal.
Hence, $A=\mathrm{e}^{\mathrm{i}\varphi}B+\alpha\,I=
U(\mathrm{e}^{\mathrm{i}\varphi}D+\alpha I)U^*$ has collinear eigenvalues and is unitarily
diagonalizable.
\end{proof}

\begin{proposition}\label{S2}
If the matrix $Z=[z_{i,j}]\in\C^{m\times m}$ is generalized Hermitian, then there exist 
$\theta\in (-\pi,\pi]$ and $\gamma \in \R$ such that
\begin{equation}\label{gherm}
z_{i,j}= \left\{
\begin{array}{lr}
\overline{z}_{j,i}\,\mathrm{e}^{\mathrm{i}\theta},& \mathrm{if}\; i\neq j,\\
\overline{z}_{i,i}\,\mathrm{e}^{\mathrm{i}\theta}+
\gamma\,\mathrm{e}^{\mathrm{i}\frac{\theta+\pi}{2}},& \mathrm {if}\; i=j,
\end{array}
\right.
%Z=\mathrm{e}^{\mathrm{i}\theta}Z^*+\gamma\mathrm{e}^{\mathrm{i}\frac{\theta+\pi}{2}}\,I
\end{equation}
where the bar denotes complex conjugation.
\end{proposition}

\begin{proof}
It follow from Proposition \ref{S1} that there exist an angle $\phi\in (-\pi,\pi]$ and a 
scalar $\beta\in \C$ such that $\mathrm{e}^{\mathrm{i}\phi}Z+\beta\,I$ is Hermitian, i.e.,  
\[
\mathrm{e}^{\mathrm{i}\phi}Z+\beta I=\mathrm{e}^{-\mathrm{i}\phi}Z^*+\overline{\beta}I.
\]
Thus, 
\[
Z=\mathrm{e}^{-2\mathrm{i}\phi}Z^*-2\,\mathrm{e}^{\mathrm{i}\frac{\pi-2\phi}{2}}\mathrm{Im}(\beta)I.
\]
Setting $\theta=-2\phi$ and $\gamma=-2\,\mathrm{Im}(\beta)$ concludes the proof.
\end{proof}

\begin{proposition}\label{S3}
Let $A=[a_{i,j}]\in\C^{m\times m}$. If 
\begin{equation}\label{cond}
m\,\mathrm{Trace}(A^2)\neq\mathrm{Trace}(A),
\end{equation}
then the unique closest generalized Hermitian matrix 
$\widehat{A}=[\widehat{a}_{i,j}]\in\C^{m\times m}$ to $A$ in the Frobenius norm is given 
by
\begin{equation}\label{clgherm}
\widehat{a}_{i,j}= \left\{
\begin{array}{lr}
\frac{1}{2}(a_{i,j}+\overline{a}_{j,i}\mathrm{e}^{\mathrm{i}\widehat{\theta}}),&
\mathrm{if}\; i\neq j,\\
\frac{1}{2}(a_{i,i}+\overline{a}_{i,i}\mathrm{e}^{\mathrm{i}\widehat{\theta}}+
\widehat{\gamma}\,\mathrm{e}^{\mathrm{i}\frac{\widehat{\theta}+\pi}{2}}),& 
\mathrm {if}\; i=j,
\end{array}
\right.
\end{equation}
where 
\[
\widehat{\theta}=\arg(\mathrm{Trace}(A^2)-\frac{1}{m}(\mathrm{Trace}(A))^2),\qquad
\widehat{\gamma}=\frac{2}{m}\,\mathrm{Im}(\mathrm{e}^{-\mathrm{i}
\frac{\widehat{\theta}}{2}}\mathrm{Trace}(A)).
\]
Moreover, the distance of $A$ to the set ${\mathbb G}$ of generalized Hermitian matrices 
is given by
\[
{\rm dist}_F(A,{\mathbb G})=\sqrt{\frac{\|A\|_F^2}{2}-\frac{1}{2}\mathrm{Re}
\left(\mathrm{e}^{-\mathrm{i}\widehat{\theta}}\sum_{i,j=1}^m a_{i,j}a_{j,i}\right)-
\frac{1}{m}\left(\mathrm{Im}\left(\mathrm{e}^{-\mathrm{i}\frac{\widehat{\theta}}{2}}
\sum_{i=1}^m a_{i,i}\right)\right)^2}.
\]
If (\ref{cond}) is violated, then there are infinitely many matrices 
$\widehat{A}(\theta)=[\widehat{a}_{i,j}(\theta)]\in\C^{m\times m}$, depending on an 
arbitrary angle $\theta$, at the same minimal distance from $A$, whose entries are given
by
\begin{equation}\label{clgherm2}
\widehat{a}_{i,j}(\theta)= \left\{
\begin{array}{lr}
\frac{1}{2}(a_{i,j}+\overline{a}_{j,i}\mathrm{e}^{\mathrm{i}\theta}),& \mathrm{if}\;
i\neq j,\\
\frac{1}{2}(a_{i,i}+\overline{a}_{i,i}\mathrm{e}^{\mathrm{i}\theta}+\widehat{\gamma}\,
\mathrm{e}^{\mathrm{i}\frac{\theta+\pi}{2}}),& \mathrm {if}\; i=j.
\end{array}
\right.
\end{equation}
\end{proposition}

\begin{proof}
The entries of the generalized Hermitian matrix 
$Z(\theta,\gamma)=[z_{i,j}(\theta,\gamma)]\in\C^{m\times m}$, that minimizes the distance 
of $A$ in the Frobenius norm from the set ${\mathbb G}$ for the 
given angle $\theta$ and real $\gamma$, are determined by minimizing $\|A-Z\|^2$, where 
the matrix $Z\in\C^{m\times m}$ is subject to the equality constraints of Proposition 
\ref{S2}. The method of Lagrange multipliers 
produces
\begin{equation}\label{cgherm}
z_{i,j}(\theta,\gamma)= \left\{
\begin{array}{lr}
\frac{1}{2}(a_{i,j}+\overline{a}_{j,i}\mathrm{e}^{\mathrm{i}\theta}),& 
\mathrm{if}\; i\neq j,\\
\frac{1}{2}(a_{i,i}+\overline{a}_{i,i}\mathrm{e}^{\mathrm{i}\theta}+\gamma\,
\mathrm{e}^{\mathrm{i}\frac{\theta+\pi}{2}}),& \mathrm {if}\; i=j.
\end{array}
\right.
%Z=\mathrm{e}^{\mathrm{i}\theta}Z^*+\gamma\mathrm{e}^{\mathrm{i}\frac{\theta+\pi}{2}}\,I
\end{equation}
Substituting these values into $\|A-Z\|_F$ yields
\begin{eqnarray*}
d(\theta,\gamma)&=&\|A-Z(\theta,\gamma)\|_F^2=\frac{1}{4}
\sum_{\substack{i,j=1\\i\neq j}}^m |a_{i,j}-\overline{a}_{j,i}
\mathrm{e}^{\mathrm{i}\theta}|^2+\frac{1}{4}\sum_{i=1}^m 
|a_{i,i}-(\overline{a}_{i,i}\mathrm{e}^{\mathrm{i}\theta}+\gamma\,
\mathrm{e}^{\mathrm{i}\frac{\theta+\pi}{2}})|^2\\
&=& \frac{\|A\|_F^2}{2}+\frac{m}{4}\gamma^2-\frac{1}{2}\mathrm{Re}
\left(\mathrm{e}^{-\mathrm{i}\theta}\sum_{i,j=1}^m a_{i,j}a_{j,i}\right)
-\gamma\mathrm{Im}\left(\mathrm{e}^{-\mathrm{i}\frac{\theta}{2}}
\sum_{i=1}^m a_{i,i}\right).
\end{eqnarray*}
The desired values of $\theta$ and $\gamma$ are determined by minimizing 
$d(\theta,\gamma)$. It follows that $\partial d(\theta,\gamma)/ \partial \gamma=0$ if and
only if 
\[
\gamma=\widehat{\gamma}(\theta)=\frac{2}{m}\,
\mathrm{Im}\left(\mathrm{e}^{-\mathrm{i}\frac{\theta}{2}}\sum_{i=1}^{m}a_{i,i}\right).
\]
Thus, we obtain 
\[
d(\theta,\widehat{\gamma}(\theta))=\frac{\|A\|_F^2}{2}-\frac{1}{2}
\mathrm{Re}\left(\mathrm{e}^{-\mathrm{i}\theta}\sum_{i,j=1}^m a_{i,j}a_{j,i}\right)-
\frac{1}{m}\left(\mathrm{Im}\left(\mathrm{e}^{-\mathrm{i}\frac{\theta}{2}}
\sum_{i=1}^m a_{i,i}\right)\right)^2.
\]
Then $d'(\theta,\widehat{\gamma}(\theta))=0$ if and only if 
\[
\left(\mathrm{Re}(w_1)-\frac{1}{m}\mathrm{Re}(w_2^2)\right) \sin\theta=
\left(\mathrm{Im}(w_1)-\frac{1}{m}\mathrm{Im}(w_2^2)\right)\cos\theta,
\]
where $w_1=\sum_{i,j=1}^m a_{i,j}a_{j,i}$ and $w_2=\sum_{i=1}^m a_{i,i}$.
Thus, if $mw_1\neq w_2^2$, one has
\[
\widehat{\theta}=\arg(w_1-\frac{1}{m}w_2^2).
\]
This concludes the proof.
\end{proof}

\begin{corollary}\label{S4}
Let the matrix $A=[a_{i,j}]\in\C^{m\times m}$ have trace zero. If 
\begin{equation}\label{cond2}
\sum_{i,j=1}^m a_{i,j}a_{j,i}\neq0,
\end{equation}
then the unique closest generalized Hermitian matrix 
$\widehat{A}=[\widehat{a}_{i,j}]\in\C^{m\times m}$ to $A$ in the Frobenius norm is given 
by
\begin{equation}\label{clgherm3}
\widehat{a}_{i,j}= 
\frac{1}{2}(a_{i,j}+\overline{a}_{j,i}\mathrm{e}^{\mathrm{i}\widehat{\theta}}),
%A=\mathrm{e}^{\mathrm{i}\theta}A^*+\gamma\mathrm{e}^{\mathrm{i}\frac{\theta+\pi}{2}}\,I
\end{equation}
where 
\[
\widehat{\theta}= \arg\left(\sum_{i,j=1}^m a_{i,j}a_{j,i}\right).
\]
Moreover, 
\[
{\rm dist}_F(A,{\mathbb G})=\sqrt{\frac{\|A\|_F^2-|\sum_{i,j=1}^m a_{i,j}a_{j,i}|}{2}}.
\]
%simplify
If (\ref{cond2}) is violated, then there are infinitely many matrices 
$\widehat{A}(\theta)=[\widehat{a}_{i,j}(\theta)]\in\C^{m\times m}$, depending on an 
arbitrary angle $\theta$, at the same minimal distance from $A$, namely
\begin{equation}\label{clgherm4}
{\widehat a}_{i,j}(\theta)= 
\frac{a_{i,j}+\overline{a}_{j,i}\mathrm{e}^{\mathrm{i}\theta}}{2},
\qquad 1\leq i,j\leq m.
%A=\mathrm{e}^{\mathrm{i}\theta}A^*+\gamma\mathrm{e}^{\mathrm{i}\frac{\theta+\pi}{2}}\,I
\end{equation}
\end{corollary}

\begin{proof}
The result follows by observing that the optimal values of $\widehat{\theta}$
and $\widehat{\gamma}$ determined by Proposition 
\ref{S3} are given by $\widehat{\theta}= \arg(\mathrm{Trace}(A^2))$ and 
$\widehat{\gamma}=0$.  
\end{proof}

We refer to a generalized Hermitian matrix $A\in\C^{m\times m}$, whose eigenvalues for 
suitable $\varphi\in (-\pi,\pi]$ and $\alpha\in \C$ satisfy
\[
\lambda_i=\rho_i\mathrm{e}^{\mathrm{i}\varphi}+\alpha,\,\,  \mathrm{with} 
\,\,\rho_i\geq 0, \;\;\; 1\leq i\leq m,
\] 
as a generalized Hermitian positive semidefinite matrix. We denote the set
of generalized Hermitian positive semidefinite matrices by ${\mathbb G}_+$.

\begin{proposition}\label{S5}
The matrix $A\in\C^{m\times m}$ is generalized Hermitian positive semidefinite if and only
if there are constants $\varphi\in (-\pi,\pi]$ and $\alpha\in \C$ such that
\begin{equation}
A=\mathrm{e}^{\mathrm{i}\varphi}B+\alpha\,I,
\end{equation}
where the matrix $B\in\C^{m\times m}$ is Hermitian positive semidefinite.
\end{proposition}

\begin{proof}
The proposition follows from the proof of Proposition \ref{S1}, where we use the fact that 
the diagonal entries of the diagonal matrix $D$ are nonnegative. 
\end{proof}

We are interested in measuring the distance between $A$ and the set 
${\mathbb G}_+$ in the Frobenius norm. We deduce from
(\ref{clgherm}) that, if (\ref{cond}) holds, then the unique closest generalized Hermitian 
matrix is of the form
\begin{equation}\label{matid}
\widehat{A}=\frac{A+\mathrm{e}^{\mathrm{i}\widehat{\theta}}A^*+\widehat{\gamma}\,
\mathrm{e}^{\mathrm{i}\frac{\widehat{\theta}+\pi}{2}}I}{2}=
\mathrm{e}^{\mathrm{i}\frac{\widehat{\theta}}{2}}\widetilde{A}+
\frac{\widehat{\gamma}}{2}\,\mathrm{e}^{\mathrm{i}\frac{\widehat{\theta}+\pi}{2}}I,
\end{equation}
where $\widetilde{A}$ denotes the Hermitian part of
$\mathrm{e}^{-\mathrm{i}\frac{\widehat{\theta}}{2}}A$.
The identity \eqref{matid} shows that the unique closest generalized Hermitian positive 
semidefinite matrix to $A$ can be written as
\[
\widehat{A}_+:=\mathrm{e}^{\mathrm{i}\frac{\widehat{\theta}}{2}}\widetilde{A}_+
+\frac{\widehat{\gamma}}{2}\,\mathrm{e}^{\mathrm{i}\frac{\widehat{\theta}+\pi}{2}}I,
\]
where $\widetilde{A}_+$ denotes the Hermitian positive semidefinite matrix closest to 
$\mathrm{e}^{-\mathrm{i}\frac{\widehat{\theta}}{2}}A$. The construction of 
$\widetilde{A}_+$ can be easily obtained following \cite{Hi}. Thus, the distance
\[
{\rm dist}_F(A,{\mathbb G}_+)=\|A-\widehat{A}_+\|_F\geq {\rm dist}_F(A,{\mathbb G})
\]
can be computed similarly as (\ref{distspsd}), taking into account the squared sum of the
negative eigenvalues of $\widetilde{A}$.

\section{Some preconditioning techniques}\label{sec4}
Preconditioning is a popular technique to improve the rate of convergence of GMRES when
applied to the solution of linear systems of equations that are obtained by discretizing
a well-posed problem; see, e.g., \cite{Saadbook} for a discussion and references. This 
technique replaces a linear system of equations (\ref{linsys}) by a left-preconditioned 
system
\begin{equation}\label{leftprec}
MA\bx=M\bb
\end{equation}
or by a right-preconditioned system
\begin{equation}\label{rightprec}
AM\by=\bb,\qquad \bx:=M\by,
\end{equation}
and applies GMRES to the solution of one of these preconditioned systems. The matrix 
$M\in\C^{m\times m}$ is referred to as a preconditioner. In the well-posed setting, $M$ 
typically is chosen so that 
the iterates generated by GMRES when applied to (\ref{leftprec}) or (\ref{rightprec})
converge to the solution faster than iterates determined by GMRES applied to the original
(unpreconditioned) linear system of equations \eqref{linsys}. We would like $M$ to have a 
structure that allows rapid evaluation of matrix-vector products $M\by$, $\by\in\C^m$. One
may apply left- and right-preconditioners simultaneously, too.
%but to keep our discussion brief, we will not discuss this possibility. 

Preconditioning also can be applied to the solution of linear discrete ill-posed problems 
(\ref{linsys}); see, e.g., \cite{DMR,DNR,HNP,HJ0,NR014,RY2}. The aim of the preconditioner $M$ in 
this context is to determine a solution subspace ${\K}_k(MA,M\bb)$ for problem
\eqref{leftprec}, or a solution subspace $M{\K}_k(AM,\bb)$ for
problem \eqref{rightprec}, that contain accurate approximations of $\bx_{\rm exact}$
already when their dimension $k$ is small. Moreover, we
would like to choose $M$ so that the error $\be$ in $\bb$ is not severely amplified and 
propagated into the computed iterates when solving (\ref{leftprec}) or (\ref{rightprec}). 
We seek to achieve these goals by choosing particular preconditioners $M$ such that the 
matrices $MA$ or $AM$ are close to the sets $\H_+$ or $\mathbb{G}_+$. We will also 
comment on the distance of these matrices to the sets $\H$ and $\A$. We remark that 
right-preconditioning generally is more useful than left-preconditioning, because the 
GMRES residual norm for the system (\ref{rightprec}) can be cheaply evaluated by computing
the residual norm of a low-dimensional system of equations. This is a favorable feature 
when a stopping criterion based on the residual norm is used, such as the discrepancy
principle. Henceforth, we focus on right-preconditioning. We describe several novel approaches to construct a preconditioner that can be effective in a variety of situations. 

When the matrix $A$ is a shift operator, GMRES may not be able 
to deliver an accurate approximation of $\bx_{\rm exact}$ within a few iterations (this is 
the case of Example 2.1). To remedy this difficulty, we propose to approximate $A$ by a
circulant matrix $C_A$. We may, for instance, determine $C_A$ as the solution of the 
matrix nearness problem discussed in \cite{CJ,TC,Ng},
\begin{equation}\label{DA}
\min_{C\in\C^{m\times m}~{\rm circulant}}\|C-A\|_F,
\end{equation}
and use the preconditioner 
\begin{equation}\label{prec3a}
M=C_A^{-1}.
\end{equation}
The minimization problem \eqref{DA} easily can be solved by using the spectral 
factorization 
\begin{equation}\label{specCA}
C_A=WD_AW^*,
\end{equation}
where the matrix $D_A\in\C^{m\times m}$ is diagonal and $W\in\C^{m\times m}$ is a unitary
fast Fourier transform (FFT) matrix; see \cite{Da} for details. Hence, 
\[
\|C_A-A\|_F=\|D_A-W^*AW\|_F,
\]
and it follows that $D_A$ is made up of the diagonal entries of $W^*AW$. The computation 
of the matrix $D_A$, with the aid of the FFT, requires ${\mathcal O}(m^2\log_2(m))$ 
arithmetic floating point operations (flops); see \cite{CJ,TC,Ng} for details. Alternatively, a circulant preconditioner may be computed as the solution of the matrix 
nearness problem 
\begin{equation}\label{tyrt}
\min_{C\in\C^{m\times m}\\~{\rm circulant}} \| I - C^{-1}A\|_F.
\end{equation} 
This minimization problem is discussed in \cite{DNR,ST,Ty}. The solution is given by
$C_{AA^*}C_{A^*}^{-1}$; see \cite{Ty}. The flop count for solving \eqref{tyrt}, by using
the FFT, also is ${\mathcal O}(m^2\log_2(m))$; see \cite{CJ,Ng,Ty}.  

A cheaper way to determine a circulant preconditioner \eqref{specCA} is to let 
$\bx\in\C^m$ be a random vector, define $\by:=A\bx$, and then determine the diagonal 
matrix $D_A$ in \eqref{specCA} by requiring that $\by=C_A\bx$. This gives
\begin{equation}\label{C3} 
D_A={\rm diag}[(W^*\by)/(W^*\bx)]\,,
\end{equation}
where the vector division is component-wise. The 
computation of $D_A$ in this way only requires the evaluation of two fast Fourier 
transforms and $m$ scalar divisions, which only demands ${\mathcal O}(m\log_2(m))$ flops. We remark that further approaches to construct circulant preconditioners are discussed in 
the literature; see \cite{CJ,Ng}. Moreover, $\rm{e}^{\rm{i}\theta}$-circulants, which 
allow an angle $\theta$ as an auxiliary parameter can be effective preconditioners; 
they generalize the preconditioners \eqref{DA} and \eqref{tyrt} and also can be 
constructed with ${\mathcal O}(m^2\log_2(m))$ flops; see \cite{NR11,NR014}. 

Having determined the preconditioner $M$, we apply the Arnoldi process to the matrix $AM$
with initial vector $\bb$. The evaluation of each matrix-vector product with $M$ can be 
carried out in ${\mathcal O}(m\log_2(m))$ flops by using the FFT for both circulant and 
$\rm{e}^{\rm{i}\theta}$-circulant preconditioners. Iterations are carried out until the 
discrepancy principle is satisfied. Let $\by_k$ be the solution of (\ref{rightprec}) so
obtained. Then $\bx_k=M\by_k$ is an approximation of $\bx_{\rm exact}$. 

A generic approach to determine a preconditioner $M$ that makes $AM$ closer to the set
$\H_+$ than $A$ is to carry out $\kP$ steps of the Arnoldi process applied to $A$ with 
initial vector $\bb$. Assuming that no breakdown occurs, this yields a decomposition of 
the form (\ref{arndec}) with $k$ replaced by $\kP$, and we define the approximation
\begin{equation}\label{Akprec}
%A_{k_{\rm prec}}:=V_{k_{\rm prec}+1} H_{k_{\rm prec}+1,k_{\rm prec}}V^*_{k_{\rm prec}}.
A_{\kP}:=V_{\kP+1} H_{\kP+1,\kP}V^*_{\kP}
\end{equation}
of $A$. If $A_{\kP}$ contains information about the dominant singular values of the matrix
$A$ only, then $A_{\kP}$ is a regularized approximation of $A$. This property is 
illustrated numerically in \cite{GNR1} for severely ill-conditioned matrices. Moreover, in
a continuous setting and under the assumption that $A$ is a Hilbert--Schmidt operator of 
infinite rank \cite[Chapter 2]{Ri}, it is shown in \cite{N17} that the SVD of $A$ can be 
approximated by computing an Arnoldi decomposition of $A$. This property is inherited in 
the discrete setting of the present paper, whenever a suitable discretization of a 
Hilbert--Schmidt operator is used. 

The approximation (\ref{Akprec}) suggests the simple preconditioner 
\begin{equation}\label{prec1}
M:=A_{\kP}^*.
\end{equation}
The rank of this preconditioner is at most $\kP$ and, therefore, GMRES 
applied to the solution of (\ref{rightprec}) will break down within $\kP$ steps; see,
e.g., \cite{BW,RY} for discussions on GMRES applied to linear systems of equations with a
singular matrix. We would like to choose $\kP$ large enough so that GMRES applied to 
(\ref{rightprec}) does not break down before a sufficiently accurate approximation of 
$\bx_{\rm exact}$ has been determined. The following proposition sheds light on some 
properties of the matrix $AM$ when $M$ is defined by (\ref{prec1}).

\begin{proposition}\label{prop2}
Assume that $\kP$ steps of the Arnoldi process applied to $A$ with initial vector
$\bb$ can be carried out without breakdown, and let the preconditioner $M$ be defined by 
(\ref{prec1}). Then $AM$ is Hermitian positive semidefinite with rank $\kP$, and 
$\mathcal{R}(AM)\subset\mathcal{R}(V_{\kprec +1})$.
\end{proposition}

\begin{proof}
From (\ref{Akprec}) and the decomposition (\ref{arndec}), with $k$ replaced by
$\kP$, it is immediate to verify that
\begin{eqnarray*}
AM&=&AA^*_{\kP}=
AV_{\kP}H_{\kP+1,\kP}^*V_{\kP+1}^* \\ 
&=& V_{\kP+1}H_{\kprec+1,\kP} H_{\kP+1,\kP}^*V_{\kP+1}^* =
C_{\kprec+1,\kP} C_{\kprec+1,\kP}^*, 
\end{eqnarray*}
where
\begin{equation}\label{Ckprec}
C_{\kprec+1,\kP} = V_{\kprec + 1}H_{\kprec+1,\kP}\in\C^{m\times \kprec}
\end{equation}
is a matrix of rank at most $\kprec$. Finally, for any $\bz\in\C^{m}$, we have
\[
AM\bz = V_{\kP+1}H_{\kP+1,\kP}H_{\kP+1,\kP}^*V_{\kP+1}^*\bz\,.
\]
This shows that $AM\bz\in V_{\kprec+1}$.
\end{proof}

Since $AA^*_{\kP}$ is singular, problem (\ref{rightprec}) should be considered a
least-square problem, i.e., instead of solving \eqref{rightprec} one should compute
\begin{equation}\label{LSrightprec}
\by = \arg\min_{\widehat{\by}\in\C^m}\left\Vert 
C_{\kprec+1,\kP}C_{\kprec+1,\kP}^*\widehat{\by}-\bb\right\Vert,\qquad \bx=A_{\kprec}^*\by\,,
\end{equation}
where $C_{\kprec+1,\kP}$ is defined by (\ref{Ckprec}). It follows from the definition 
\eqref{Akprec} of $A_{\kP}$, and the fact that 
$\mathcal{R}(V_{\kprec})=\K_{\kprec}(A,\bb)$, that the solution $\bx$ of 
(\ref{LSrightprec}) belongs to $\K_{\kprec}(A,\bb)$. A regularized solution of 
the minimization problem (\ref{LSrightprec}) can be determined in several ways.
For instance, one can apply a few steps of the Arnoldi process (Algorithm \ref{alg:arnoldi}) 
to compute an approximate solution of the least-squares problem \eqref{LSrightprec}, i.e.,
one applies the Arnoldi process to the matrix $C_{\kprec+1,\kP}C_{\kprec+1,\kP}^*$ with initial 
vector $\bv_1=\bb/\|\bb\|$. We note that the latter application of the Arnoldi process  
does not require additional matrix-vector product evaluations 
with the matrix $A$. Alternatively, we may determinine a regularized solution of 
\eqref{LSrightprec} by using Tikhonov regularization or applying the truncated singular 
value decomposition (TSVD) of the matrix $C_{\kprec+1,\kP}$. We will discuss the latter
regularization techniques in detail in Section \ref{sec5}. Computational experiments 
reported in Section \ref{sec7} show that it is often possible to determine a meaningful 
approximation of  $\bx_{\rm exact}$ by computing a regularized solution of 
(\ref{LSrightprec}) even when GMRES applied to the original problem (\ref{linsys}) yields 
a poor approximation of $\bx_{\rm exact}$.

The approximation (\ref{Akprec}) of $A$ also can be used to define the preconditioner
\begin{equation}\label{prec2}
M:=A_{\kP}^* + (I-V_{\kP}V^*_{\kP})=V_{\kP}H_{\kP+1,\kP}^*V_{\kP+1}^*+(I-V_{\kP}V_{\kP}^*).
\end{equation}
The number of steps $\kP$ should be chosen so that the matrix $AM$ is fairly close to 
the set $\H_+$. Differently from the preconditioner \eqref{prec1}, the preconditioner 
(\ref{prec2}) has rank $m$, as $\mathcal{R}(M) = 
\mathcal{R}(V_{\kprec})\oplus \mathcal{R}(V_{\kprec}^{\perp})$. The preconditioned 
coefficient matrix defined by the preconditioner \eqref{prec2},
\begin{equation}\label{AM2}
\begin{array}{rcl}
AM&=&AA^*_{\kP}+A(I-V_{\kP}V_{\kP}^*)\\
&=& V_{\kP+1}H_{\kP+1,\kP} H_{\kP+1,\kP}^*V_{\kP+1}^* + A(I-V_{\kP}V_{\kP}^*),
\end{array}
\end{equation}
is non-Hermitian. A few steps of the Arnoldi process (Algorithm \ref{alg:arnoldi}) can be 
applied to the matrix \eqref{AM2} to determine a regularized solution of \eqref{rightprec}. 
However, differently from the situation when using the preconditioner (\ref{prec1}), this requires additional 
matrix-vector product evaluations with $A$. Regularization of \eqref{rightprec} 
when the preconditioner is defined by \eqref{prec2} can again be achieved by applying Tikhonov 
or TSVD regularization. An analogue of Proposition \ref{prop2} does not hold for the preconditioner $M$ defined by 
(\ref{prec2}). Instead, we can show the following result.

\begin{proposition}\label{prop3}
Assume that $\kP+j$ steps of the Arnoldi process applied to $A$ with initial
vector $\bb$ can be carried out without breakdown, and let the preconditioner $M$ be 
defined by (\ref{prec2}). Then the iterate $\by_j$ determined in the $j$th step of GMRES
applied to the preconditioned system (\ref{rightprec}) with initial approximate solution 
$\by_0=\bzero$ belongs to the Krylov subspace $\K_{\kP+j}(A,\bb)$.
\end{proposition}

\begin{proof}
We show the proposition by induction. It is immediate to verify that
\[
\by_1\in\K_1(AM,\bb) =\spn\{\bb\}=\K_1(A,\bb)\subset \K_{\kprec}(A,\bb)\subset\K_{\kprec+1}(A,\bb)\,.
\]
Assume that $\by_i\in\K_{\kprec+i}(A,b)$. Then, since
%$\by_{i+1}\in\K_{i+1}(AM,b)\subset\spn\{b,\,AM\K_{\kprec + i}(A,b)\}\,,$ and
%\[
%\K_{i+1}(AM,b)\subset\spn\{b,\,AM\K_{\kprec + i}(A,b)\}\,,
%\]
\[
\by_{i+1}\in\K_{i+1}(AM,\bb)\subset\spn\{\bb,\,AM\K_{\kprec + i}(A,\bb)\}\,,
\]
$\by_{i+1}$ is a linear combination of vectors of this subspace, i.e.,
\[
\by_{i+1}=s_1\bb+AMV_{\kprec+i}\bs_{\kprec+i}=s_1\bb+V_{\kprec+1}\bs_{\kprec+1}+V_{\kprec+i+1}\bs_{\kprec+i+1}\,,
\]
where $s_1\in\C$, $\bs_{\kprec+1}\in\C^{\kprec+1}$, $\bs_{\kprec+i}\in\C^{\kprec+i}$, and 
$\bs_{\kprec+i+1}\in\C^{\kprec+i+1}$. Here we have used the definition \eqref{prec2} of 
$M$ and the Arnoldi decomposition (\ref{arndec}), with $k$ replaced by $\kprec$. Hence, 
$\by_{i+1}\in\mathcal{R}(V_{\kprec+i+1}) = \K_{\kprec+i+1}(A,\bb)$.
%The proof follows by induction on $j$. Fow $j=1$ we get $\by_1\in\mathrm{span}\{\bb\}\subset\K_{\kprec}(A,\bb)$. Now assume that \linebreak[4]$\by_i\in\K_{\kprec+i-1}(A,\bb)$, i.e., $\by_i=V_{\kprec+i-1}\bt_{\kprec+i-1}$, where $\bt_{\kprec+i-1}\in\C^{\kprec+i-1}$. Then, considering that
%\[
%\by_{i+1}\in\K_{i+1}(AM, \bb)\supset AM\K_i(AM,b)\,,
%\]
%one gets in particular that 
%\begin{eqnarray*}
%AMV_{\kprec+i-1}\bs_{\kprec+i-1}&=&AA_{k}^*V_{\kprec+i-1}\bs_{\kprec+i-1}\\
%&+& AV_{\kprec+i-1}\bs_{\kprec+i-1}\\
%&-& AV_{\kprec}V_{\kprec}^*V_{\kprec+i-1}\bs_{\kprec+i-1}.
%\end{eqnarray*}
%The second term in the above sum belongs to $\K_{\kprec+i}(A.\bb)$, so that \linebreak[4]$\by_{i+1}\in\K_{\kprec+i}(A,\bb)$. %This concludes the proof.
%%We obtain similarly as in the proof of Proposition \ref{prop2} that
%%\begin{eqnarray*}
%%AM&=&AA^*_{k_{\rm prec}}+A(I-V_{k_{\rm prec}}V_{k_{\rm prec}}^*)\\
%%&=& V_{k_{\rm prec}+1}H_{k_{\rm prec}+1,k_{\rm prec}}
%%H_{k_{\rm prec}+1,k_{\rm prec}}^*V_{k_{\rm prec}+1}^*+
%%A(I-V_{k_{\rm prec}}V_{k_{\rm prec}}^*).
%%\end{eqnarray*}
%%It follows that $AM\bb\in\K_{k_{\rm prec}+1}(A,\bb)$. Analogously,
%%$(AM)^2\bb\in\K_{k_{\rm prec}+2}(A,\bb)$. The proposition now follows by 
%%straighforward computations.
\end{proof}

Assume that the conditions of Proposition \ref{prop3} hold, and let
\linebreak[4]$\by_j = V_{\kprec + j}\bs_{\kprec + j}$ with 
$\bs_{\kprec + j}\in\C^{\kprec + j}$. Then the corresponding approximate solution $\bx_j$
of (\ref{rightprec}) satisfies
\[
\bx_j=M\by_j = MV_{\kprec + j}\bs_{\kprec + j}\in\mathcal{R}(V_{\kprec + j})=
\K_{\kprec + j}(A,\bb)\,.
\]
Hence, application of the GMRES method with the right-preconditioner 
(\ref{prec2}) determines an approximate solution in the (unpreconditioned) Krylov subspace \linebreak[4]
$\K_{\kP+j}(A,\bb)$. 

We conclude this section by considering two more preconditioners, that are related to (\ref{prec1}) and 
(\ref{prec2}), and which may enhance the regularization properties of Arnoldi methods even if they are not 
designed with the goal of reducing the distance of $A$ to the sets $\H$ or $\H_+$. Assume, as above, 
that the Arnoldi algorithm does not break down during the first $\kprec$ steps. Then the matrix 
$A_{\kprec}$ defined by (\ref{Akprec}) can be computed, and one may use 
\begin{equation}\label{prec3}
M := A_{\kprec}\,.
\end{equation}
as a preconditioner. Similarly to (\ref{prec1}), this preconditioner has at most rank $\kprec$ and, assuming 
that $A_{\kprec}$ only contains information about the $\kprec$ dominant singular values of $A$, $M$ may be 
regarded as a regularized approximation of $A$. Note that, by exploiting the Arnoldi decomposition (\ref{arndec}) 
with $k$ replaced by $\kprec+1$, one obtains the following expression
\begin{equation}\label{Aprec3}
AM = V_{\kprec+2}H_{\kprec+2,\kP+1}H_{\kprec+1,\kP}V_{\kprec}^*.
\end{equation}
We note that when applying a few (at most $\kP$) steps of GMRES to compute an approximate 
solution of the preconditioned system (\ref{rightprec}), no additional matrix-vector 
product evaluations with the matrix $A$ are necessary, in addtion to the $\kprec+1$ 
matrix-vector product evaluations required to determine the right-hand side of 
(\ref{Aprec3}). The iterate $\bx_j$ determined at the $j$th step of GMRES applied to the 
preconditioned system (\ref{rightprec}) belongs to 
$\mathcal{R}(V_{\kprec + 2})=\K_{\kprec + 2}(A,\bb)$.

The preconditioner 
\begin{equation}\label{prec4}
M:=A_{\kP} + (I-V_{\kP}V^*_{\kP})=V_{\kP+1}H_{\kP+1,\kP}V_{\kP}^*+(I-V_{\kP}V_{\kP}^*).
\end{equation}
is analogous to (\ref{prec2}). This preconditioner also was considered in \cite{LRT} in 
the framework of the solution of a sequence of slowly-varying linear systems of equations.
Similarly to (\ref{prec3}), the preconditioner (\ref{prec4}) does guarantee that the 
precondioned matrix $AM$ is close to the set $\H_+$. By using the Arnoldi decomposition 
(\ref{arndec}) with $k$ replaced by $\kP+1$, we obtain 
\[
AM=AA_{\kP}+A(I-V_{\kP}V_{\kP}^*)
= V_{\kP+2}H_{\kP+2,\kP+1} H_{\kP+1,\kP}V_{\kP}^* + A(I-V_{\kP}V_{\kP}^*).
\]
It is evident that, even if $\kP+1$ steps of the Arnoldi process have been carried out to
define $M$, additional matrix-vector products with $A$ are required when applying the Arnoldi 
process to the preconditioned system (\ref{rightprec}). Using the same arguments as in 
Proposition \ref{prop3}, one can show that, if $\kP + j$ steps of the Arnoldi process applied 
to $A$ with initial vector $\bb$ can be carried out without breakdown, then the iterate $\by_j$ 
determined at the $j$th iteration of GMRES applied to the preconditioned system (\ref{rightprec}) 
and the corresponding approximate solution $\bx_j=M\by_j$ of \eqref{linsys} belong to 
$\K_{\kP+j}(A,\bb)$. We note that  Tikhonov or TSVD regularization can be applied when solving 
the preconditioned system (\ref{rightprec}) with either one of the preconditioners (\ref{prec3}) 
or (\ref{prec4}).

\section{Solving the preconditioned problems}\label{sec5}
As already suggested in the previous section, instead of using GMRES to solve the 
preconditioned system (\ref{rightprec}) with one of the preconditioners described, one may
wish to apply additional regularization in order to determine an approximate solution of 
\eqref{linsys} of higher quality. In the following we discuss application 
of Tikhonov and TSVD regularilzation. We refer to the solution methods so obtained as
the Arnoldi--Tikhonov and Arnoldi-TSVD methods, respectively. Due to the additional 
regularization, both these method allow the use of a solution subspace of larger dimension 
than preconditioned GMRES without additional regularization. This helps reduce so-called 
``semi-convergence''. 

The Arnoldi--Tikhonov method for (\ref{rightprec}) determines an approximate solution 
$\bx_{\mu}$ of (\ref{linsys}) by first computing the solution $\by_\mu$ of the Tikhonov 
minimization problem
\begin{equation}\label{tik1}
\min_{\by\in\K_k(AM,\bb)}\{\|AM\by-\bb\|^2+\mu\|\by\|^2\},
\end{equation}
where $\mu>0$ is a regularization parameter to be specified, and then evaluates the
approximation $\bx_\mu = M\by_\mu$ of $\bx_{\rm exact}$. The minimization problem 
\eqref{tik1} has a unique solution for any $\mu>0$. Application of $k$ steps of the 
Arnoldi process to the matrix $AM$ with initial vector $\bb$ gives the Arnoldi 
decomposition
\begin{equation}\label{arndec2}
AMV_k=V_{k+1}H_{k+1,k}\,,
\end{equation}
which is analogous to (\ref{arndec}). Using (\ref{arndec2}), the minimization
problem \eqref{tik1} can be expressed as the reduced Tikhonov minimization problem
\begin{equation}\label{tik2}
\min_{\bz\in\C^k}\{\|H_{k+1,k}\bz-\|\bb\|\be_1\|^2+\mu\|\bz\|^2\},
\end{equation}
whose minimizer $\bz_\mu$ gives the approximate solution $\by_\mu:=V_k\bz_\mu$ of 
(\ref{tik1}), so that\linebreak[4] $\bx_\mu:=M\by_\mu$ is an approximate solution of (\ref{linsys}). 

The Arnoldi-TSVD method seeks to determine an approximate solution of (\ref{rightprec}) by
using a truncated singular value decomposition of the (small) matrix $H_{k+1,k}$ in 
(\ref{arndec2}). Let $\by:=V_k\bz$. Then, using (\ref{arndec2}), we obtain
\begin{equation}\label{reduced}
\min_{\by\in\K_k(AM,\bb)}\|AM\by-\bb\|=\min_{\bz\in\C^k}\|H_{k+1,k}\bz-\|\bb\|\be_1\|.
\end{equation}
Let $H_{k+1,k}=U_{k+1}\Sigma_{k}W_k^*$ be the singular value decomposition. Thus, the
matrices $U_{k+1}\in\C^{(k+1)\times(k+1)}$ and $W_k\in\R^{k\times k}$ are unitary, and
\[
\Sigma_{k}={\rm diag}[\sigma_1^{(k)},\sigma_2^{(k)},\ldots,\sigma_k^{(k)}]\in\R^{(k+1)\times k} 
\]
is diagonal (and rectangular), with nonnegative diagonal entries ordered according to
$\sigma_1^{(k)}\geq\sigma_2^{(k)}\geq\ldots\geq\sigma_k^{(k)}\geq 0$. Define the diagonal
matrix 
\[
\Sigma_{k}^{(j)}={\rm diag}[\sigma_1^{(k)},\ldots,\sigma_j^{(k)},0,\ldots,0]\in
\R^{(k+1)\times k}
\]
by setting the $k-j$ last diagonal entries of $\Sigma_{k}$ to zero, and introduce the
associated rank-$j$ matrix $H_{k+1,k}^{(j)}:=U_{k+1}\Sigma_{k}^{(j)}W_k^*$. Let $\bz^{(j)}$ 
denote the minimal norm solution of 
\begin{equation}\label{red2}
\min_{\bz\in\C^k}\|H_{k+1,k}^{(j)}\bz-\|\bb\|\be_1\|.
\end{equation}
Problem (\ref{red2}) is the truncated singular value decomposition (TSVD) 
method applied to the solution of the reduced minimization problem in the right-hand side of 
(\ref{reduced}); see, e.g., \cite{EHN,Hansenbook} for further details on the TSVD method. 
Once the solution $\bz^{(j)}$ of (\ref{red2}) is computed, we get the approximate solution 
$\by^{(j)}:=V_k\bz^{(j)}$ of (\ref{reduced}), from which we obtain the approximate solution
$\bx^{(j)}:=M\by^{(j)}$ of (\ref{linsys}). A modified TSVD method described in \cite{NR0}
also can be used.

All the methods discussed in this section are inherently multi-parameter, i.e., their 
success depends of the appropriate tuning of more than one regularization parameter. In 
the remainder of this section we will discuss reliable strategies to effectively choose
these parameters. First of all, when the preconditioners (\ref{prec1}), (\ref{prec2}), 
(\ref{prec3}), and (\ref{prec4}) are used, an initial number of Arnoldi iterations, $\kprec$,
has to be carried out. Since we would like these preconditioners $M$ to be suitable 
regularized approximations of the matrix $A$, a natural way to determine $\kprec$ is to 
monitor the expansion of the Krylov subspace $\K_{\kprec}(A,\bb)$. The subdiagonal 
elements $h_{i+1,i}$, $i=1,2,\dots,k$, of the Hessenberg matrix $H_{k+1,k}$ in (\ref{arndec}) are
helpful in this respect; see \cite{GNR2,NR14}. We terminate the initial Arnoldi 
process as soon as an index $\kP$ such that
\begin{equation}\label{stop1}
h_{\kP+1,\kP}<\tau_1^\prime\quad\mbox{and}\quad \frac{\left\vert h_{\kP+1, \kP} - h_{\kP, \kP-1}\right\vert}{h_{\kP, \kP-1}}>\tau_1^{\prime\prime}
\end{equation}
is found. By choosing $\tau_1^\prime$ small, we require some stabilization to take
place while generating the Krylov subspace $\K_k(A,\bb)$; simultaneously, by setting 
$\tau_1^{\prime\prime}$ close to 1, we require the subdiagonal entries of $H_{k+1,k}$ to
stabilize. In terms of regularization, this criterion is partially justified by the bound
\[
\prod\nolimits_{j=1}^{\kprec}h_{j+1,j}\leq \prod\nolimits_{j=1}^{\kprec}\sigma _{j}\,,
\]
see \cite{Mo}, which states that, on geometric average, the sequence 
$\{h_{j+1,j}\}_{j\geq 1}$ decreases faster than the singular values. Numerical 
experiments reported in \cite{GNR2} indicate that the quantity 
$\Vert A-V_{k+1}H_{k+1,k}V_{k}^*\Vert $ decreases to zero as $k$ increases with about the same
rate as the singular values of $A$. More precisely, even though no theoretical results are 
available at present, one can experimentally verify that typically
\[
\Vert A-V_{k+1}H_{k+1,k}V_{k}^*\Vert \simeq {\sigma}_{k+1}^{(k+1)},
\]
where $\sigma_{k+1}^{(k+1)}$ is the $(k+1)$st singular value of $H_{k+2,k+1}$ ordered in
decreasing order. Here $\|\cdot\|$ denotes the spectral norm of the matrix.

Note that there is no guarantee that the above estimate is tight: 
Firstly, we would have equality only if the matrices $V_{k+1}U_{k+1}$ and $V_{k}W_{k}$ 
coincide with the matrices with the right and left singular vectors of the TSVD of the 
matrix $A$. If this is not the case, then we may have 
$\Vert A-V_{k+1}H_{k}V_{k}^*\Vert \gg {\sigma}_{k+1}^{(k+1)}$. Secondly, one cannot 
guarantee that ${\sigma}_{k+1}^{(k+1)}\geq \sigma _{k+1}$. Nevertheless, experimentally
it appears reliable to terminate the Arnoldi iterations when the product 
$p_{\sigma}^{(k)}:={\sigma}_{1}^{(k)}{\sigma}_{k+1}^{(k+1)}$ is sufficiently small, i.e., one should stop 
as soon as
\begin{equation}
p_{\sigma}^{(\kP)}:={\sigma}_{1}^{(\kP)}{\sigma}_{\kP+1}^{(\kP+1)}<\tau_2\,,  \label{stop2}
\end{equation}%
where $\tau_2$ is a user-specified threshold.
%When setting the number of iterations of the Arnoldi-MINRES (\ref{MINRESredrel}) and Arnoldi-CG (\ref{CGredrel}) methods, 
%some standard parameter choice strategies can be used. For instance, if one
%has a good estimate of the noise level $\delta$,
%$=\|e\|/\|b^{\text{ex}}\|$,
%the discrepancy principle can be applied and the
%iterations can be stopped as soon as
%\begin{equation}
%\Vert \bb-A\bx_{j}\Vert = \Vert \bb-C_{\kP}C_{\kP}^*\by_{j}\Vert =\left\Vert \Vert \bb\Vert \be_{1}-H%
%_{\kP}H_{\kP}^*\bt_j\right\Vert <\tau \delta\,,  \label{stopI}
%\end{equation}%
%where $\tau$ and $\delta$ are the same as in (\ref{discr}).

Once the preconditioner $M$ has been determined, other regularization parameters should be suitably chosen: 
Namely, the number of preconditioned Arnoldi iterations and, in case the Arnoldi--Tikhonov (\ref{tik2}) 
or Arnoldi-TSVD (\ref{red2}) methods are considered, one also has to determine a value for the regularization 
parameter $\mu>0$ or truncation parameter $j\in\mathbb{N}$, respectively. Since choosing the number of 
Arnoldi iterations is less critical (i.e., one can recover good solutions provided that 
suitable values for $\mu$ or $j$ are set at each iteration), we propose to use the discrepancy principle 
to determine the latter and stop only when a maximum number of preconditioned Arnoldi iterations have been 
computed. Specifically, when using the Arnoldi--Tikhonov method, we choose $\mu$ so that the computed 
solution $\bx_\mu$ satisfies 
\begin{equation}\label{discr2}
\|A\bx_\mu-\bb\|=\tau\delta.
\end{equation}
We remark
that this $\mu$-value can be computed quite rapidly by substituting the Arnoldi 
decomposition (\ref{arndec2}) into (\ref{discr2}); see \cite{CMRS,GNR2,LR} for discussions
on unpreconditioned Tikhonov regularization. There also are other approaches to determining 
the regularization parameter; see, e.g., \cite{Ki,RR}.
When applying the Arnoldi-TSVD method, we choose $j$ as small as
possible so that the discrepancy principle is satisfied, i.e., 
\begin{equation}\label{discr3}
\|H_{k+1,k}^{(j)}\bz^{(j)}-\|\bb\|\be_1\|\leq\tau\delta,
\end{equation}
and tacitly assume that $j<k$; otherwise $k$ has to be increased.
For most reasonable values of
$\tau$ and $\delta$, the equations (\ref{discr2}) and (\ref{discr3}) have a unique solution 
$\mu>0$ and $j>0$, respectively. 

\section{Computed examples}\label{sec7}
This section illustrates the performance of the preconditioners introduced in Section 
\ref{sec4} used with GMRES, or with the Arnoldi--Tikhonov and Arnoldi-TSVD methods described 
in Section \ref{sec5}. The Arnoldi algorithm is implemented with reorthogonalization. A first set of experiments considers moderate-scale test problems 
from \cite{Ha}, and takes into account the preconditioners described in the second part of Section \ref{sec4} only. A second set of experiments considers realistic large-scale problems 
arising in the framework of 2D image deblurring, and also includes comparisons with circulant preconditioners. Comparisons with the unpreconditioned 
counterparts of these methods are presented. All the 
computations were carried out in MATLAB R2016b on a single processor 2.2 GHz Intel Core i7
computer.

To keep the notation light, we let $C_1$, $C_2$, and $C_3$ be the preconditioners obtained by solving (\ref{DA}), (\ref{tyrt}), (\ref{C3}), respectively. Also, we let $M_1$, $M_2$, $M_3$, and $M_4$ be the preconditioners
in (\ref{prec1}), (\ref{prec2}), (\ref{prec3}), (\ref{prec4}), respectively. 
%$M_2$ be the preconditioner in (\ref{prec2}), $M_3$ be the 
%preconditioner in (\ref{prec3}), and $M_4$ the preconditioner in (\ref{prec4}).
The unpreconditioned GMRES, Arnoldi--Tikhonov, and Arnoldi-TSVD methods are 
referred to as ``GMRES'', ``Tikh'', and ``TSVD'', respectively; their preconditioned 
counterparts are denoted by ``GMRES($P_a$)'', ``Tikh($P_a$)'', and ``TSVD($P_a$)'', 
where $P\in\{C,M\}$ and $a\in\{1,2,3,4\}$. In the following graphs, specific markers are used for the 
different preconditioners: `$\circ$' denotes $C_1$, `$\Box$' denotes $C_2$, `$\lhd$' 
denotes $C_3$, `$\diamond$' denotes $M_4$, and `$\ast$' indicates that no preconditioner
is used. For some test problem we report results for LSQR, with associated marker `$+$'. 
The stopping criteria (\ref{stop1}), (\ref{stop2}), (\ref{discr2}), and (\ref{discr3}) are
used with the parameters $\tau_1^\prime = 10^{-4}$, $\tau_1^{\prime\prime}=0.9$, 
$\tau_2=10^{-10}$, and $\tau = 1.01$. We use the relative reconstruction error, 
defined by $\|\bx_{\rm exact} - \bx_k\|/\|\bx_{\rm exact}\|$ or 
$\|\bx_{\rm exact} - \bx_\mu\|/\|\bx_{\rm exact}\|$, as a measure of the reconstruction 
quality.

\paragraph{First set of experiments}
We consider problems \eqref{linsys} with a nonsymmetric coefficient matrix of size $m=200$
and a right-hand side vector that is affected by Gaussian white noise, with relative noise 
level $\|\be\|/\|\bb\|= 10^{-2}$. For all the tests, the maximum allowed number of Arnoldi 
iterations in Algorithm \ref{alg:arnoldi} is $k=60$. 

\texttt{\textbf{baart}}. This is a Fredholm integral equation of the first kind \cite{Ba}. All the methods 
are tested with and without additional regularization, and with different preconditioners. The standard GMRES 
method is known to perform well on this test problem; nonetheless, we can experimentally show that the new 
preconditioned solvers can outperform GMRES. In the left frames of Figure \ref{fig:baart}, we report the 
relative error history for different preconditioners (also defined with different parameters $\kP$) and for
different solvers. We can clearly see that, if no additional Tikhonov or TSVD regularization is incorporated 
(top left frame of Figure \ref{fig:baart}), ``semi-convergence'' appears after only few steps, though its 
effect is less evident when the preconditioner $M_2$ is used. When additional regularization in Tikhonov or 
TSVD form is incorporated (mid and bottom left frames of Figure \ref{fig:baart}), all the preconditioned methods 
are more stable and exhibit improved relative errors (when compared with GMRES without Tikhonov or TSVD 
regularization. For the present test problem, the preconditioners $M_3$ and $M_4$ perform the best. Indeed, 
the reconstructions displayed in the right-hand side of Figure \ref{fig:baart} show that the boundary values of 
the solution are accurately recovered when $M_3$ or $M_4$ are used. Applying the stopping rule (\ref{stop2}) 
to determine the number of Arnoldi steps that define the preconditioner yields $\kP=9$; the stopping rule 
(\ref{stop1}) gives the same value. We report the behavior of relevant quantities used to set $\kP$ in the 
top frames of Figure \ref{fig:sc}. Note that increasing the number of Arnoldi iterations, $\kP$, is not 
always beneficial. Indeed, a larger $\kP$-value may result in a more severe loss of orthogonality in the 
Arnoldi process (Algorithm \ref{alg:arnoldi}), even if reorthogonalization is used, so that numerical inaccuracies may affect the computation of 
all the preconditioners (\ref{prec1})-(\ref{prec4}). Moreover, preconditioners (\ref{prec1}) and (\ref{prec3}) 
should be a rank-$\kP$ regularized approximation of the original matrix $A^T$ and $A$, respectively; by increasing 
$\kP$ these approximations become increasingly ill-conditioned and, therefore, less successful in regularizing the 
problem at hand. The best relative errors attained by each iterative method (considering different choices of 
solvers and preconditioners) are reported in Table \ref{tab:expo}, where averages over 30 different realizations 
of the noise in the vector $\bb$ are shown.
%We set $N=200$, so that $\|A-A^T\|/\|A\|=6.0345\cdot 10^{-1}$. The value \linebreak[4]$\tau = 10^{-10}$ is chosen for the stopping criterion in (\ref{stopO}), which is satisfied after 9 iterations; $\tau^{\prime}=10^{-14}$ is chosen for (\ref{stopOvar3}), which holds after 15 iterations. 
%The modified Gram-Schmidt implementation of the Arnoldi algorithm is considered.
\begin{figure}[tbp]
\centering
\begin{tabular}{cc}
\hspace{-0.7cm}\textbf{{\small {(a)}}} & \hspace{-0.7cm}\textbf{{\small {(b)}}} \\
\hspace{-0.7cm}\includegraphics[width=6.2cm]{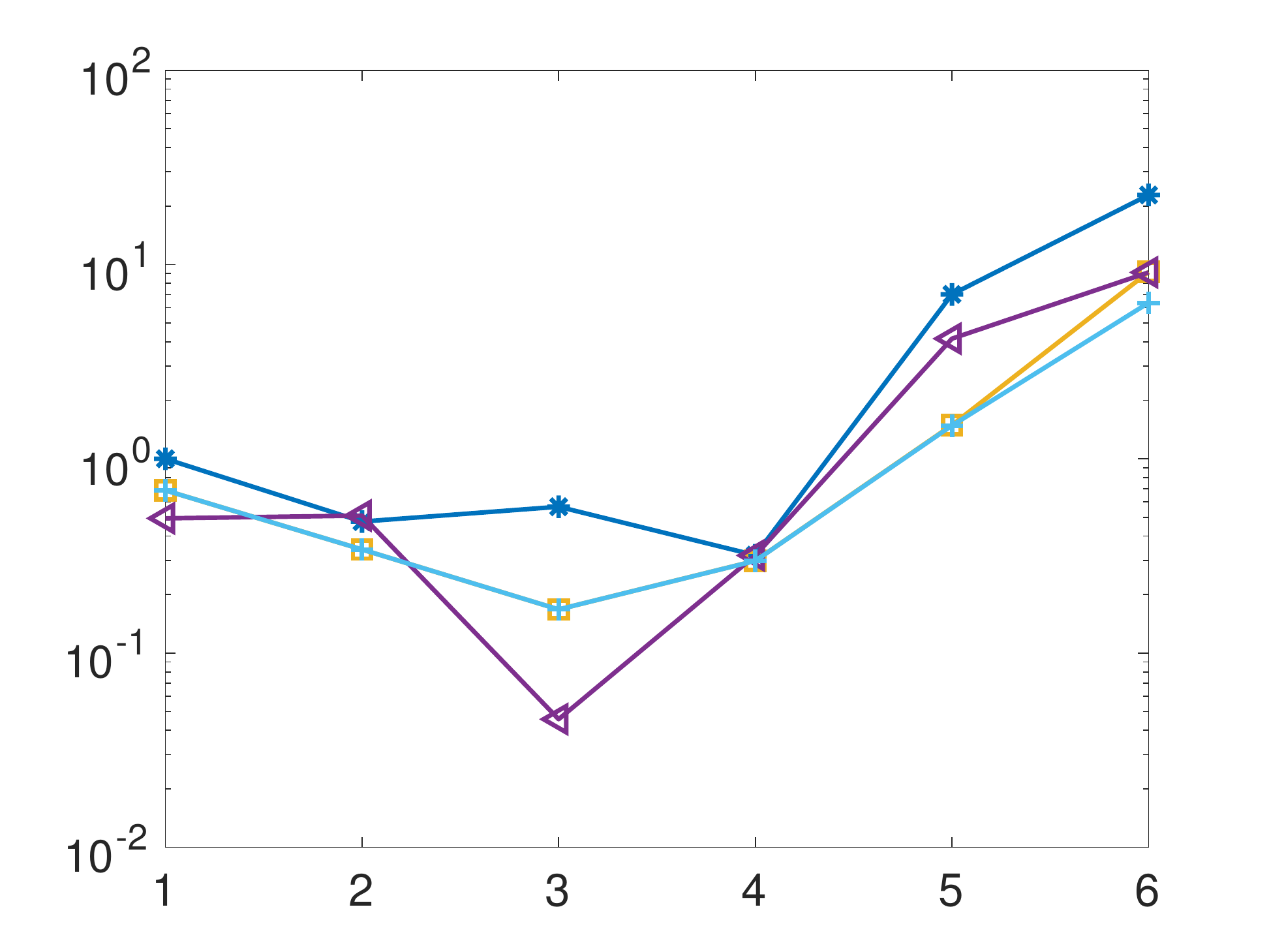} &
\hspace{-0.7cm}\includegraphics[width=6.2cm]{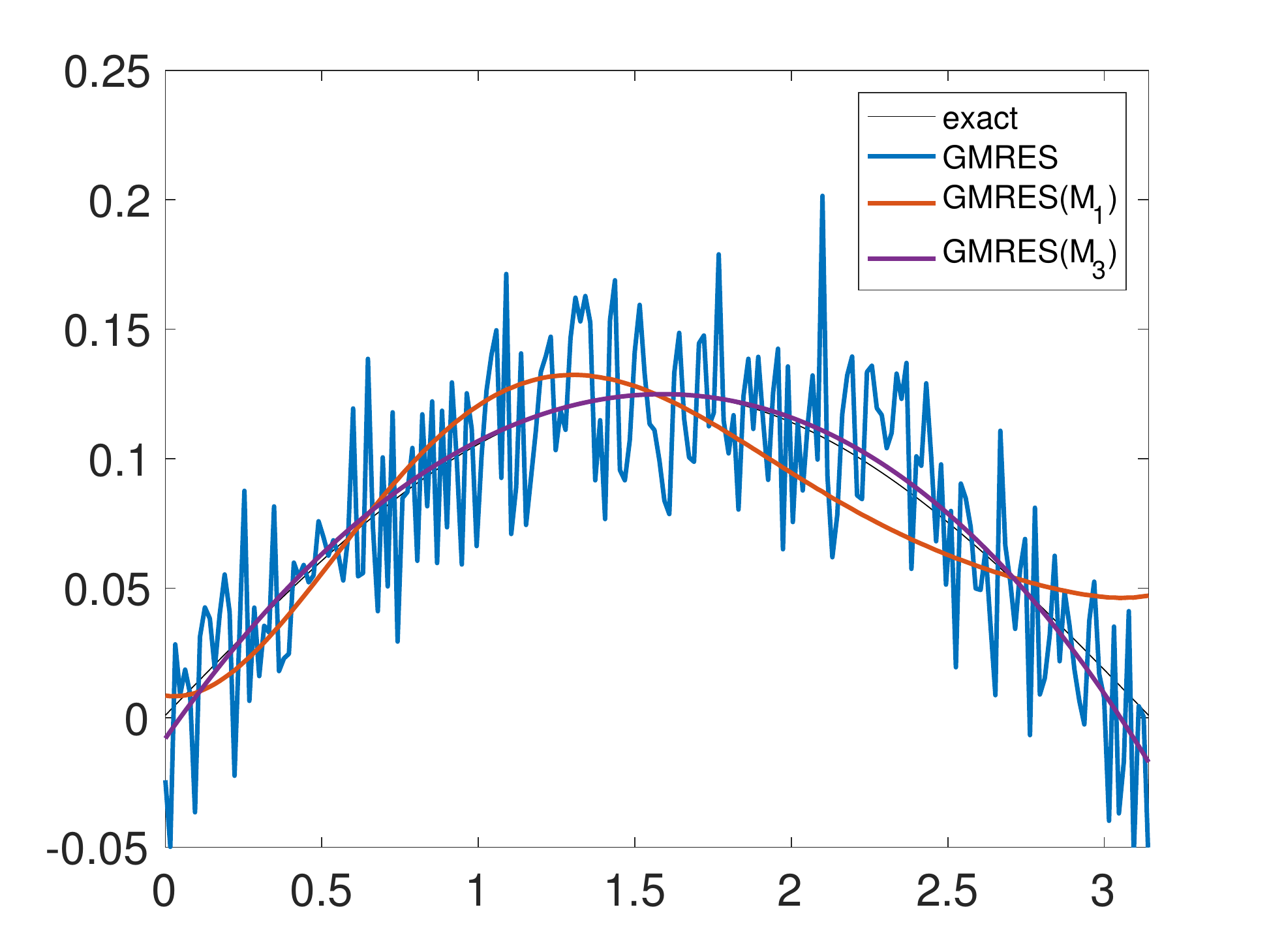}\\
\hspace{-0.7cm}\textbf{{\small {(c)}}} & \hspace{-0.7cm}\textbf{{\small {(d)}}} \\
\hspace{-0.7cm}\includegraphics[width=6.2cm]{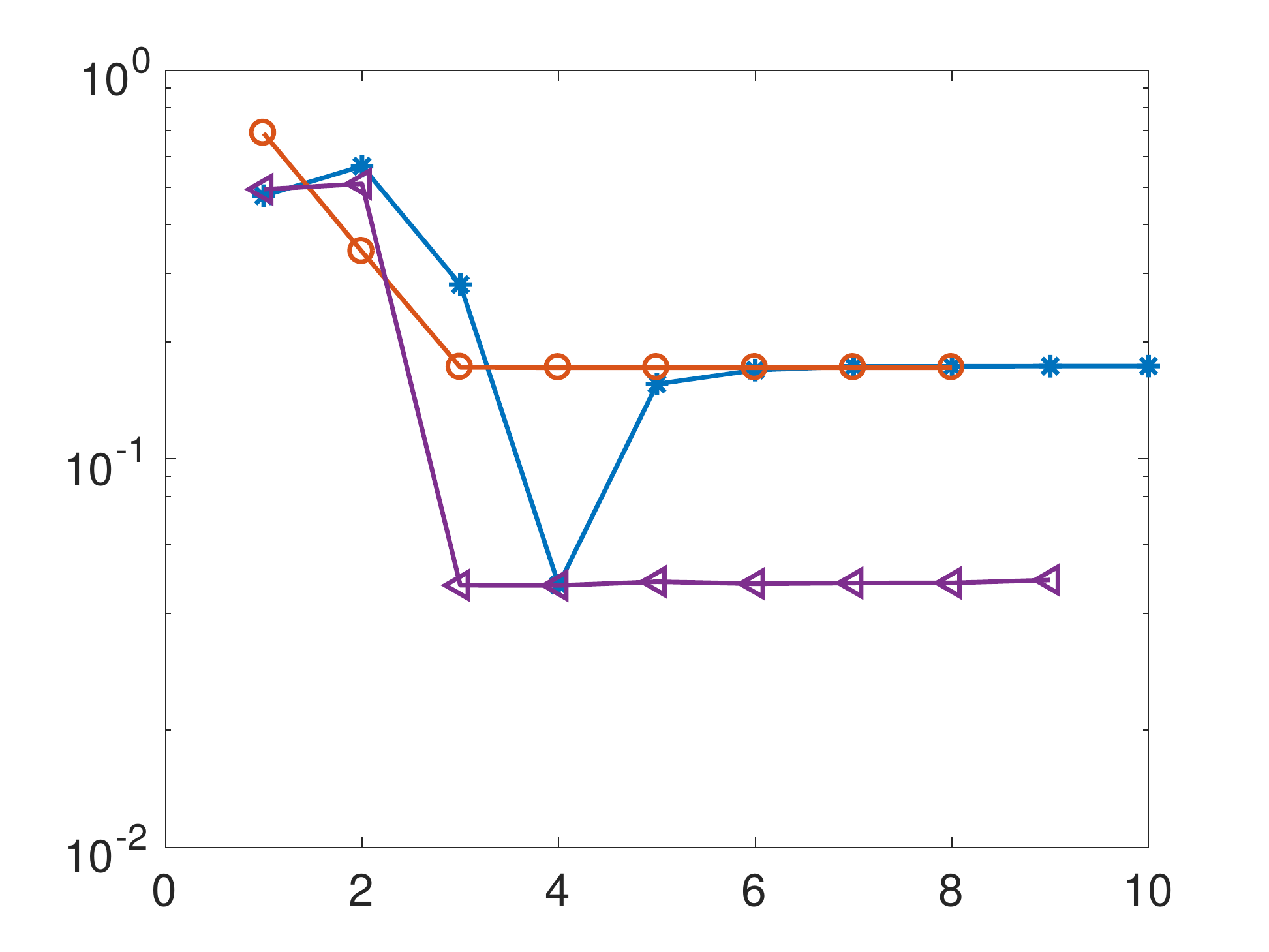} &
\hspace{-0.7cm}\includegraphics[width=6.2cm]{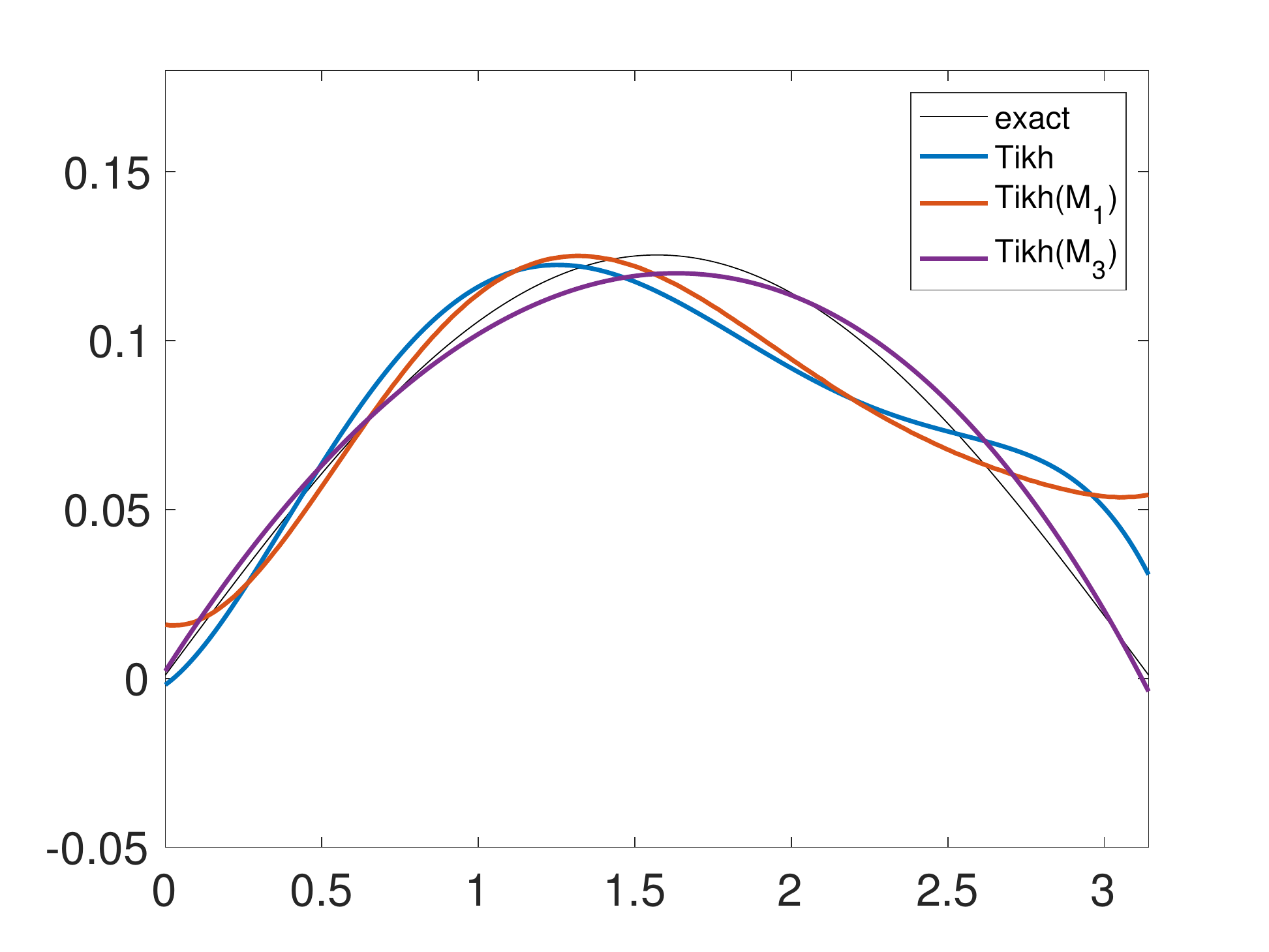}\\
\hspace{-0.7cm}\textbf{{\small {(e)}}} & \hspace{-0.7cm}\textbf{{\small {(f)}}} \\
\hspace{-0.7cm}\includegraphics[width=6.2cm]{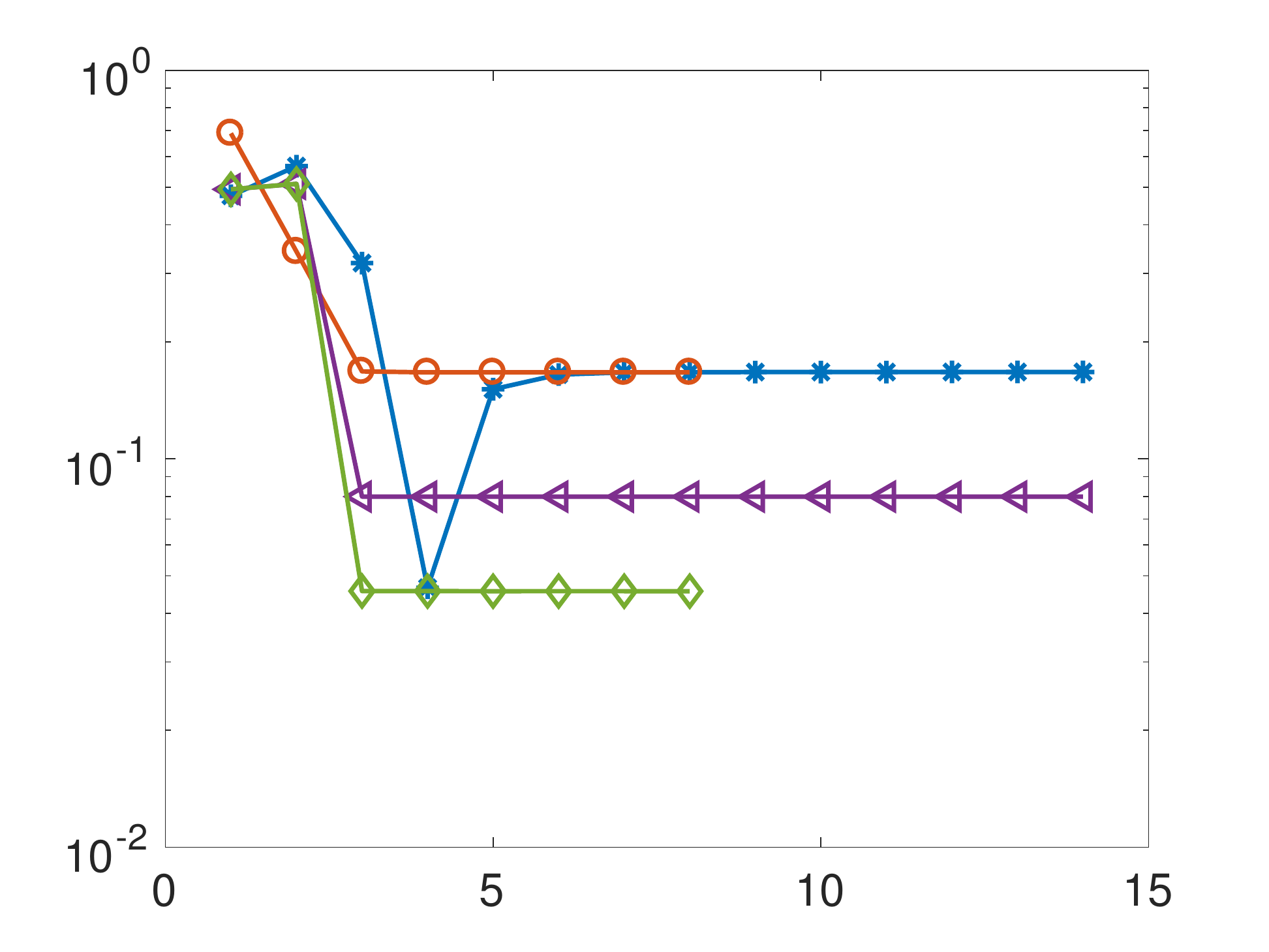} &
\hspace{-0.7cm}\includegraphics[width=6.2cm]{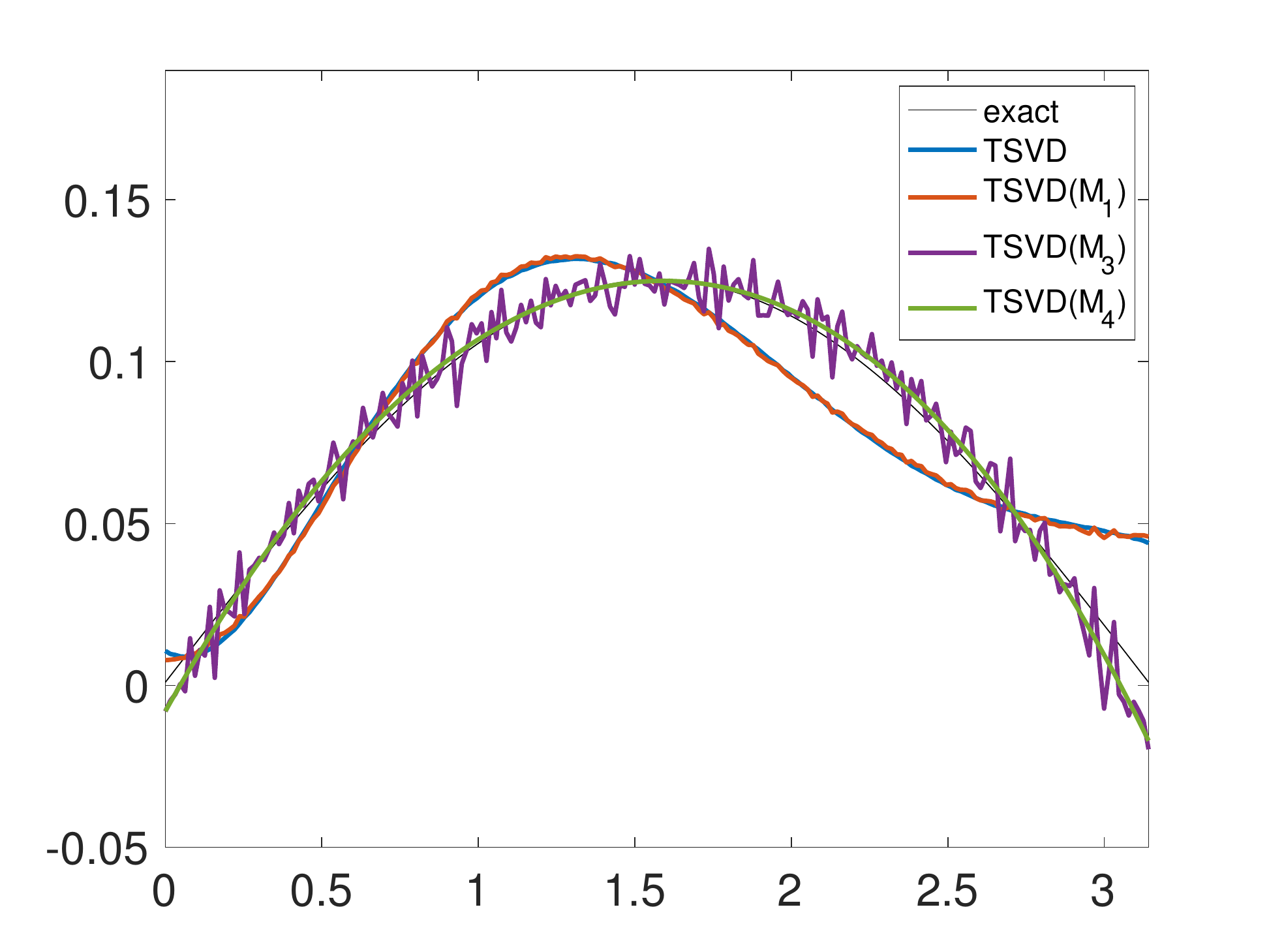}
\end{tabular}%
\caption{Test problem \texttt{baart}, with $m=200$ and $\|\be\|/\|\bb\| = 10^{-2}$. 
\textbf{(a)} Relative error history, without any additional regularization and $\kP = 9$. \textbf{(b)} Best
approximations, without any additional regularization and $\kP = 9$. \textbf{(c)} Relative error history, with 
Arnoldi--Tikhonov and $\kP = 9$. \textbf{(d)} Approximations for $k=6$, with Arnoldi--Tikhonov and $\kP = 9$. 
\textbf{(e)} Relative error history, with Arnoldi-TSVD and $\kP = 39$. \textbf{(f)} Approximations for $k=9$, 
with Arnoldi-TSVD and $\kP = 39$. }
\label{fig:baart}
\end{figure}

\texttt{\textbf{heat}}. We consider a discretization of the inverse heat equation 
formulated as a Volterra integral equation of the first kind; this problem can be 
regarded as numerically rank-deficient, with numerical rank equal to 195. According to the
analysis in \cite{JH}, GMRES does not converge to the minimum norm solution of 
(\ref{linsys}) for this problem, as the null spaces of $A$ and $A^T$ are different. For 
this test problem, considering the preconditioned methods described in Section \ref{sec5}, 
with some of the preconditioners derived in Section \ref{sec4}, can make a dramatic difference. 
When applying stopping rule (\ref{stop2}) to set the number of Arnoldi iterations defining the 
preconditioners, we get $\kP=23$; also for this problem, stopping rule (\ref{stop1}) prescribes 
a similar $\kP$. We report the behavior of relevant quantities used to set $\kP$ in the bottom 
frames of Figure \ref{fig:sc}. In the left frames of Figure \ref{fig:heat} we report the 
relative error history when different preconditioners (also defined with respect to different 
parameters $\kP$) and different solvers are considered. In all these graphs, the unpreconditioned 
Arnoldi--Tikhonov and Arnoldi-TSVD solutions diverge, with the best approximations being the ones 
recovered during the first iteration, i.e., the ones belonging to $\spn\{\bb\}$. The approximate 
solutions computed using the preconditioners (\ref{prec1})-(\ref{prec4}) with $\kP=23$ do not look 
much improved; indeed, while the ones obtained with $M_1$ and $M_2$ do not degenerate as quickly,
the ones obtained with $M_3$ and $M_4$ are even worse than the unpreconditioned ones. However, if
the maximum allowed value of $\kP$ (i.e., $\kP=60$) is chosen, the gain of using a preconditioned 
approach is evident. While the regularizing preconditioners $M_3$ and $M_4$ still perform very 
poorly, the preconditioners $M_1$ and $M_2$, which seek to make the matrix $AM$ Hermitian positive semidefinite 
by incorporating an approximate regularized version of $A^T$, allow us to reconstruct a solution, 
whose quality is close to the one achieved with LSQR. The right frames of Figure \ref{fig:heat} display 
the history of the corresponding relative residuals (or discrepancies). We can clearly see that the 
residuals are good indicators of the performance of these methods. Indeed, for $\kP=23$ all the residuals 
(except for the LSQR one) have quite large norm and, in particular, the discrepancy principle (\ref{discr}), 
(\ref{discr2}), (\ref{discr3}) is far from being satisfied. For $\kP = 60$, the preconditioned Arnoldi--Tikhonov 
method with the preconditioners $M_1$ or $M_2$ eventually satisfies the discrepancy principle. Also, the 
approximate solution obtained with $M_1$ reproduces the main features of of the exact solution, though some
spurious oscillations are present. This is probably due to the tiny value $\mu=1.2287\cdot 10^{-8}$ selected 
for the regularization parameter according to the discrepancy principle (\ref{discr2}); spurious oscillations 
are likely to be removed if a larger value for $\mu$ is used. The smallest relative errors attained by each 
iterative method (considering different choices of solvers and preconditioners) are reported in 
Table \ref{tab:expo}, where averages over 30 different realizations of the noise in the vector $\bb$ are 
shown.    

\begin{figure}[tbp]
\centering
\begin{tabular}{cc}
\hspace{-0.7cm}\textbf{{\small {(a)}}} & \hspace{-0.7cm}\textbf{{\small {(b)}}} \\
\hspace{-0.7cm}\includegraphics[width=6.2cm]{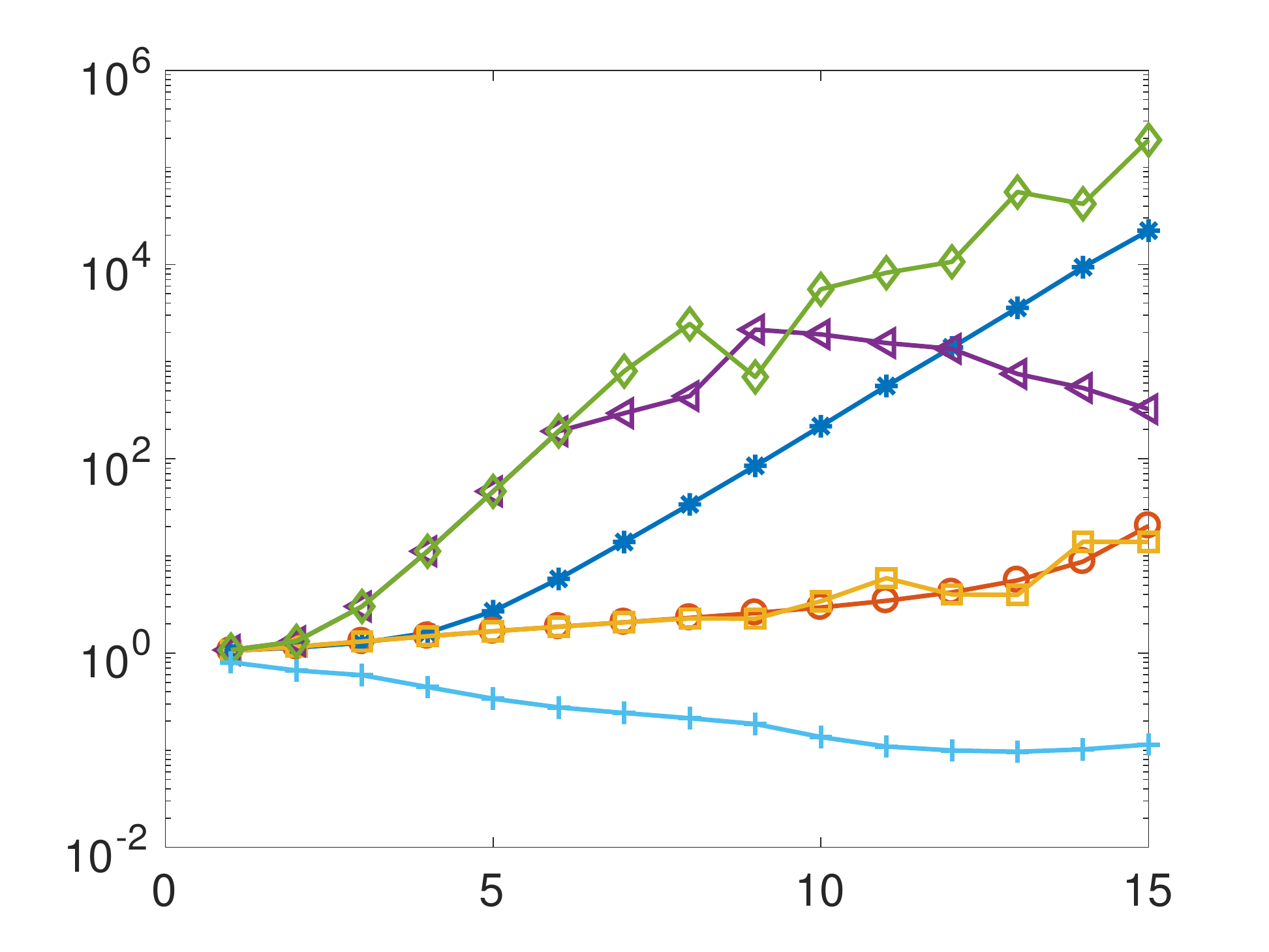} &
\hspace{-0.7cm}\includegraphics[width=6.2cm]{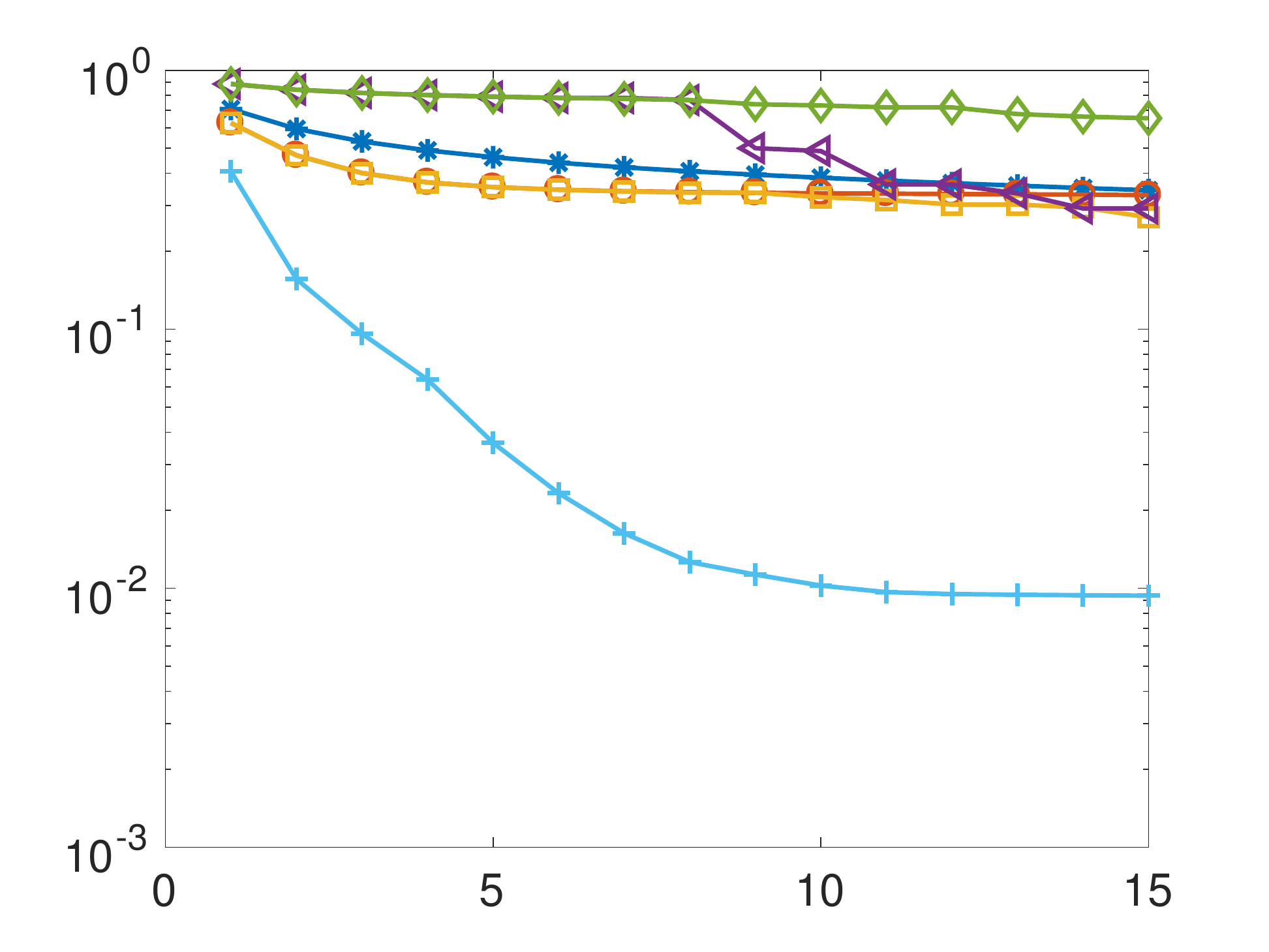}\\
\hspace{-0.7cm}\textbf{{\small {(c)}}} & \hspace{-0.7cm}\textbf{{\small {(d)}}} \\
\hspace{-0.7cm}\includegraphics[width=6.2cm]{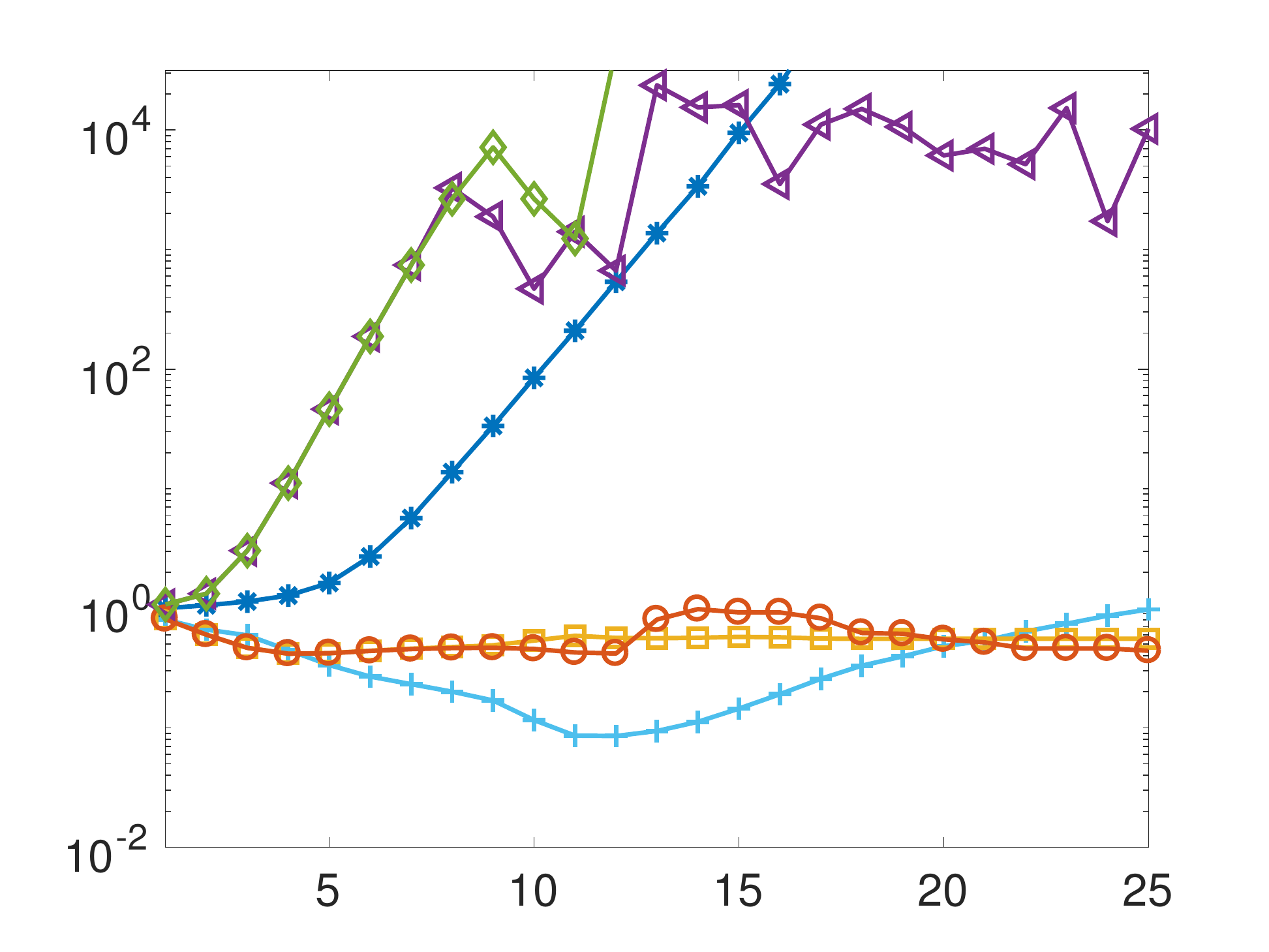} &
\hspace{-0.7cm}\includegraphics[width=6.2cm]{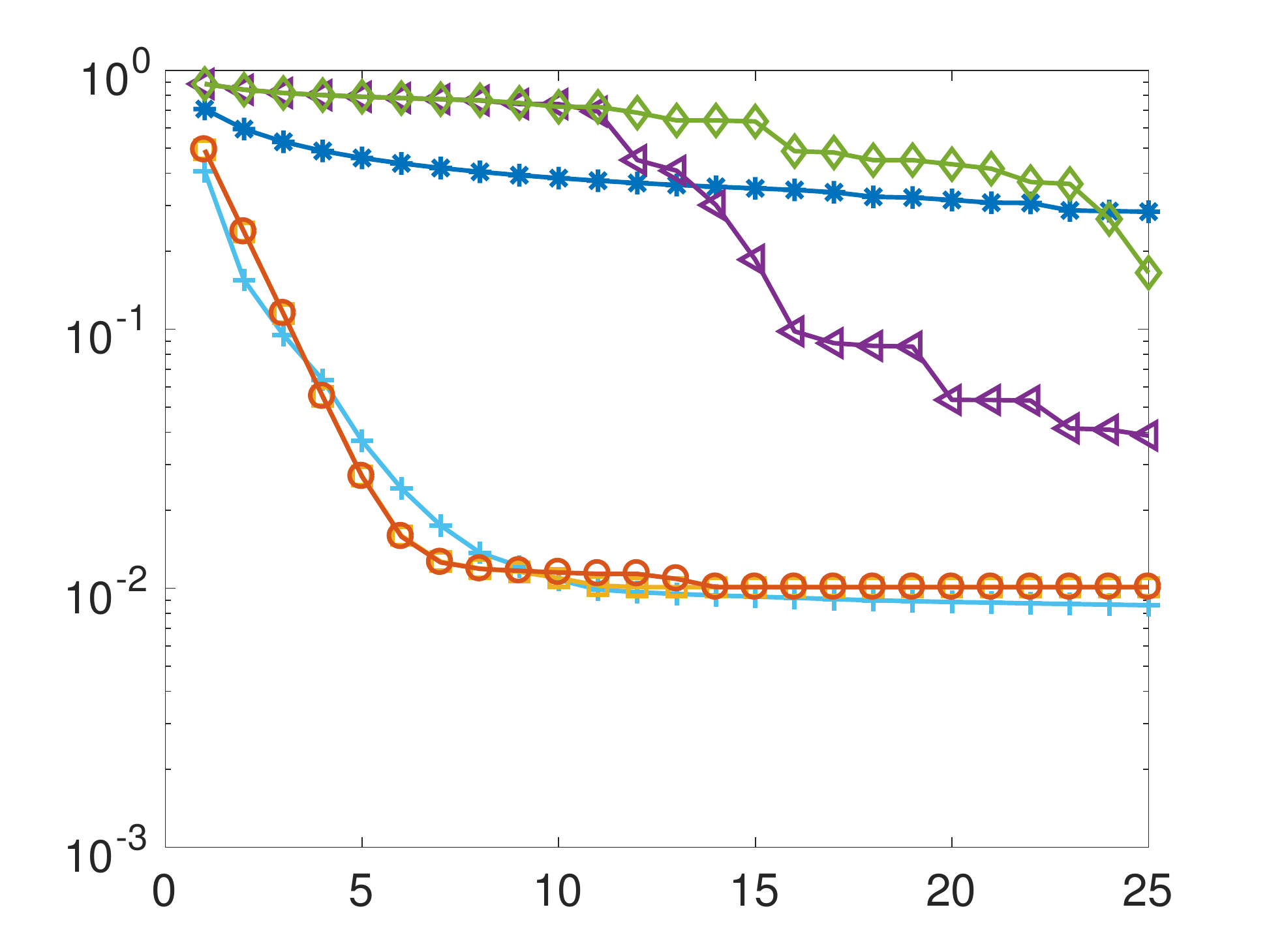}\\
\multicolumn{2}{c}{\textbf{{\small {(e)}}}}\\
\multicolumn{2}{c}{\includegraphics[width=8.2cm]{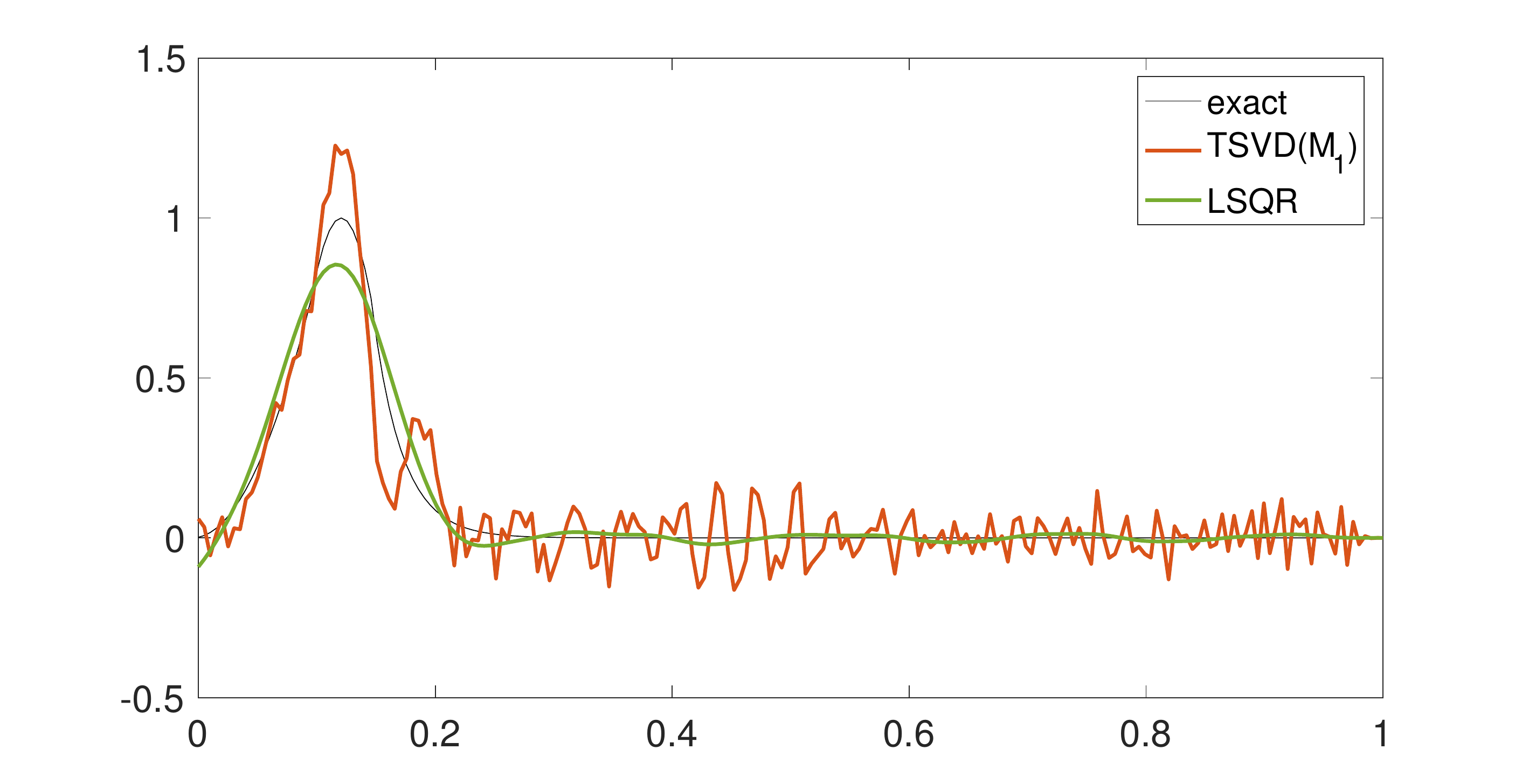}}
\end{tabular}%
\caption{Test problem \texttt{heat}, with $m=200$ and $\|\be\|/\|\bb\| = 10^{-2}$. 
\textbf{(a)} Relative error history, with Arnoldi-TSVD and $\kP = 23$. \textbf{(b)} Relative residual history, with Arnoldi-TSVD and $\kP = 23$. \textbf{(c)} Relative error history, with Arnoldi--Tikhonov and $\kP = 60$. \textbf{(d)} Relative residual history, with Arnoldi--Tikhonov and $\kP = 60$. \textbf{(e)} Best approximations, with Arnoldi--Tikhonov and $\kP = 60$.}
\label{fig:heat}
\end{figure}

\begin{figure}[tbp]
\centering
%\begin{tabular}{ccc}
%\hspace{-0.9 cm}\textbf{{\small {(a)}}} & \hspace{-0.9cm}\textbf{{\small {(b)}}} \\
%\hspace{-0.9cm}\includegraphics[width=5.0cm]{sc_baart_H} &
%\hspace{-0.9cm}\includegraphics[width=5.0cm]{sc_baart_Hstab} & 
%\hspace{-0.9cm}\includegraphics[width=5.0cm]{sc_baart_sv} \\
%\hspace{-0.9cm}\textbf{{\small {(c)}}} & \hspace{-0.9cm}\textbf{{\small {(d)}}} \\
%\hspace{-0.9cm}\includegraphics[width=5.0cm]{sc_heat_H} &
%\hspace{-0.9cm}\includegraphics[width=5.0cm]{sc_heat_Hstab} & 
%\hspace{-0.9cm}\includegraphics[width=5.0cm]{sc_heat_sv}
%\end{tabular}%
\begin{tabular}{cc}
\hspace{-0.7 cm}\textbf{{\small {\texttt{baart}, $h_{i+1,i}$}}} & \hspace{-0.7cm}\textbf{{\small {\texttt{baart}, $p_{\sigma}^{(i)}$}}} \\
\hspace{-0.7cm}\includegraphics[width=6.2cm]{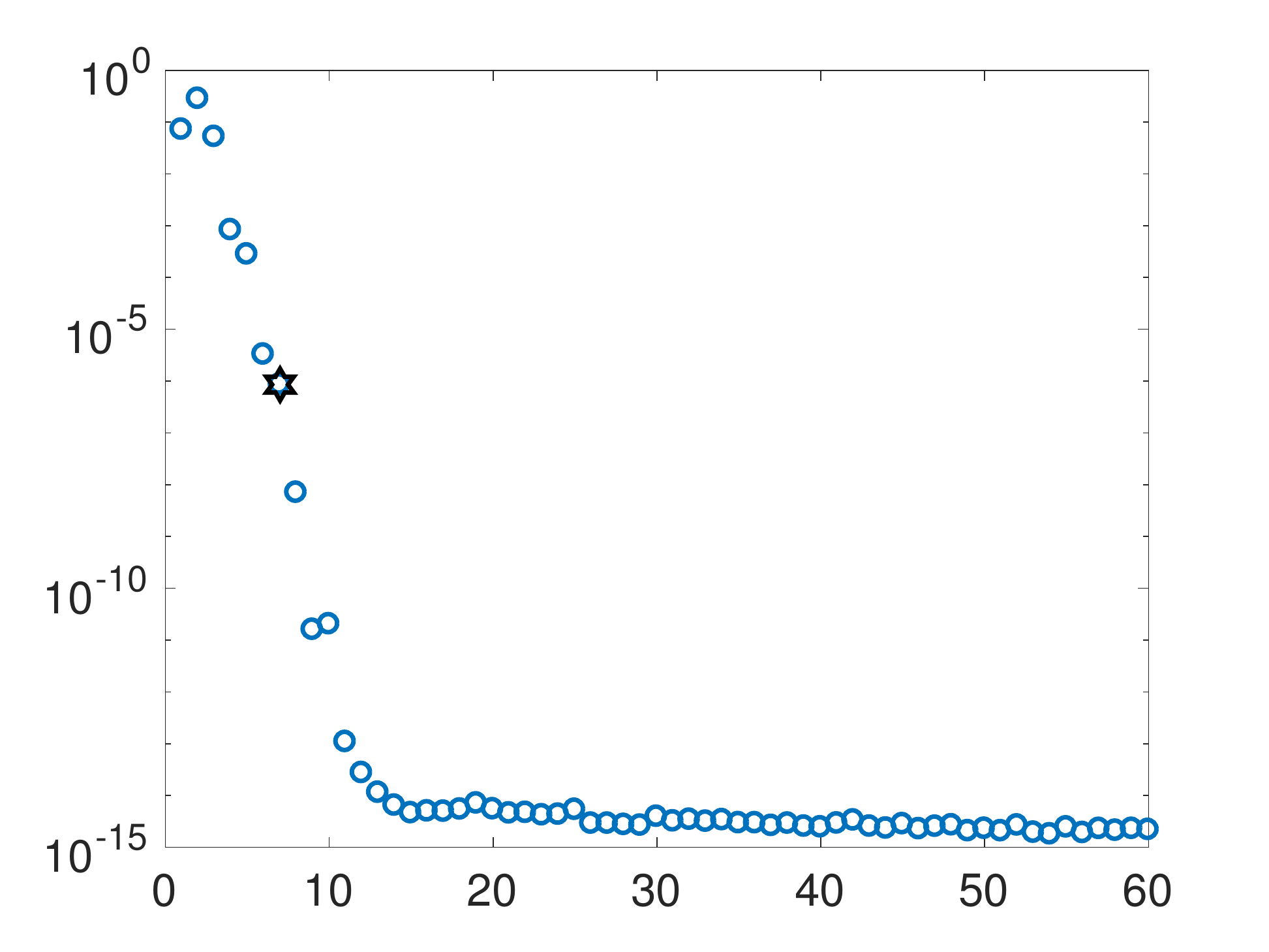} &
\hspace{-0.7cm}\includegraphics[width=6.2cm]{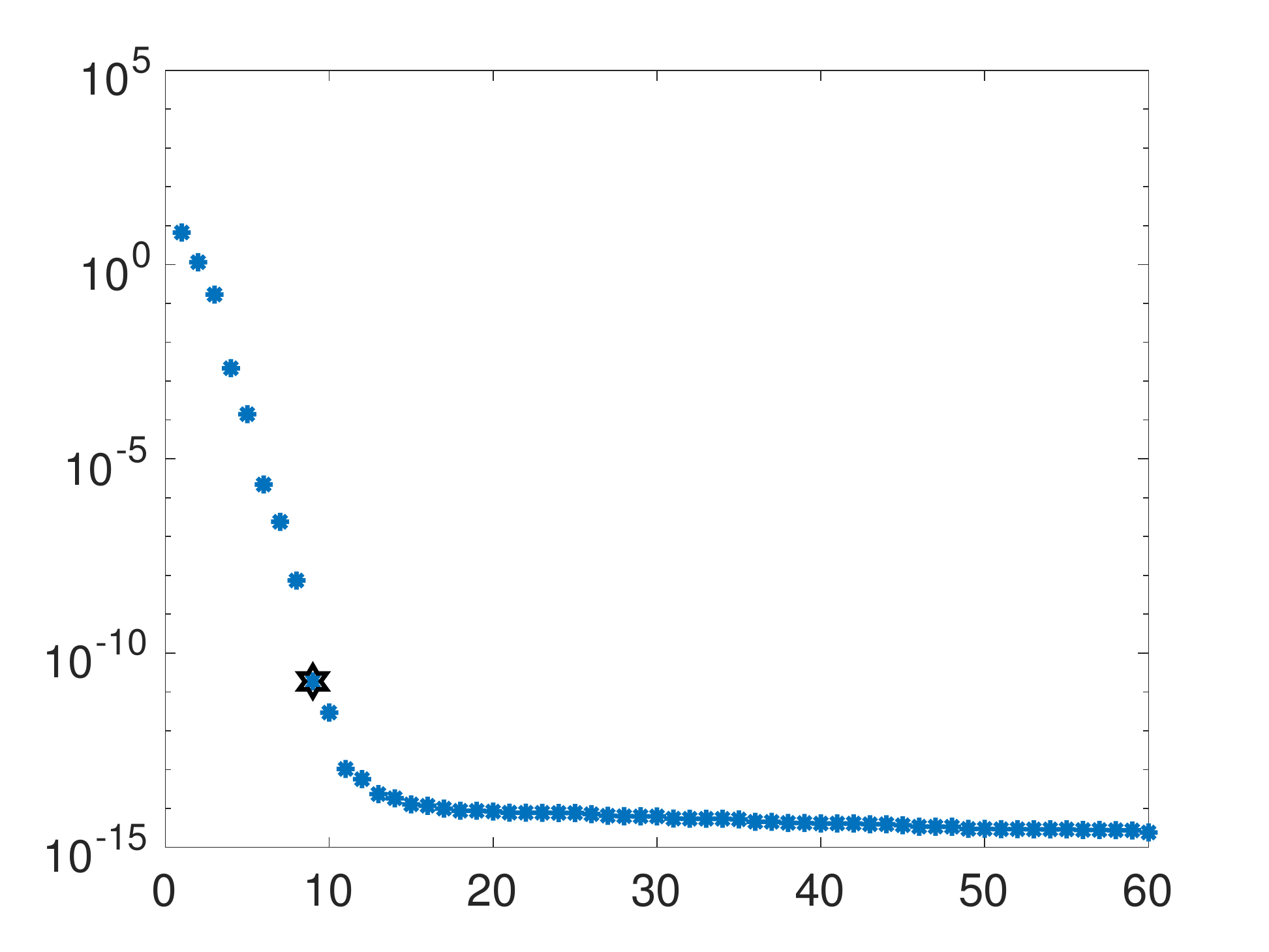} \\
\hspace{-0.7cm}\textbf{{\small {\texttt{heat}, $h_{i+1,i}$}}} & \hspace{-0.7cm}\textbf{{\small {\texttt{heat}, $p_{\sigma}^{(i)}$}}} \\
\hspace{-0.7cm}\includegraphics[width=6.2cm]{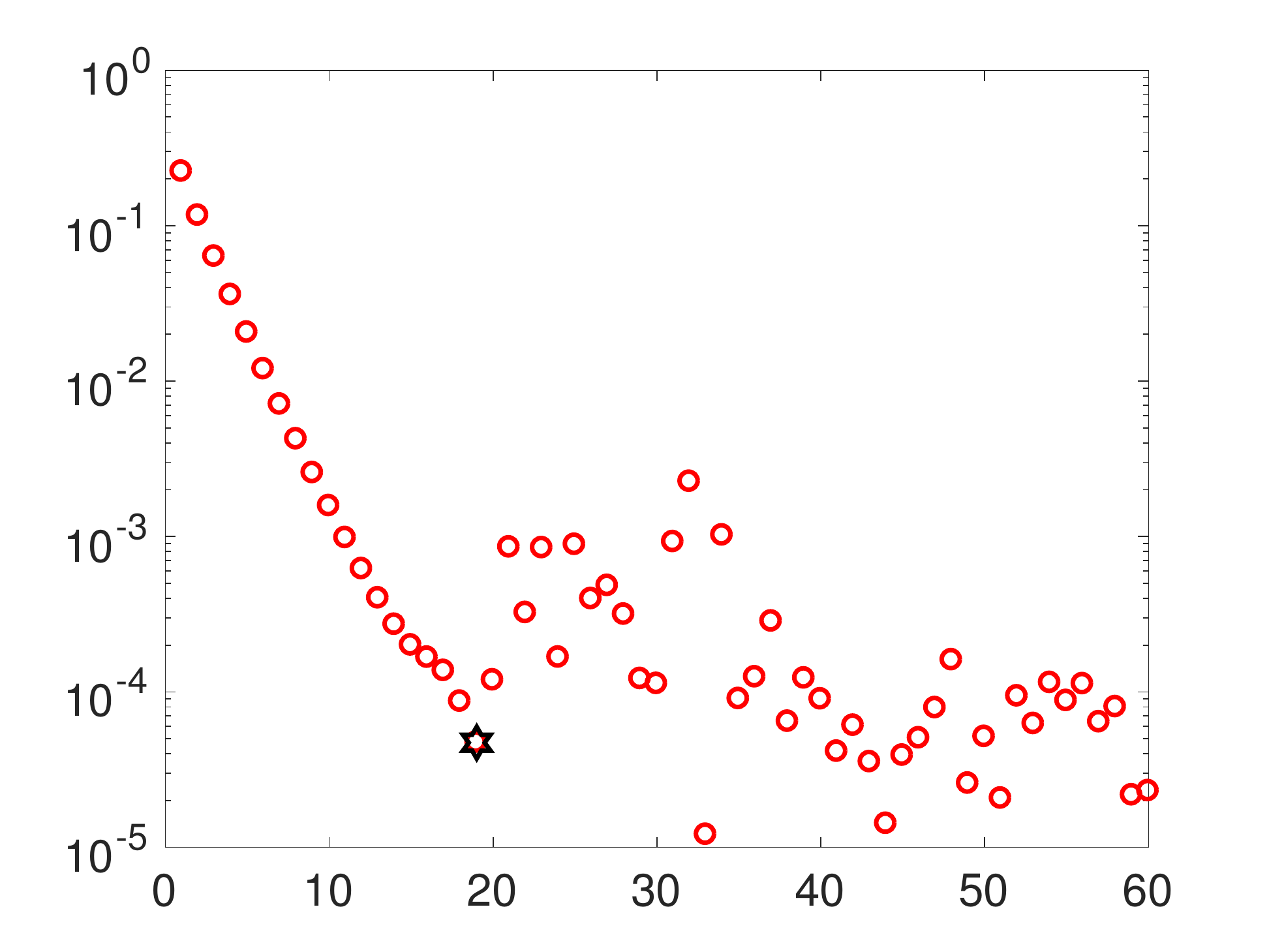} &
\hspace{-0.7cm}\includegraphics[width=6.2cm]{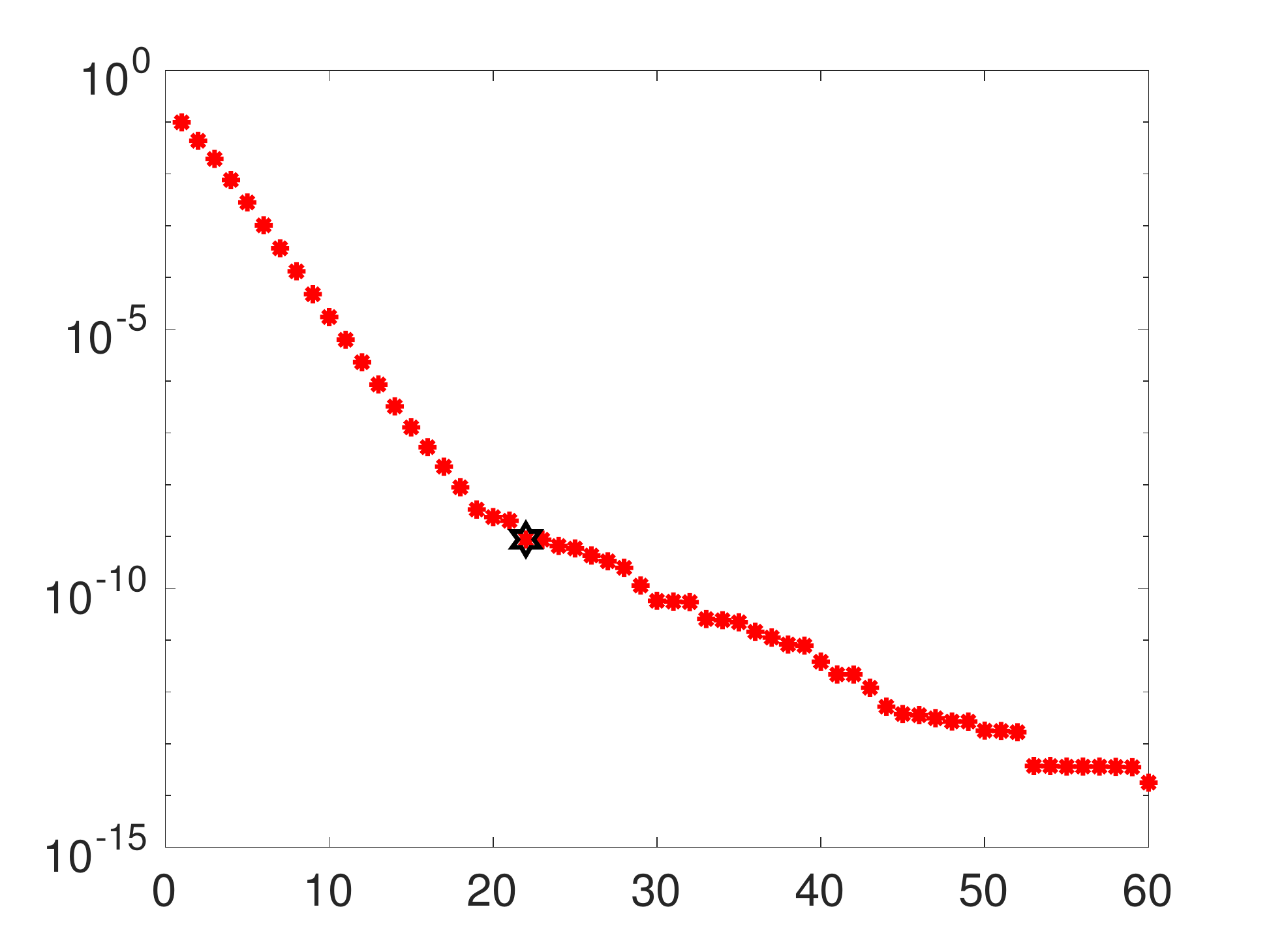}
\end{tabular}%
\caption{Illustration of the stopping criteria (\ref{stop1}) and (\ref{stop2}) for the test problems \texttt{baart} (top row) and \texttt{heat} (bottom row). In the left column, the values of the subdiagonal entries $h_{i+1,i}$ of the Hessenberg matrix $H_{k+1,k}$ are plotted against $i$ ($i=1,\dots,k$). In the right column, the products $p_{\sigma}^{(i)}:=\sigma_1^{(i)}\sigma_{i+1}^{(i+1)}$ of the extremal singular values of $H_{i+1,i}$ are plotted against $i$.}
\label{fig:sc}
\end{figure}

\begin{center}
\begin{table}
\caption{Average values of the best relative errors over 30 runs of the test problems in the first set of experiments, with $\|\be\|/\|\bb\|=10^{-2}$. The smaller parameter $\kP$ satisfies, on average, the stopping rule (\ref{stop2}); the larger parameter $\kP$ is obtained adding 30 to the smaller parameter $\kP$.}
\footnotesize
\label{tab:expo}   
\begin{center}
\begin{tabular}{|c|c|c|c|c|c|c|}\hline
&\multicolumn{6}{c|}{\texttt{baart}}\\\hline
&\multicolumn{2}{c|}{TSVD} &\multicolumn{2}{c|}{Tikh} &\multicolumn{2}{c|}{none} \\\hline
& $\kP=9$ & $\kP=39$ & $\kP=9$ & $\kP=39$& $\kP=9$ & $\kP=39$\\\hline
 -- & 4.7202e-02 & 4.7202e-02 & 6.7530e-02 & 6.7530e-02 & 3.0950e-01 & 3.0950e-01\\\hline
$M_1$ &   2.2148e-02  &1.6744e-01 & 2.4002e-02 & 1.7926e-01 & 1.8452e-02 & 1.5647e-01\\\hline
$M_2$ & 1.6689e-01 & 1.2429e-01 & 1.7733e-01 & 1.3091e-01 & 1.5838e-01 & 1.2517e-01\\\hline
$M_3$ &  4.5578e-02 & 6.1255e-02 & 6.6982e-02 & 6.7486e-02 & 4.5029e-02 & 6.1259e-02\\\hline
$M_4$ & 1.7025e-02 & 4.5678e-02 & 2.4297e-02 & 6.8386e-02 & 1.7027e-02 & 4.1604e-02\\\hline
LSQR & 1.5787e-01 & 1.5787e-01 & 1.5787e-01 & 1.5787e-01 & 1.5787e-01 & 1.5787e-01\\\hline
&\multicolumn{6}{c|}{\texttt{heat}}\\\hline
&\multicolumn{2}{c|}{TSVD} &\multicolumn{2}{c|}{Tikh} &\multicolumn{2}{c|}{none} \\\hline
& $\kP=20$ & $\kP=50$ & $\kP=20$ & $\kP=50$& $\kP=20$ & $\kP=50$\\\hline
-- & 6.5870e-01 & 6.5870e-01 & 5.6767e-01 & 5.6767e-01 & 1.0584e+00 & 1.0584e+00\\\hline
$M_1$ & 1.0296e+00 & 3.6071e-01 & 1.0296e+00 & 3.6173e-01 & 1.0296e+00 & 3.6136e-01\\\hline
$M_2$ & 1.0296e+00 & 3.6390e-01 & 1.0119e+00 & 3.0444e-01 & 1.0296e+00 & 3.6390e-01\\\hline
$M_3$ & 1.0747e+00 & 1.0747e+00 & 1.0747e+00 & 1.0747e+00 & 1.0747e+00 & 1.0747e+00\\\hline
$M_4$ & 1.0747e+00 & 1.0747e+00 & 1.0747e+00 & 1.0375e+00 & 1.0747e+00 & 1.0747e+00\\\hline
LSQR & 9.2105e-02 & 9.2105e-02 & 9.2105e-02 & 9.2105e-02 & 9.2105e-02 & 9.2105e-02\\\hline
\end{tabular}
\end{center}
\end{table}  
\end{center} 

\paragraph{Second set of experiments}
We consider 2D image restoration problems, where the available images are affected by a 
spatially invariant blur and Gaussian white noise. In this setting, given a point-spread 
function (PSF) that describes how a single pixel is deformed, a blurring process is modeled 
as a 2D convolution of the PSF and an exact discrete finite image 
$X_{\rm exact}\in\R^{n\times n}$. Here and in the following, a PSF is represented as a 2D 
image.
%One can immediately see that, if $P_{i,j}\neq 0$, $i,j=1,\dots,q$, the deblurring problem is underdetermined since, when convolving $P$ with $\Xex$,
%additional (and unavailable) values of the exact image outside $\Xex$ should be considered.
%A popular approach to overcome this phenomenon is to impose boundary conditions within the blurring process, i.e., to prescribe the behavior of the exact image outside $\Xex$ (see \cite{BerNagy13} and the references therein). 
A 2D image restoration problem can be expressed as a linear system of equations (\ref{linsys}), 
where the 1D array $\bb$ is obtained by stacking the columns of the 2D blurred and noisy image 
(so that $m=n^2$), and the square matrix $A$ incorporates the convolution process together with 
some given boundary conditions.
Our experiments consider two different gray scale test images, 
two different PSFs, and reflective boundary conditions; the sharp images are artificially 
blurred, and noise of several levels is added. Matrix-vector products are computed 
efficiently by using the routines in \emph{Restore Tools} \cite{RestTools}. The maximum 
allowed number of Arnoldi iterations in Algorithm \ref{alg:arnoldi} is $k=100$, and $\kP$ 
is set according to (\ref{stop2}). 

\texttt{Anisotropic motion blur}. For this experiment, a geometric test image of size $64\times 64$ pixels is taken as sharp image. It is displayed, together with a PSF modeling motion in two orthogonal directions, and the available corrupted data (with noise level $\|\be\|/\|\bb\|=2\cdot10^{-2}$), in the top frames of Figure \ref{fig:geo}. GMRES and right-preconditioned GMRES with preconditioners $C_1$, $C_2$, and $C_3$ are considered. The preconditioners $M_i$, $i=1,\dots,4$, do not perform well in this case, even when the maximum number of Arnoldi steps $\kP=100=k$ is performed; this is probably due to the fact that the PSF is quite unsymmetric. Figure \ref{fig:geo_relerr} displays the history of the relative reconstruction errors for these solvers. The most effective preconditioner for this problem is $C_2$. Moreover, both $C_1$ and $C_2$ require only a few iterations to compute an accurate solution and exhibit a quite stable behavior afterwards: for this reason we do not consider the Arnoldi--Tikhonov and Arnoldi-TSVD methods for this test problem. 
%the elements of the PSF $P=[p_{i,j}]$ 
%are given by 
%\[
%p_{i,j} = \exp\left(-\frac{1}{2(s_1^2s_2^2-\rho^4)}
%\left(s_2^2(i-k)^2-2\rho^2(i-k)(j-\ell)+s_1^2(j-\ell)^2\right)\right)\,,
%\]
%where $i,j=1,\dots,d$, and $[k,\ell]$ is the center of the PSF. The values $s_1=4$, 
%$s_2=1.3$, $\rho=2$, and $d=21$ are considered, and the noise level is 
%$\|\be\|/\|\bb\|=2\cdot10^{-2}$. The test data are displayed in Figure \ref{fig:camerset}. 
\begin{figure}[tbp]
\centering
\includegraphics[width=9.0cm]{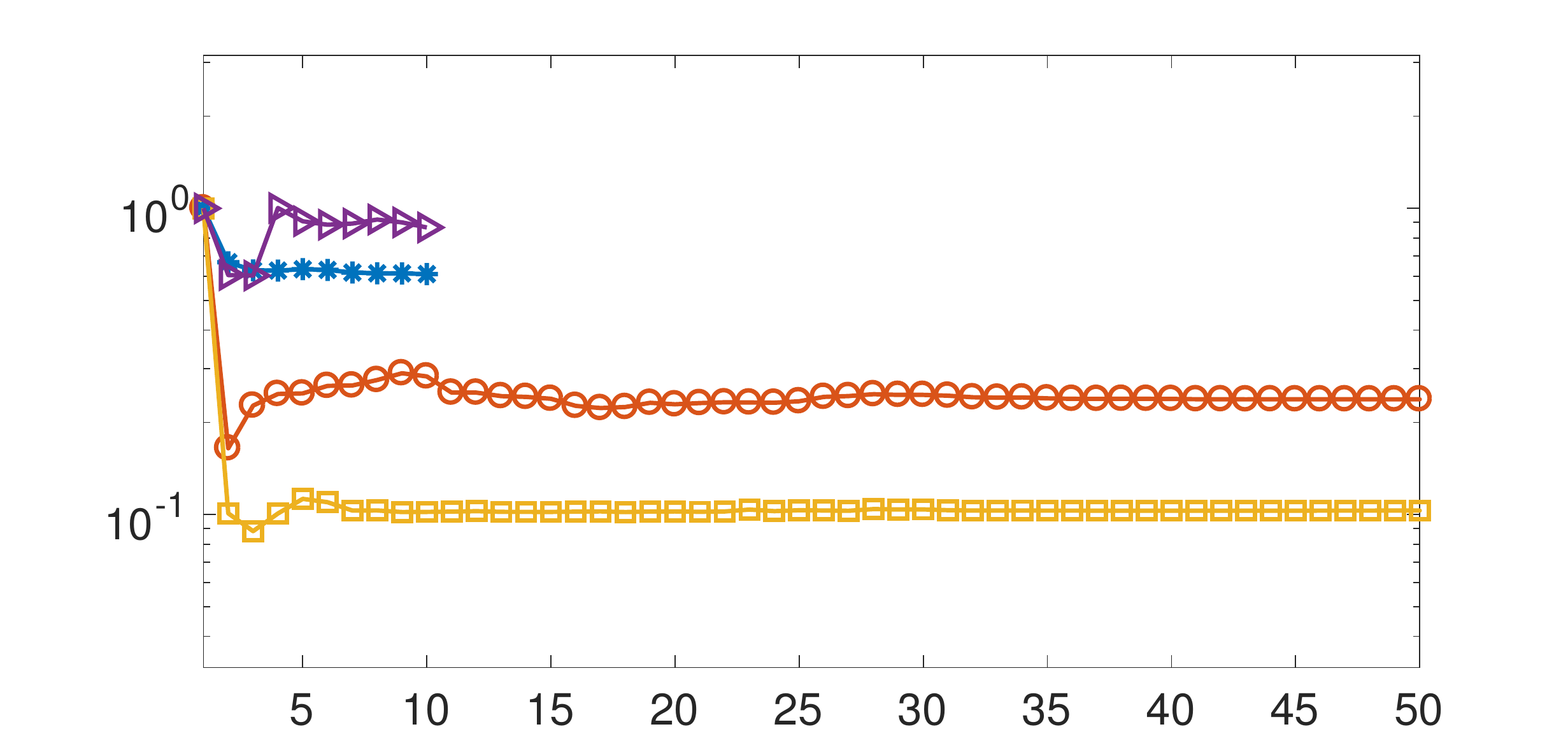}
\caption{Image deblurring problem with anisotropic motion blur. Relative error history of GMRES and right-preconditioned GMRES methods (with preconditioners $C_1$, $C_2$, and $C_3$). The GMRES and GMRES($C_3$) curves are truncated because severe ``semi-convergence'' occurs.}
\label{fig:geo_relerr}
\end{figure}
Figure \ref{fig:geo} shows the best restorations achieved by each method; relative 
errors and the corresponding number of iterations are displayed in the caption.
\begin{figure}[tbp]
\centering
\begin{tabular}{ccc}
\hspace{-1.2cm}\textbf{{\small {exact}}} &
\hspace{-1.2cm}\textbf{{\small {PSF}}} &
\hspace{-1.2cm}\textbf{{\small {corrupted}}} \\
\hspace{-1.2cm}\includegraphics[width=5.0cm]{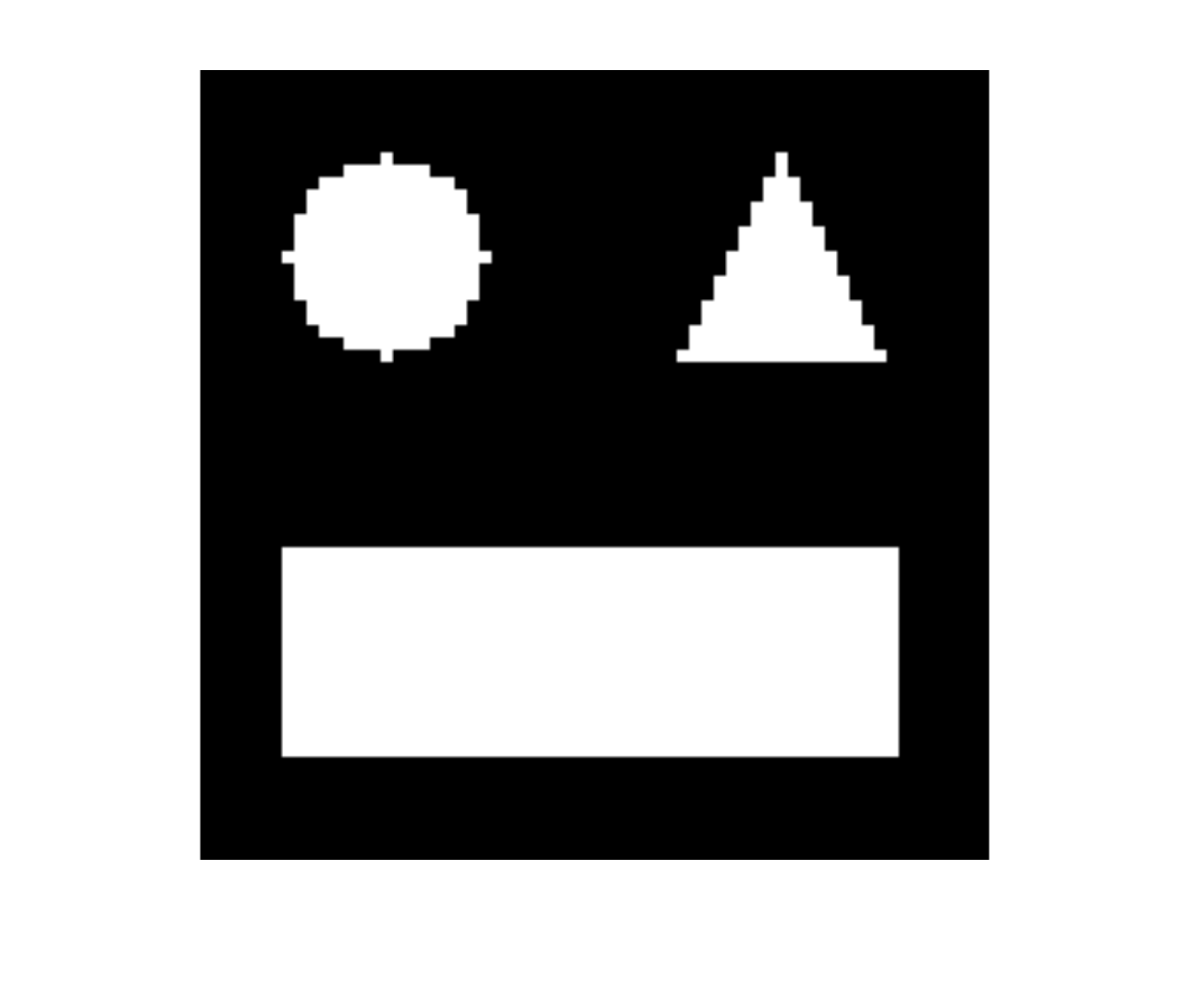} &
\hspace{-1.2cm}\includegraphics[width=5.0cm]{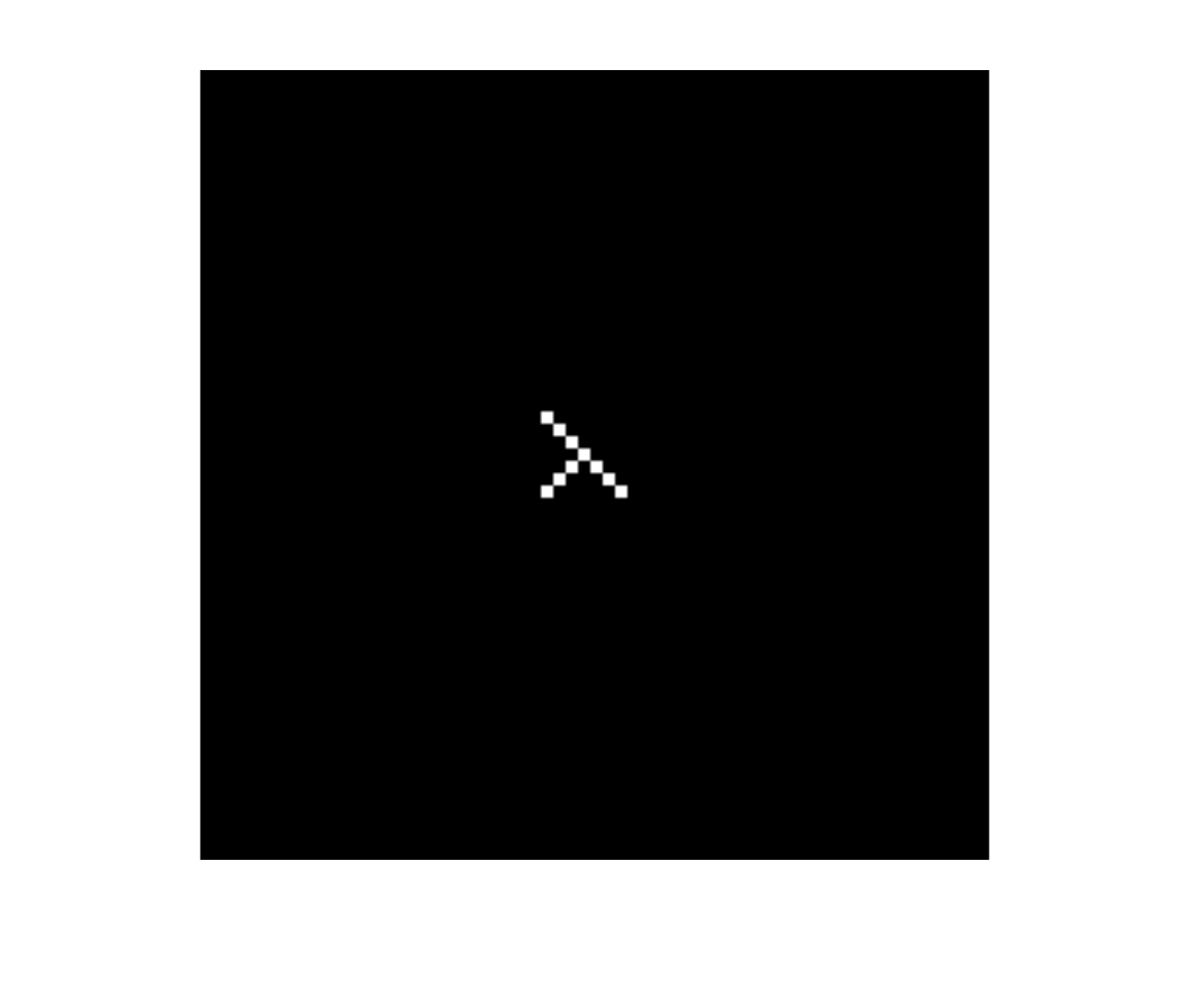} &
\hspace{-1.2cm}\includegraphics[width=5.0cm]{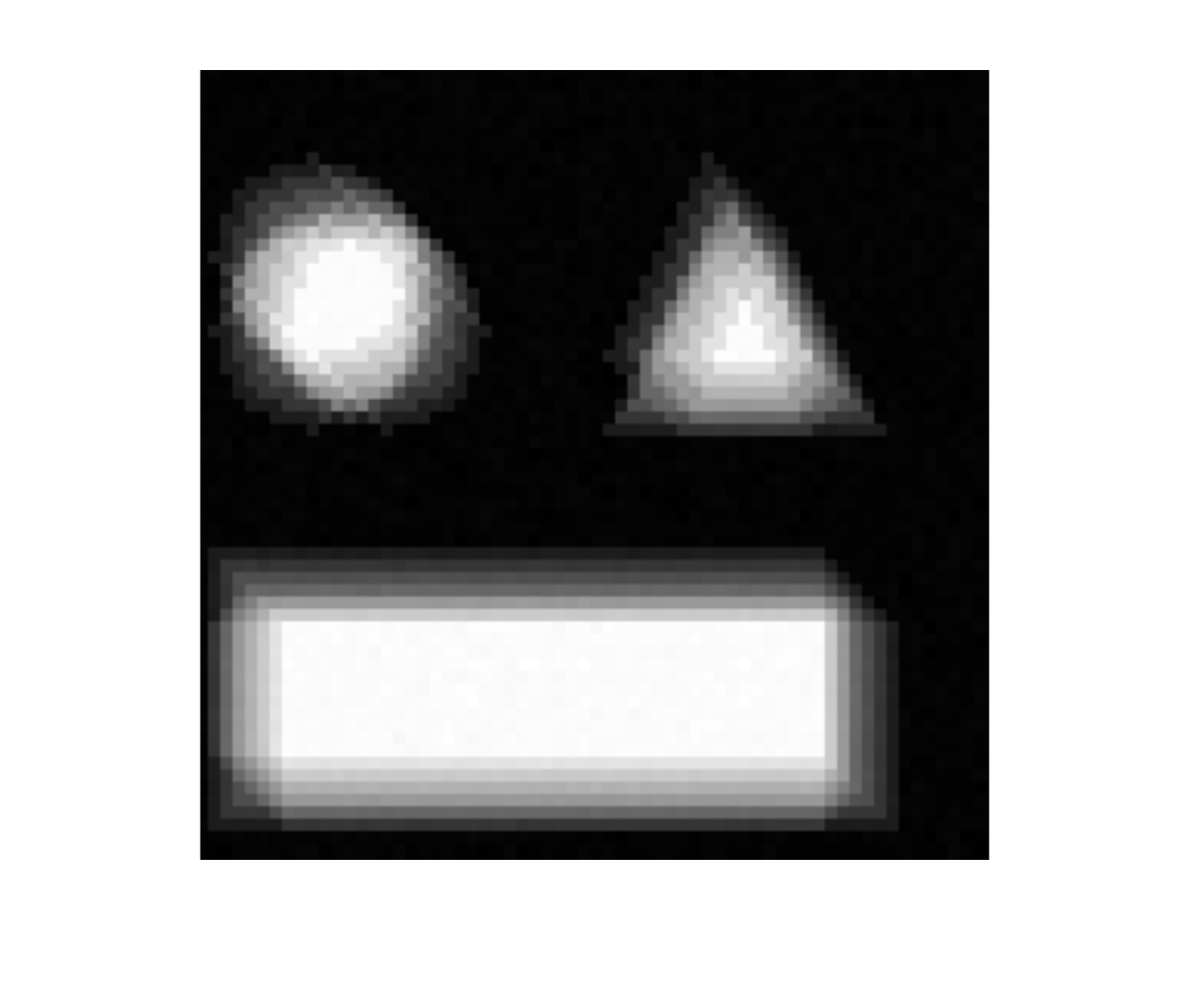}\vspace{-0.2cm}\\
\hspace{-1.2cm}\textbf{{\small {GMRES}}} &
\hspace{-1.2cm}\textbf{{\small {GMRES($C_1$)}}} &
\hspace{-1.2cm}\textbf{{\small {GMRES($C_2$)}}} \\
\hspace{-1.2cm}\includegraphics[width=5.0cm]{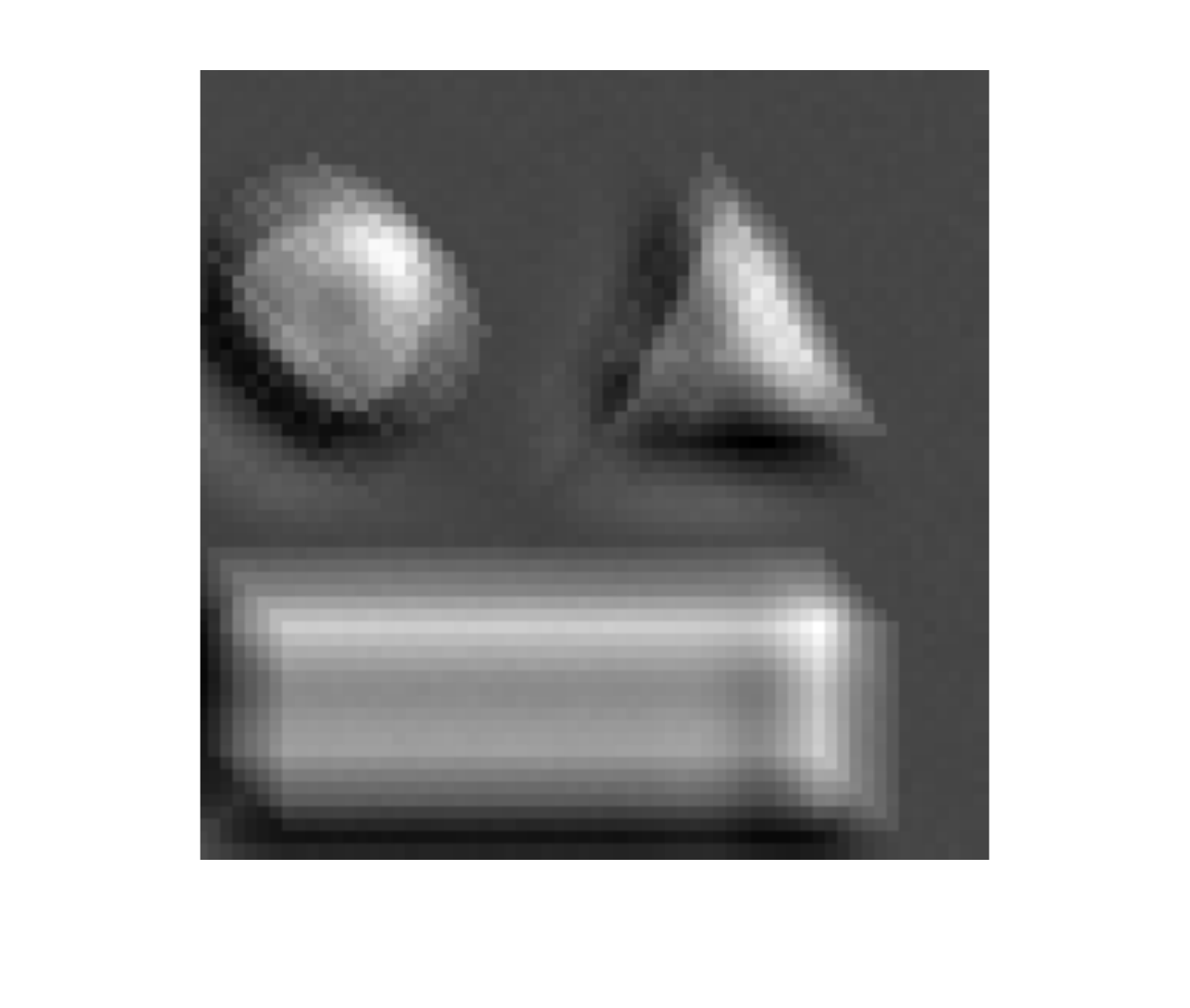} &
\hspace{-1.2cm}\includegraphics[width=5.0cm]{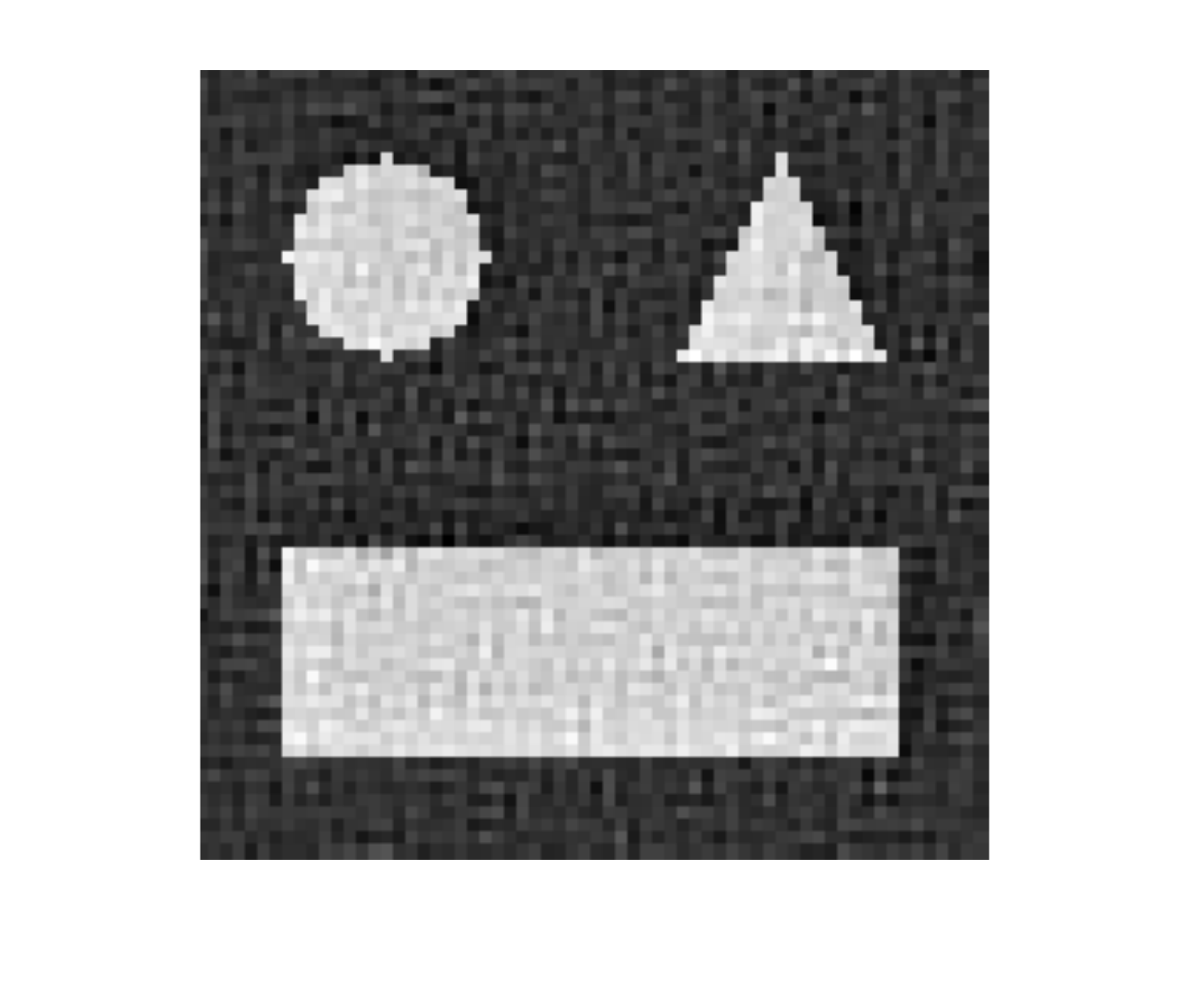} &
\hspace{-1.2cm}\includegraphics[width=5.0cm]{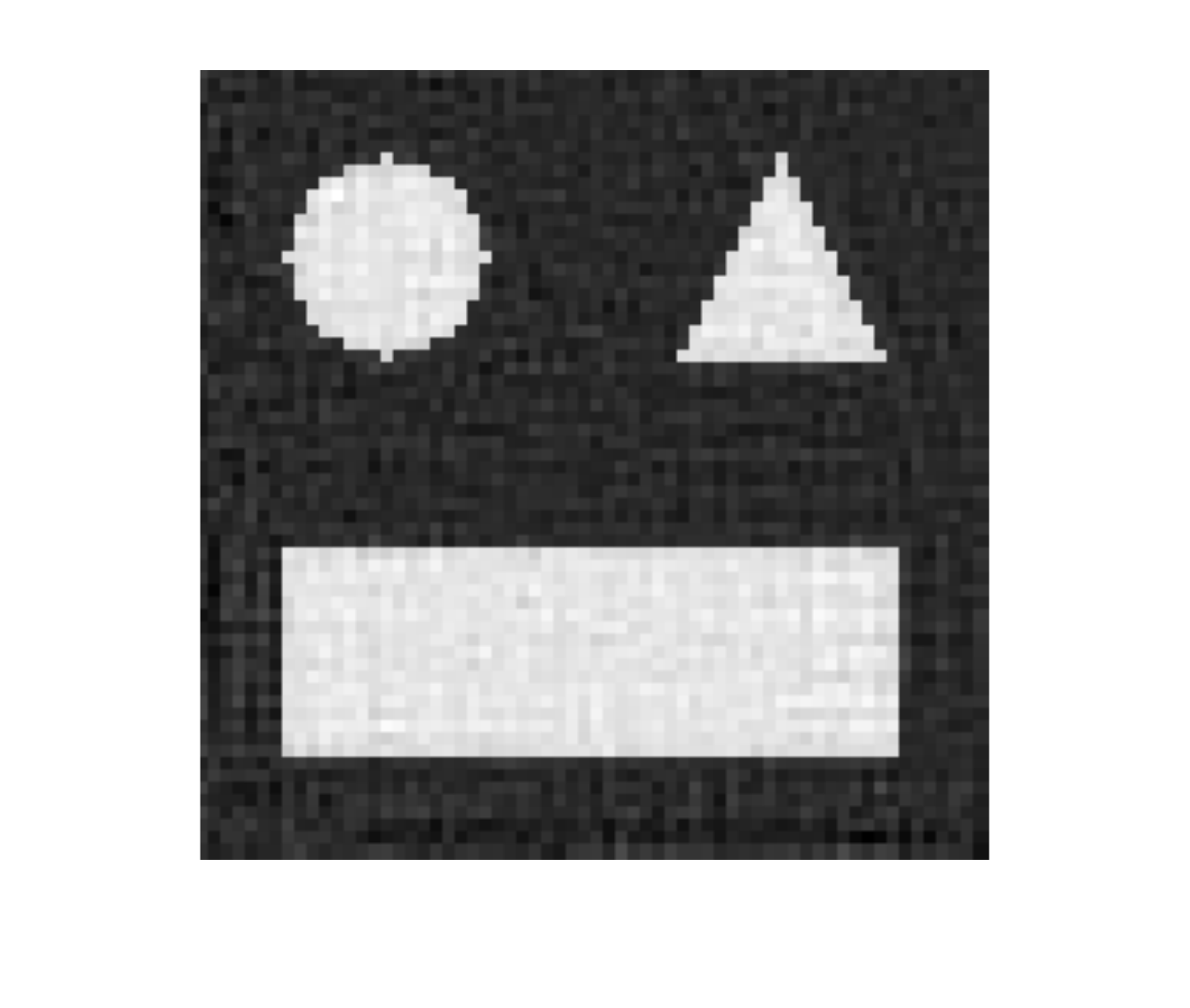}
\end{tabular}
\caption{Image deblurring problem with anisotropic motion blur. The upper row displays the test data. The lower row do spays the best reconstructions obtained by: GMRES method ($6.0804e-01$, $k=13$); GMRES($C_1$) method ($1.6452e-01$, $k=2$); GMRES($C_2$) method ($8.8234e-02$, $k=3$); GMRES($C_3$) method ($6.0316e-01$, $k=3$).}
\label{fig:geo}
\end{figure}

\texttt{Isotropic motion blur.} The test data for this experiment are displayed in Figure \ref{fig:boatset}. 
We consider a $17\times 17$ PSF modeling diagonal motion blur. The noise level is 
$5\cdot 10^{-3}$. Figure \ref{fig:boatrec} shows the best restorations achieved by
each method; relative errors and the corresponding number of iterations are displayed in 
the caption.
\begin{figure}[tbp]
\centering
\begin{tabular}{ccc}
\hspace{-1.2cm}\textbf{{\small {exact}}} &
\hspace{-1.2cm}\textbf{{\small {PSF}}} &
\hspace{-1.2cm}\textbf{{\small {corrupted}}} \\
\hspace{-1.2cm}\includegraphics[width=5.0cm]{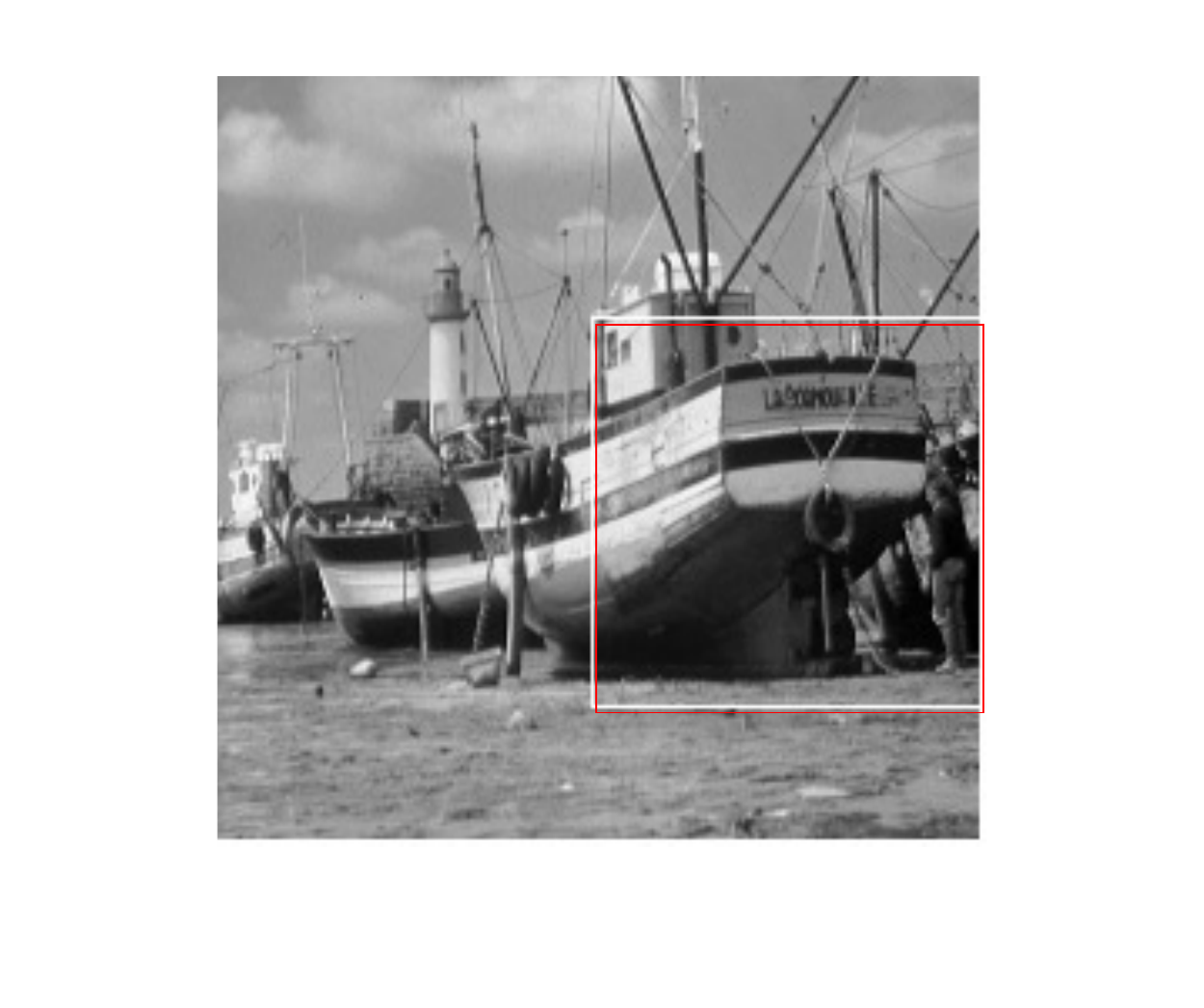} &
\hspace{-1.2cm}\includegraphics[width=5.0cm]{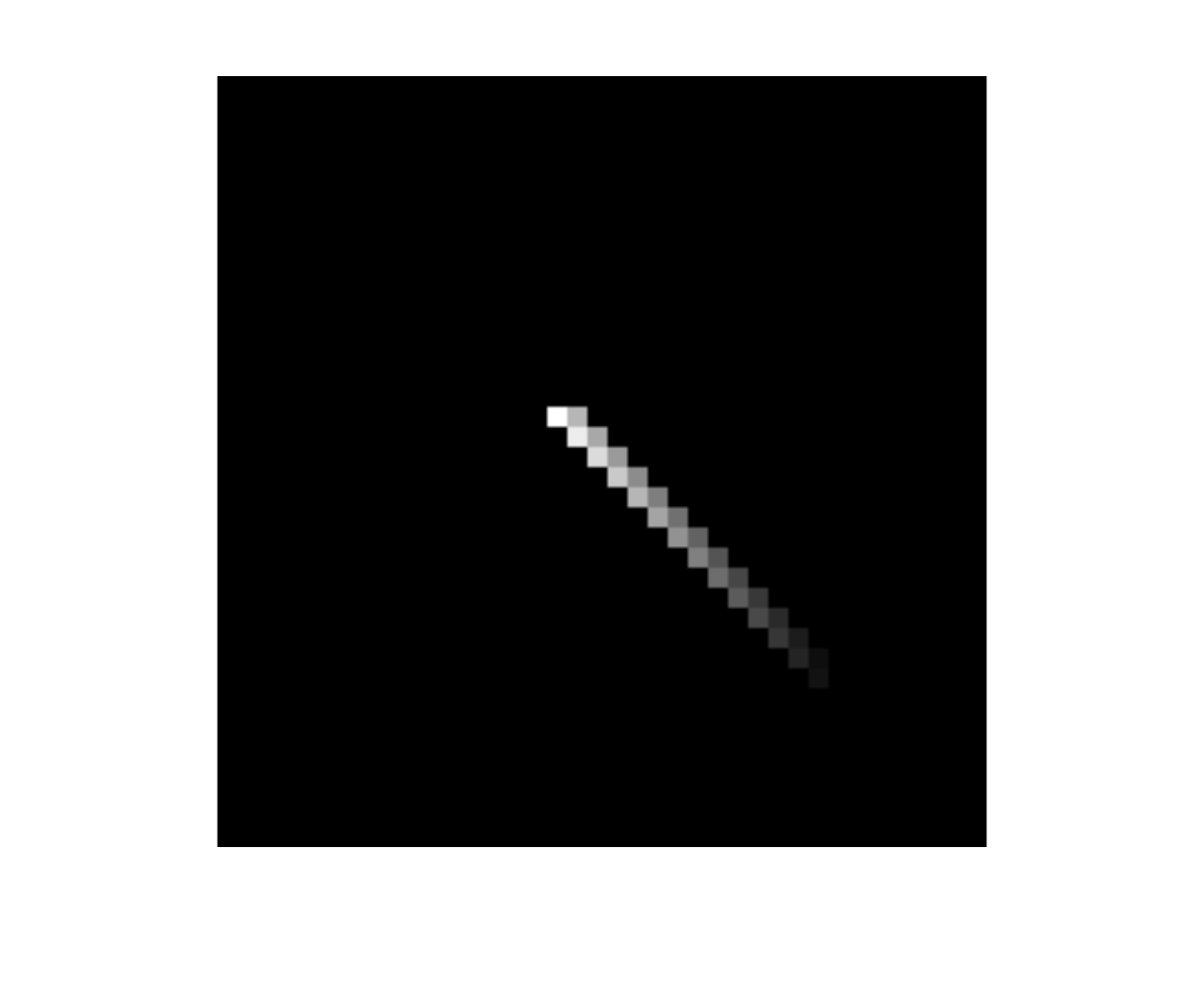} &
\hspace{-1.2cm}\includegraphics[width=5.0cm]{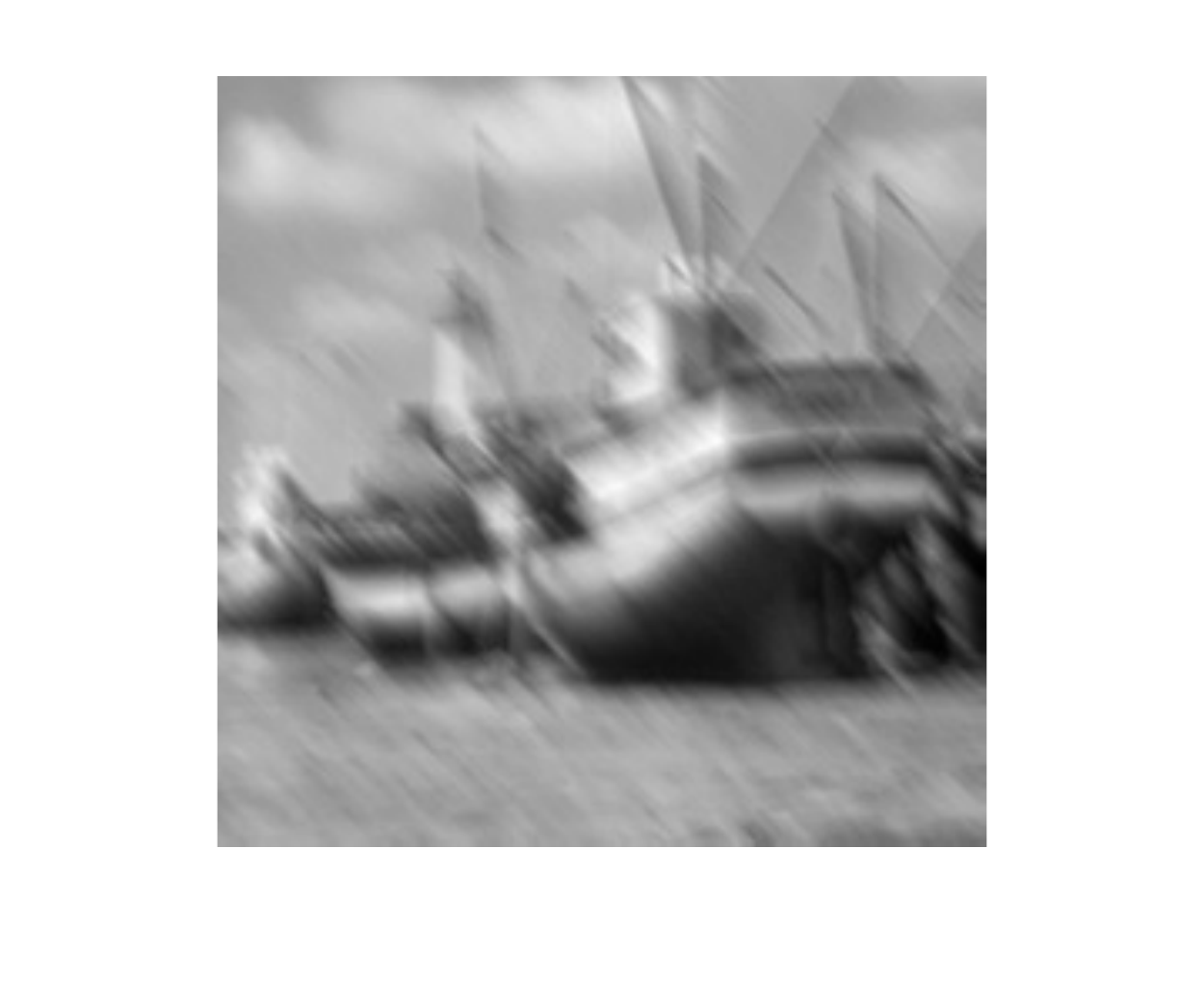}
\end{tabular}%
\caption{From left to right: exact image; blow-up ($600\%$) of the diagonal motion PSF; blurred and noisy available image, with $\|\be\|/\|\bb\|=5\cdot 10^{-3}$.}
\label{fig:boatset}
\end{figure}
\begin{figure}[tbp]
\centering
\begin{tabular}{ccc}
\hspace{-1.2cm}\textbf{{\small {TSVD}}} &
\hspace{-1.2cm}\textbf{{\small {TSVD($M_1$)}}} &
\hspace{-1.2cm}\textbf{{\small {TSVD($M_3$)}}} \\
\hspace{-1.2cm}\includegraphics[width=5.0cm]{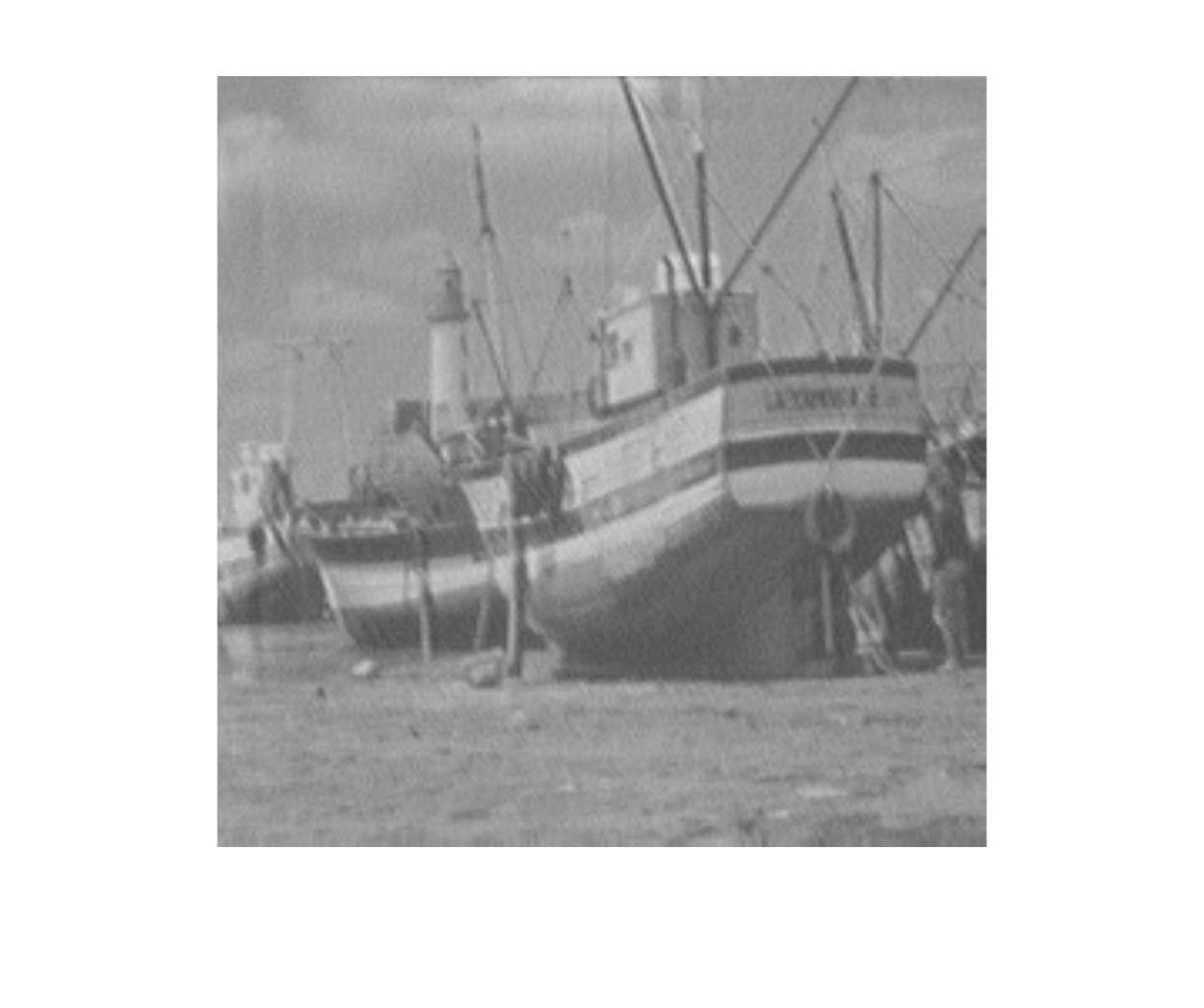} &
\hspace{-1.2cm}\includegraphics[width=5.0cm]{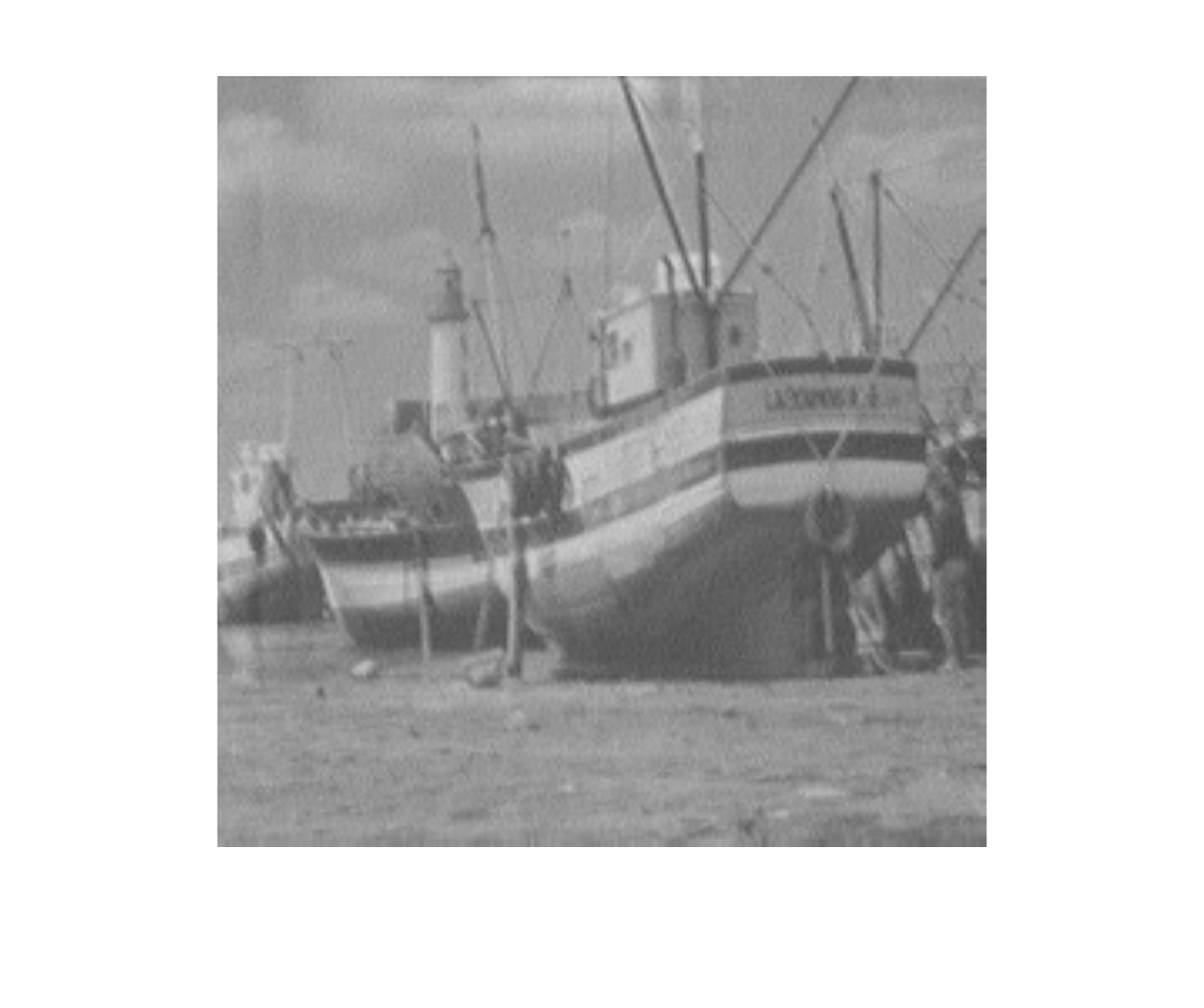} &
\hspace{-1.2cm}\includegraphics[width=5.0cm]{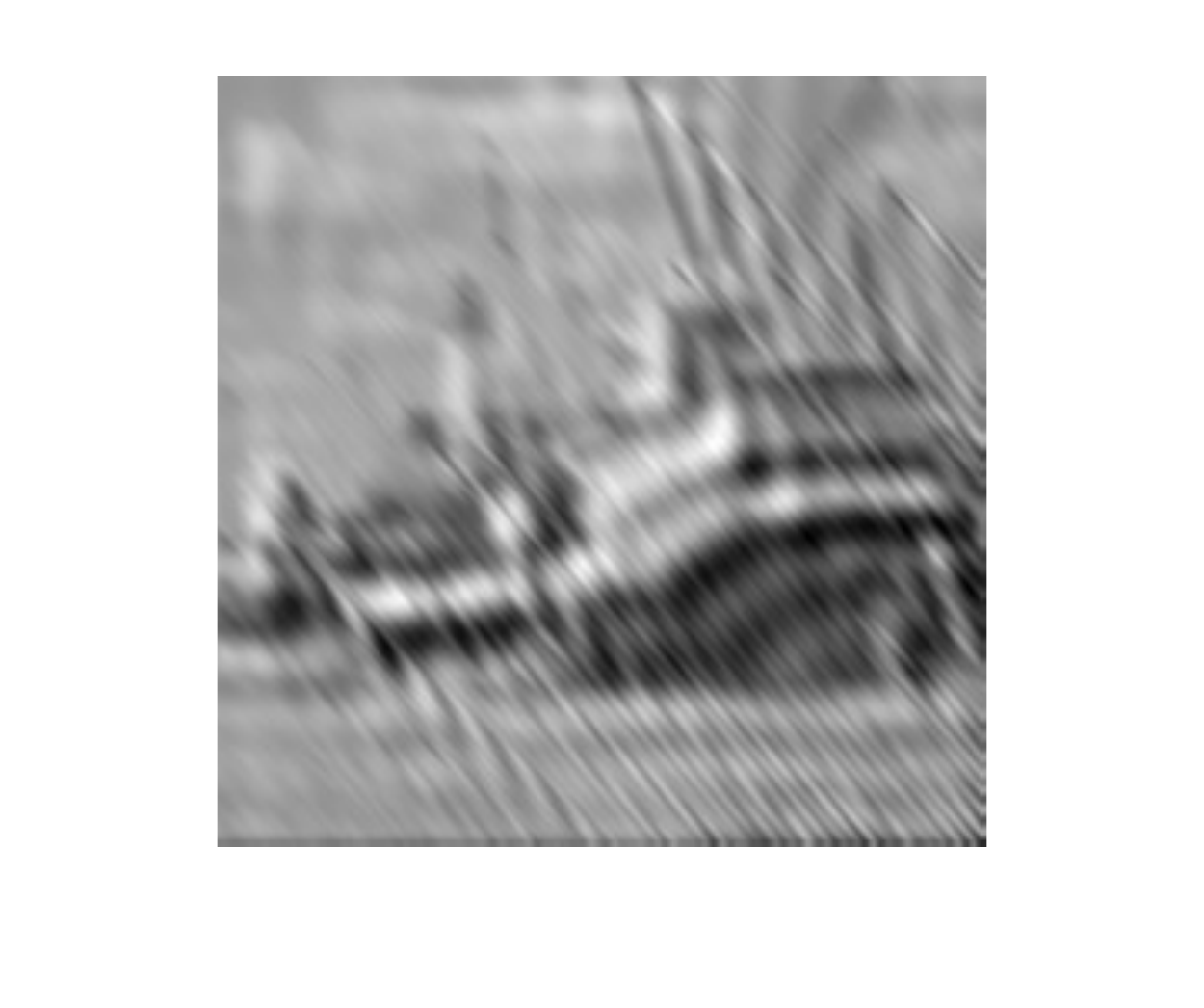}\vspace{-0.5cm}\\
\hspace{-1.2cm}\includegraphics[width=5.0cm]{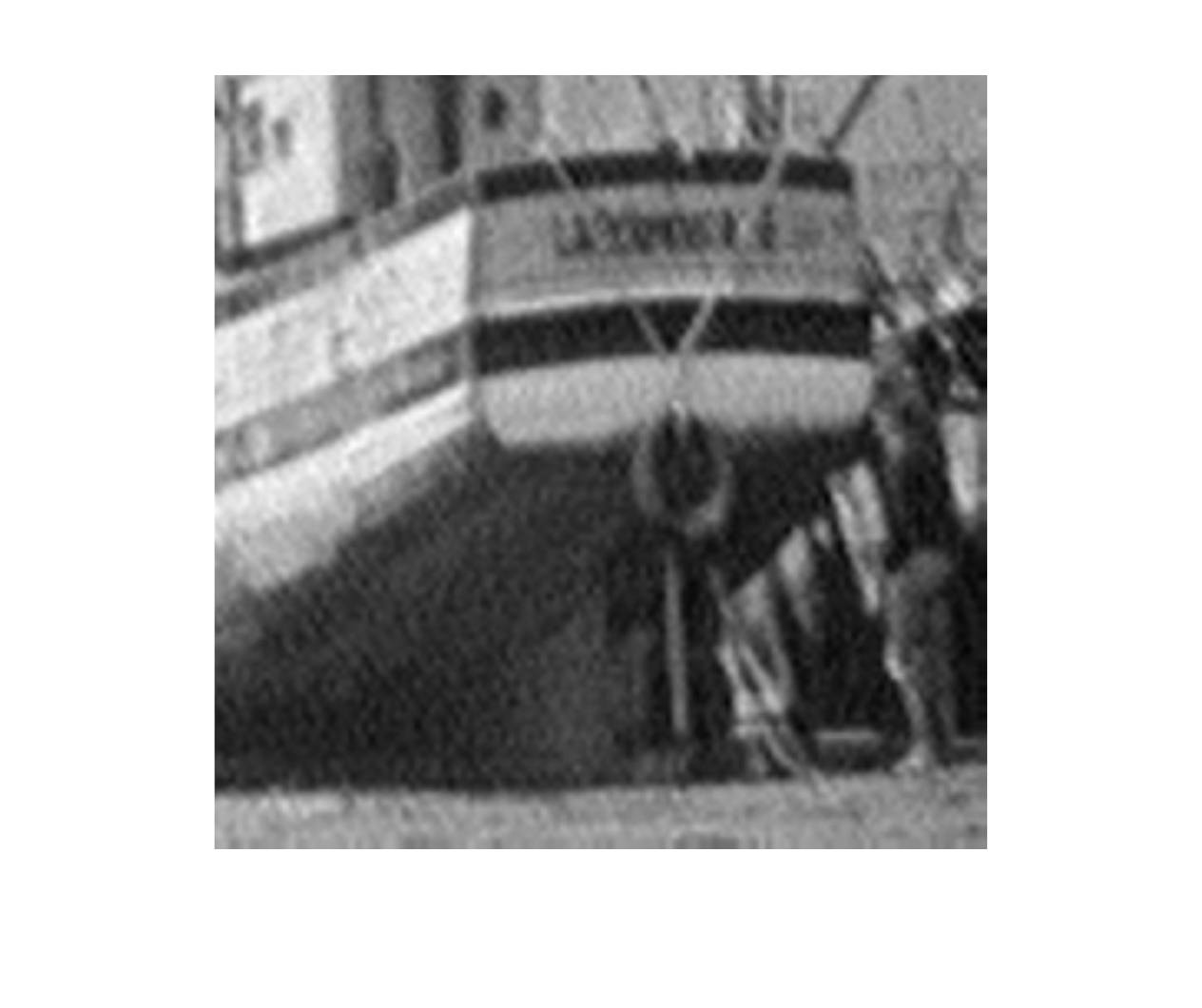} &
\hspace{-1.2cm}\includegraphics[width=5.0cm]{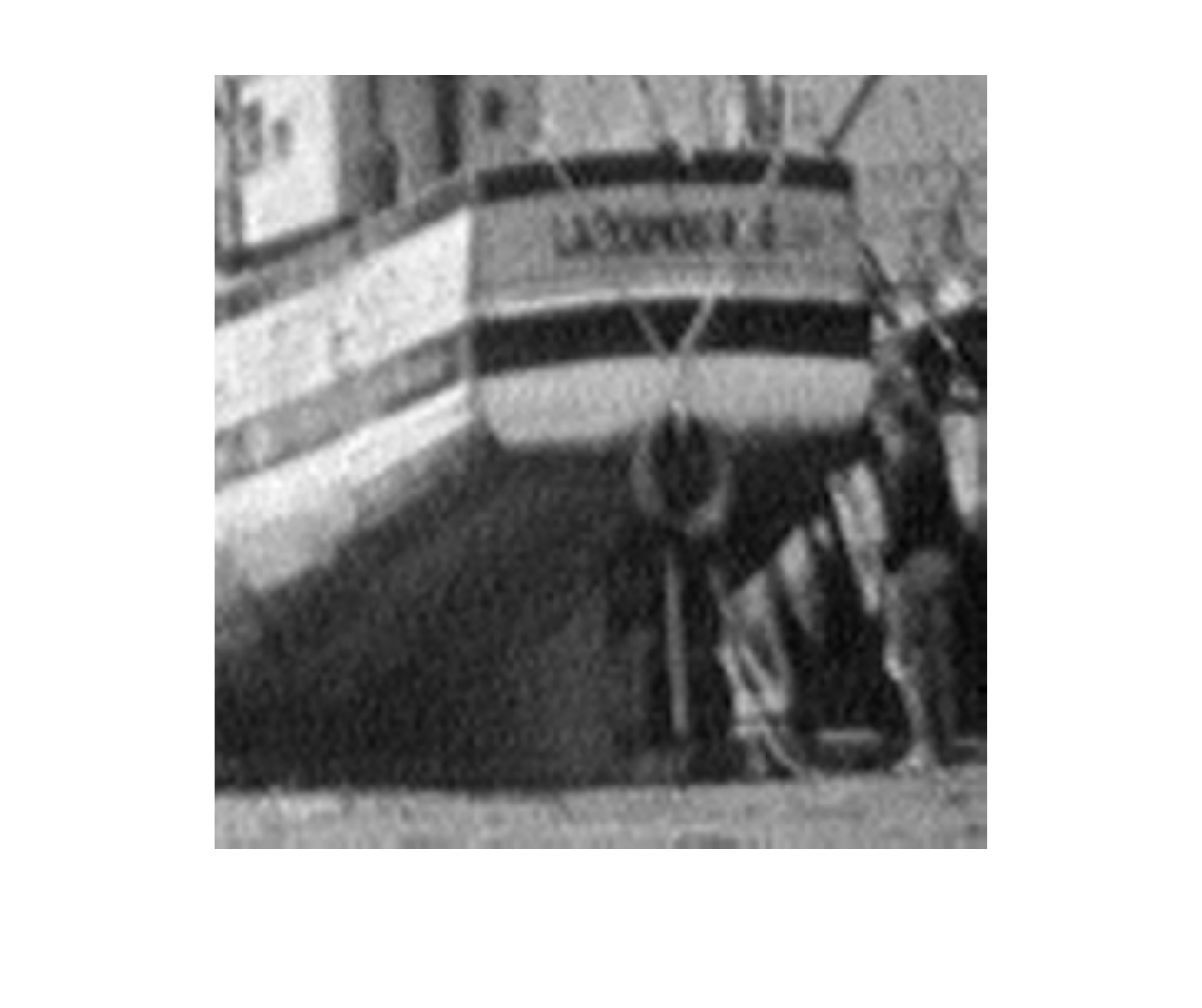} &
\hspace{-1.2cm}\includegraphics[width=5.0cm]{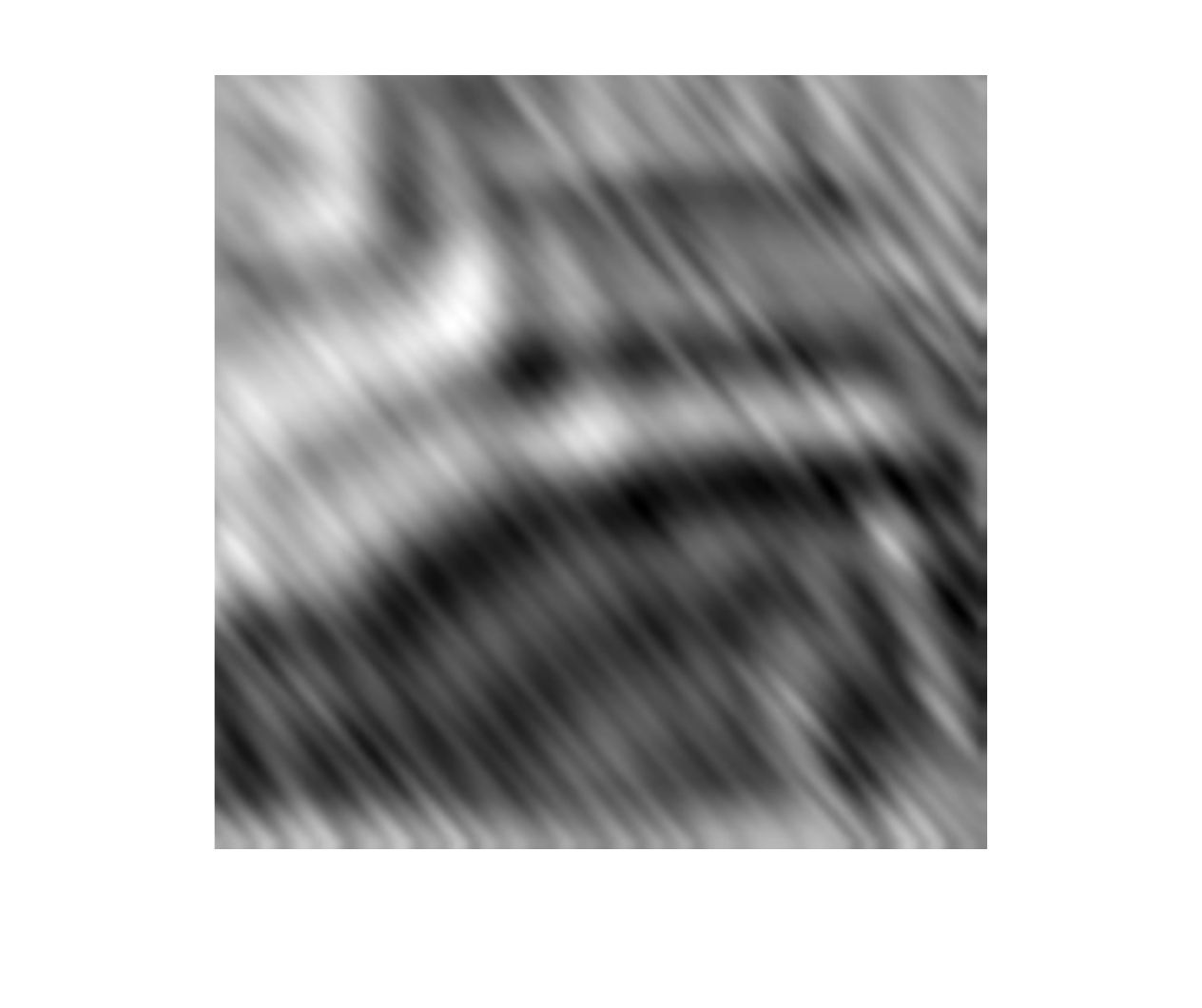}
\end{tabular}
\caption{The lower row displays blow-ups ($200\%$) of the restored images in the upper row. From left to right: unpreconditioned  Arnoldi-TSVD method ($1.0481e-01$, $k=26$); TSVD($M_1$) method ($1.0081e-01$, $\kP=50$, $k=7$); TSVD($M_3$) method ($2.5948e-01$, $\kP=50$, $k=35$).}
\label{fig:boatrec}
\end{figure}

All the methods carry out more iterations than in the previous example, due to the smaller amount 
of noise in the present example. Visual inspection of the images in Figure \ref{fig:boatrec} 
shows that the unpreconditioned Arnoldi-TSVD solution to bear some motion artifacts, as the restored 
image displays some shifts in the diagonal directions, i.e., in the direction of the motion blur. These 
spurious effects are not so pronounced in the TSVD($M_1$) restoration, as the preconditioner (\ref{prec1}) 
makes the problem more symmetric. The reconstruction produced by TSVD($M_3$) is noticeably worse; indeed, 
the preconditioner (\ref{prec3}) merely approximates a regularized inverse of $A$, and this is not desirable 
when applying the Arnoldi algorithm to a very unsymmetric blur. The results obtained when applying 
the Arnoldi--Tikhonov methods are very similar to the ones obtained with the Arnoldi-TSVD methods. We therefore
do not show the former. 

\section{Conclusions}\label{sec8}
This paper presents an analysis of the GMRES method and the Arnoldi algorithm with applications to the 
regularization of large-scale linear ill-posed problems. Theoretical properties that involve the distance 
of the original coefficient matrix to classes of generalized Hermitian matrices are derived. Novel 
preconditioners based on matrices stemming from the standard Arnoldi decomposition are introduced, and the 
resulting right-preconditioned linear systems are solved with methods based on the preconditioned Arnoldi 
algorithm, or the new preconditioned Arnoldi--Tikhonov and Arnoldi-TSVD methods. Numerical results on a 
variety of test problems clearly show the benefits of applying the new preconditioning techniques.

\end{document}